\documentclass[11pt]{amsart}

\usepackage{amssymb
,amsthm
,amsmath
,amscd
,mathtools}

\usepackage{mathrsfs}
\usepackage{tikz-cd}
\usepackage{enumitem}
\usepackage[all]{xypic}
\usepackage{amssymb, amsmath, amscd, amsthm}
\usepackage{epstopdf}
\usepackage{mathdots}
\usepackage[bookmarks=true,bookmarksnumbered=true]{hyperref}
\usepackage[hmargin=2cm,vmargin=3cm]{geometry}
\usepackage[all]{xy}
\allowdisplaybreaks

\numberwithin{equation}{section}

\theoremstyle{plain}
\newtheorem{thm}{Theorem}[section]
\newtheorem{lem}[thm]{Lemma}
\newtheorem{prop}[thm]{Proposition}

\newtheorem*{main}{Main Theorem}

\theoremstyle{definition}
\newtheorem{defn}[thm]{Definition}
\newtheorem{rem}[thm]{Remark}

\def\Aut{\operatorname{Aut}}

\def\Hom{\operatorname{Hom}}

\def\Ind{\operatorname{Ind}}

\def\Re{\operatorname{Re}}

\def\Tr{\operatorname{Tr}}
\def\beq{\begin{equation}}
\def\eeq{\end{equation}}
\def\beqn{\begin{equation*}}
\def\eeqn{\end{equation*}}
\def\beqna{\begin{eqnarray}}
\def\eeqna{\end{eqnarray}}
\def\beqnan{\begin{eqnarray*}}
\def\eeqnan{\end{eqnarray*}}

\def\vi{\varphi}
\def\mc{\mathcal}

\def\mf{\mathfrak}

\def\bs{\backslash}

\def\As{\mathrm{As}}

\def\A{\mathbb{A}}
\def\As{\mathscr{A}}

\def\CC{\mathbb{C}}

\def\RR{\mathbb{R}}

\def\col{\coloneqq}
\def\co{\colon}
\def\ss{\subseteq}

\def\S{\mathcal{S}}

\def\aa{\mathfrak{a}}

\def\la{\langle}
\def\ra{\rangle}
\def\hra{\hookrightarrow}
\def\bs{\backslash}

\def\1{\Eins}

\makeatletter
\def\varddots{\mathinner{\mkern1mu
    \raise\p@\hbox{.}\mkern2mu\raise4\p@\hbox{.}\mkern2mu
    \raise7\p@\vbox{\kern7\p@\hbox{.}}\mkern1mu}}
\makeatother

\newcommand{\BIGOP}[1]{\mathop{\mathchoice%
{\raise-0.22em\hbox{\huge $#1$}}
{\raise-0.05em\hbox{\Large $#1$}}{\hbox{\large $#1$}}{#1}}}

\newcommand{\BIGboxplus}{\mathop{\mathchoice%
{\raise-0.35em\hbox{\huge $\boxplus$}}%
{\raise-0.15em\hbox{\Large $\boxplus$}}{\hbox{\large $\boxplus$}}{\boxplus}}}

\title{Regularized Periods and the global Gan--Gross--Prasad conjecture : The case of $U(n+2r) \times U(n)$}
\author{Jaeho Haan}
\address{Department of Mathematics
Yonsei University
50 Yonsei-Ro, Seodaemun-Gu, Seoul 03722, Republic of Korea}
\email{jaehohaan@gmail.com}
\date{\today}

\begin{document}

\begin{abstract}In this paper, we introduce regularized trilinear periods on certain non-reductive groups. It has two direct applications. Firstly, it enables us to define the regularized Bessel periods and the regularized Fourier-Jacobi periods for all classical and metaplectic groups. Secondly, by using the properties of the regularized Fourier-Jacobi periods, we can prove one direction of the global Gan--Gross--Prasad conjecture on skew-hermitian unitary groups for tempered cases.
\end{abstract}

\maketitle

\setcounter{tocdepth}{1}
\tableofcontents

\section{\textbf{Introduction}}
\subsection{General introduction}
The study of special $L$-values is a pivotal domain in modern number theory because these values encode a considerable amount of arithmetic information. Among all other special values, central $L$-values are of particular interest due to the implications of the Birch and Swinnerton-Dyer conjecture and its generalizations.

In 1992, Gross and Prasad \cite{GP} proposed a fascinating conjecture regarding the relationship between automorphic period integrals and central $L$-values of certain tensor product $L$-functions on special orthogonal groups. Later, in 2012, they, together with Gan \cite{Gan2}, extended this conjecture to all classical groups and metaplectic groups. This extended conjecture is now referred to as the global \emph{Gan--Gross--Prasad} (GGP) conjecture.

The global GGP conjecture involves two types of periods: the Bessel periods and the Fourier--Jacobi periods. The Bessel periods pertain to automorphic forms on orthogonal groups or hermitian unitary groups, while the Fourier--Jacobi periods relate to automorphic forms on metaplectic-symplectic groups or skew hermitian unitary groups. We will now briefly recapitulate the definitions of both periods, focusing at least on (skew) hermitian unitary groups.

Let $E$ and $F$ be quadratic extensions of number fields with adele rings $\A_E$ and $\A_F$, respectively. Consider the nontrivial quadratic character $\omega$ of $F^{\times} \backslash \A_F^{\times}$ associated with $E/F$ through global class field theory. Choose a unitary character $\mu$ of $E^{\times} \bs \A_E^{\times}$ such that its restriction to $\A_F^{\times}$ is $\omega$. Occasionally, we regard $\mu$ as a character of $GL_n(\A_E)$, in which case it refers to $\mu \circ \det$. Fix a nontrivial continuous character $\psi$ of $F \bs \A_F$. For $\epsilon \in \{\pm\}$, let $W_m \subset W_n$ be non-degenerate $m$ and $n$-dimensional $\epsilon$-Hermitian spaces over $E$ satisfying:
\begin{itemize}
\item $\epsilon \cdot (-1)^{\text{dim}W_m^{\perp}}=-1$;
\item $W_m^{\perp}$ in $W_n$ is a split space.
\end{itemize}

Define $G_n$ and $G_m$ as the isometry groups of $W_n$ and $W_m$ respectively. Consider $G_m$ as a subgroup of $G_n$ that acts trivially on the orthogonal complement of $W_m$ within $W_n$. Let $X$ be a maximal isotropic subspace of $W_m^{\perp}$, and let $P$ be the parabolic subgroup of $G_n$ defined as the stabilizer of a complete flag of subspaces within $X$. Write $n-m=2r+1$ if $n-m$ is odd, and $n-m=2r$ if $n-m$ is even. Denote by $N_{n,r}$ the unipotent radical of $P$, on which $G_m$ acts through conjugation. Let $H=N_{n,r} \rtimes G_m$. 
In the case where $n-m=2r+1$, there exists an essentially unique generic character $\chi':N_{n,r}(\A_F) \bs N_{n,r}(\A_F) \to \CC^{\times}$ stabilized by $G_m(\A_F)$. This character can be extended to a character of $H(F) \bs H(\A_F)$.

When $n-m=2r$, there exists a global generic Weil representation denoted as $\nu_{\psi^{-1},\mu^{-1},W_m}$ of $H(\A_F)$. This representation is realized on the Schwartz space $\mc{S} = \mc{S}(Y(\A_F))$, where $Y$ is a Lagrangian $F$-space of $\mathrm{Res}_{E/F}W_m$.

For each $f \in  \mc{S}$, the associated  theta series is defined by 
\[
\Theta_{\psi^{-1},\mu^{-1},W_m}(h ,f)=\sum_{x\in Y(F)} \left( \nu_{\psi^{-1},\mu^{-1},W_{m}}(h)f \right) (x), 
\quad
h \in H(\A_F).
\]
The space of theta series provides an automorphic realization of $\nu_{\psi^{-1},\mu^{-1},W_m}$ that is invariant under the action of $H(F)$. Moving forward, we will denote $G(F) \bs G(\mathbb{A}_F)$ as $[G]$ for any algebraic group $G$ defined over $F$.

Consider two irreducible cuspidal automorphic representations, $\pi_1$ and $\pi_2$, of $G_n(\A_F)$ and $G_m(\A_F)$ respectively. We view $H$ as a subgroup of $G_n$ through the mapping $(n,g) \mapsto ng$. Depending on whether $n-m$ is odd or even, we define the \emph{Bessel periods} and the \emph{Fourier--Jacobi periods} for $\pi_1,\pi_2$ as the following integrals:
\begin{itemize}

\item If $n-m$ is odd, for $\vi_1 \in \pi_1, \vi_2 \in \pi_2$, the Bessel period for $\vi_1,\vi_2$ is defined by
\[\mathcal{B}_{\chi'}(\vi_1,\vi_2) \col \int_{ [ N_{n,r} \rtimes G_m]}\vi_1(ng)\vi_2(g)\chi'^{-1}((n,g))dndg.
\]

\item If $n-m$ is even, for $\vi_1 \in \pi_1, \vi_2 \in \pi_2, f\in \nu_{\psi^{-1},\mu^{-1},W_m}$, the Fourier--Jacobi period $\vi_1,\vi_2$ is defined by 
\[
\mathcal{FJ}_{\psi,\mu}(\vi_1,\vi_2,f) \col \int_{ [ N_{n,r} \rtimes G_m]}\vi_1(ng)\vi_2(g)\Theta_{\psi^{-1},\mu^{-1},W_m}\left((n,g),f\right)dndg.
\]
\end{itemize}
\par

The GGP conjecture \cite[Conjecture 24.1]{Gan2} predicts that for an irreducible tempered cuspidal representation $\pi_1 \boxtimes \pi_2$ of $G_n \times G_m$, the following two conditions are equivalent in each case.
In what follows, $BC$ stands for the standard base change and the $L$-functions are the Rankin-Selberg $L$-functions.

\begin{itemize}
\item \textbf{Bessel case}
\begin{enumerate}
\item $\mathcal{B}_{\chi'}$ is non-vanishing,

\item $\text{Hom}_{H(F_v)}(\pi_{1,v} \boxtimes \pi_{2,v},\CC)\ne0$ for all places $v$ and $L(\frac{1}{2},BC(\pi_1) \times BC(\pi_2))\ne0.$
\end{enumerate}
\par
\end{itemize}
\par

\begin{itemize}
\item \textbf{Fourier-Jacobi case}
\begin{enumerate}
\item $\mathcal{FJ}_{\psi,\mu}$ is non-vanishing,

\item $\text{Hom}_{H(F_v)}(\pi_{1,v} \boxtimes \pi_{2,v} \boxtimes \nu_{\psi^{-1}_{v},\mu^{-1}_{v},W_m},\CC)\ne0$ for all places $v$ and $L(\frac{1}{2},BC(\pi_1) \times BC(\pi_2) \otimes \mu^{-1})\ne0$.
\end{enumerate}
\par
\end{itemize}

\par
In their influential paper \cite{Ich}, Ichino--Ikeda formulated a refined GGP conjecture expressing the Bessel periods explicitly
in terms of special values for orthogonal groups of co-rank 1, i.e., the $SO(n+1) \times SO(n)$ case. Thereafter, Liu \cite{L}  extended their conjecture to higher co-rank cases for both orthogonal groups and hermitian unitary groups and Xue \cite{Xue} formulated the corresponding refined conjecture for metaplectic-symplectic groups. The refined GGP conjecture implies the (i) $\to$ (ii) direction of the original GGP conjecture.

After the formulation of both the original and refined GGP conjectures, there has been huge progress towards the low co-rank (i.e., co-rank $\le 1$) GGP conjectures until recently. Wei Zhang was the first who made a breakthrough toward the low co-rank GGP conjectures. Using the relative trace formula of Jacquet and Rallis \cite{JR}, he proved the original and refined GGP conjectures for the co-rank $1$ hermitian unitary groups, i.e., $U(n+1) \times U(n)$ under two local assumptions (\cite{W1,W2}). Shortly thereafter, following a similar idea of W. Zhang, H. Xue \cite{Xue1, Xue2} proved the original and refined GGP conjectures for the co-rank $0$ skew-hermitian unitary groups, i.e., $U(n) \times U(n)$, under the similar local assumptions. Those two local assumptions arise when one applies a simple version of the Jacquet--Rallis relative trace formula. The first condition has now been completely dropped in a recent stunning work of R. Beuzart-Plessis, Y. Liu, W. Zhang and X. Zhu \cite{BLZZ}, which proves the original and refined GGP conjectures for $U(n+1) \times U(n)$ in the stable case. However, to remove the second condition, which forces $BC(\pi_1) \boxtimes BC(\pi_2)$ to be cuspidal, was regarded very hard for a long time because it would require the fine spectral expansion of the  Jacquet--Rallis relative trace formula. Subsequently, in a brilliant contribution \cite{BPCZ}, R. Beuzart-Plessis, Pierre-Henri  Chaudouard and M. Zydor settled the original GGP conjecture and refined GGP conjectures for $U(n+1) \times U(n)$ across all endoscopic cases circumventing the fine spectral expansion of the  Jacquet--Rallis relative trace formula. Including the powerful tools developed in \cite{BLZZ}, their proof combines many results concerning  the  Jacquet--Rallis relative trace formula in a very clever way. As a consequence, the original and refined GGP conjectures for $U(n+1) \times U(n)$ have been completely settled. However, it is not clear whether such an approach using the Jacquet--Rallis's relative trace formula is available for other classical groups other than unitary groups.
\par

On the other hand, there has been a different approach towards the low co-rank GGP conjectures. Inspired by the construction of the regularized period integral of Jacquet--Lapid--Rogawski \cite{JLR, LR}, Ichino--Yamana \cite{IY} and Yamana \cite{Y} invented regularized period integrals for co-rank 1 hermitian unitary groups, i.e. $U(n+1) \times U(n)$ and for co-rank 0 metaplectic--symplectic groups, i.e. $Mp(2n) \times Sp(2n)$, respectively. By computing the regularized period  involving certain residual Eisenstein series, they proved the (i) $\to$ (ii) direction of the original GGP conjecture in such low co-rank cases without any local assumptions. It is also worthwhile to mention that the Ichino--Yamana's mixed truncation for $GL(n+1) \times GL(n)$ (resp. $U(n+1) \times U(n)$) invented in \cite{IY1} (resp. \cite{IY}) was also used in \cite{BPCZ}.

\par For the higher co-rank (i.e., co-rank $> 1$) GGP conjecture, there has been a significant advancement in the Bessel case recently. In their seminal paper, under some hypothesis, Jiang and Zhang \cite[Theorem 5.7]{JZ} proved the (i) $\to$ (ii) direction of the original full GGP conjecture in all cases and the (ii) $\to$ (i) direction in some special cases for both orthogonal groups and hermitian unitary groups. Their method crucially uses the \emph{reciprocal non-vanishing of Bessel periods} \cite[Theorem 5.3]{JZ} to establish the global non-vanishing of the twisted automorphic descent which extends the automorphic descent developed by Ginzburg--Rallis--Soudry \cite{gjr}. Since the proof of \cite[Theorem 5.3]{JZ} heavily depends on the unproven general calculation \cite[Theorem 4.8]{JZ} of the unramified local zeta integrals defined using the local Bessel functionals, their proof of the (i) $\to$ (ii) direction of the GGP conjecture is also conditional on such hypothetical computation. Though some special cases of \cite[Theorem 4.8]{JZ} were carried out in several papers, for example, \cite{GPSR}, \cite{JZ3}, \cite{S-ICM}, \cite{S-I}, \cite{S-II}, the general computation, which is required to prove \cite[Theorem 5.7]{JZ}, does not seem to be easy. (As written in \cite{JZ}, since 2018, Jiang, Zhang and Soudry have been preparing a paper \cite{JSZ}, which contains the computation in full generality for quasi-split classical groups. Their preprint, however, has not appeared yet and so \cite[Theorem 5.7]{JZ} is still conditional.) We also mention  that Furusawa--Morimoto \cite{MF1},\cite{MF2} recently proved the original and refined GGP conjecture for certain special irreducible tempered representations of $SO(2n+1) \times SO(2)$.

In the Fourier--Jacobi case, on the other hand,  the most well-known result for the higher co-rank GGP conjecture is the one by Ginzburg--Jiang--Rallis on metaplectic-symplectic groups \cite{GJR} which has been recently supplemented with \cite{GJBR}. In that paper, by combining the Arthur truncation method and the Rankin-Selberg method, they proved the (i) $\to$ (ii) direction of the original GGP conjecture for arbitrary even co-rank metapletic-symplectic groups, i.e., $Mp(2n+2r) \times Sp(2n)$ under a stability assumption. The assumption forces the Langlands functorial transfer of $\pi_1 \boxtimes \pi_2$ to their corresponding general linear groups to be cuspidal. Since this assumption is critical in their argument, it appears hard to remove it while adopting their approach. (See the remarks after \cite[Theorem 5.7]{JZ}.) On the other hand, it is plausible to apply the similar approach used in \cite{JZ} to the Fourier--Jacobi case. Indeed, as it is written in \cite{JZ1}, Jiang and Zhang have been preparing a paper \cite{JZ2} concerning the Fourier--Jacobi case for metapeltic-symplectic groups since 2015.  However no preprint seems available yet, and in any case, the results would be  conditional on the non-existent results in \cite{JSZ} if they use the same method as in \cite{JZ}.

In contrast to the metaplectic-symplectic case, very little is known for non-equal rank skew-hermitian unitary groups. To the best of the author's knowledge, Gelbart--Rogawski's paper \cite{GR} which studied the Fourier-Jacobi period of $U(3) \times U(1)$ is the only work that deals with non-equal rank skew-hermitian groups. \par

The main goal of this paper is to prove the (i) $\to$ (ii) direction of the original GGP conjecture for skew-hermitian unitary groups, i.e., $U(n+2r) \times U(n)$ without any hypotheses. To do this, we first define regularized trilinear periods which involve an integration over a unipotent subgroup. 
This enables us to construct the regularized Bessel and Fourier-Jacobi periods for all classical and metaplectic groups. In particular, when $r=0$, these are exactly the same regularized periods introduced in \cite{IY} and \cite{Y}, respectively. By exploiting several properties of regularized Fourier--Jacobi periods, we prove the following:

\begin{main} \label{main}
Let $n-m=2r$ for some non-negative integer $r$ and $\pi_1, \pi_2$ be irreducible cuspidal automorphic representations 
of $G_n(\A_F)$ and $G_m(\A_F)$ with generic $A$-parameters, respectively. 
If there are $\vi_1\in \pi_1, \vi_2 \in \pi_2$ and $f\in  \nu_{W_{m}}$ 
such that \[\mathcal{FJ}_{\psi,\mu}(\vi_1,\vi_2,f)\ne 0,\] 
then $L(\frac{1}{2},BC(\pi_1) \times BC(\pi_2)\otimes \mu^{-1})\ne 0$.
\end{main}

We remark on the proof of the main theorem.

\begin{rem}\label{ec} Based on the works of Arthur (\cite{Ae}), Mok (\cite{Mok}) and Kaletha--Minguez--Shin--White (\cite{KMSW}) on the endoscopic classification for unitary groups, the irreducible cuspidal representations $\pi_1,\pi_2$ with generic $A$-parameters have the following properties.

\begin{enumerate}
\item The weak base changes $BC(\pi_1)$ and $BC(\pi_2)$ exist
\item $BC(\pi_2)$ is decomposed as the isobaric sum $\sigma_1 \boxplus \cdots \boxplus \sigma_t$, where $\sigma_i$'s are irreducible unitary cuspidal automorphic representations of general linear groups such that the (twisted) Asai $L$-function $L(s,\sigma_i,As^{(-1)^{m-1}})$ has a pole at $s=1$ for all $1\le i\le t$.
\item For each place $v$ of $F$ and $1\le i \le t$, the local normalized intertwining operator $N_v(z)$ for $\sigma_{i,v} \times \pi_{1,v}$ is holomorphic and nonzero for all $z$ with $\Re(z) \ge \frac{1}{2}$.
\item Let $\pi$ be either $\pi_1$ or $\pi_2$. Then for some place $v$ of $F$, $\pi_v$ is the unique unramified subquotient of some principal series representations 
$\Ind_{{B_k'}}^{G_k}
(\chi_{1} \boxtimes \cdots  \boxtimes \chi_{[\frac{k}{2}]} )$ such that  $\chi_i=\chi_i'\cdot |\cdot|^{s_i}$, where $\chi_i'$'s are unramified unitary characters and $0\le s_i<\frac{1}{2}$.
\end{enumerate}

In (iii), the local normalized intertwining operator $N_v(z)$ is constructed by multiplying the local intertwining operator $M_v(z)$ for the representation $\sigma_{i,v} \times \pi_{1,v}$ of $GL_a(F_v) \times G_n(F_v)$ (see \cite[Sect 4.1]{Sha0}) with the following 
\[\frac{L_v(s+1,\sigma_{i,v} \times \pi_{1,v})\cdot L_v(2s+1,\sigma_{i,v}, As^{(-1)^{n-1}})\cdot \epsilon_v(s,\sigma_{i,v} \times \pi_{1,v},\psi_v)\cdot \epsilon_v(2s,\sigma_{i,v},As^{(-1)^{n-1}},\psi_v) }{L_v(s,\sigma_{i,v} \times \pi_{1,v})\cdot L_v(2s,\sigma_{i,v},As^{(-1)^{n-1}})},\]
where the local $L$-factors and $\epsilon$-factors are defined through the localization of the global $A$-parameter of $\pi_{1}$ as described in \cite{KMSW} and \cite{Mok}. Then the property (iii) mentioned above follows from \cite[Theorem~B.2]{JZ}.

It is known that globally generic cuspidal representations of unitary groups also satisfy the above properties (see \cite{KM1}, \cite{KM2}.) Therefore, our main theorem would also hold for such globally generic cuspidal representations (possibly non-tempered because of the absence of the generalized Ramanujan conjecture.)
In the course of proving our main theorem, we use the above properties in a crucial manner. 

\begin{rem}The proof of our main theorem, to some extent, is in the same vein with \cite{GJR} in that it uses the relation of non-vanishing of central $L$-values with the existence of the residual Eisenstein series. However, comparing with the proof of \cite{GJR}, there are two main innovations in our approach. First, to compute the Fourier--Jacobi period integrals involving residual Eisenstein series, Ginzburg, Jiang and Rallis \cite{GJR} used the Arthur truncation method to render meaning to a divergent integral involving Eisenstein series. In their approach using the Arthur truncation, the stability assumption on $\pi_1, \pi_2$ becomes indispensable. However, if we use the mixed truncation method to compute the regularized Fourier--Jacobi period involving residual Eisenstein series, we don't need such a stability assumption on $\pi_1,\pi_2$. 

Secondly, though we took a different truncation from G--J--R's, in order to relate the Fourier--Jacobi periods of residual Eisenstein series with the original Fourier--Jacobi periods, we reach the almost same place in G--J--R's, that is, \cite[Proposition 5.3]{GJR}. (Compare with Lemma \ref{l3}). To prove the direction (i) $\to$ (ii) of the GGP conjecture, G--J--R employed the isobaric sum of $FL(\pi_1)$ (the Langlands functorial transfer of $\pi_1$). However, this choice of the isobaric sum makes the proof of \cite[Proposition 5.3]{GJR} very difficult. Around 2013, some experts found a serious gap in the original proof of \cite[Proposition 5.3]{GJR}. Thereafter, in 2019, G--J--R, together with B. Liu, uploaded their corrected proof of \cite[Proposition 5.3]{GJR} to arXiv (\cite{GJBR}.)
But their corrected proof is rather involved and roundabout. In this paper, we overcome such difficulty by using the isobaric sum of $BC(\pi_2)$ instead of $BC(\pi_1)$. With this choice of the isobaric sum, we can avoid the difficulty arising in the proof of \cite[Proposition 5.3]{GJR} and streamline the discussion substantially. (See Remark \ref{ir} for a more detailed explanation.)
\end{rem}

\begin{rem}Theorem \ref{rthm} bears a resemblance to \cite[Theorem 5.3]{JZ}, which pertains to the reciprocal non-vanishing of Bessel periods. However, it is worth noting that while the proof of \cite[Theorem 5.3]{JZ} indispensably relies on the assumption regarding the unramified calculation of specific local zeta integrals, we establish Theorem \ref{rthm} without recourse to such an assumption. Furthermore, it's noteworthy that Theorem \ref{rthm} harbors potential utility in establishing the direction (ii) to (i) within the GGP conjecture, as demonstrated in \cite[Cor. 6.11]{JZ}.
\end{rem}
\end{rem}



This paper is organized as follows. In Section 2, we establish some general notations and in Section 3, we give a definition of automorphic forms on certain non-reductive groups modeled after Jacobi forms. In Section 4, we explain the GGP triplets to which our mixed truncation is designed and in Section 5, we define the mixed truncation operator which maps an automorphic form on a certain non-reductive group to a rapidly decreasing automorphic form on a reductive group. Using this mixed truncation, we can define the regularized trilinear period integrals and the regularized Bessel and Fourier--Jacobi periods in Section 6. In Section 7, after introducing the Jacquet module corresponding to the Fourier--Jacobi character, we prove a lemma which is required to prove a vanishing property of regularized periods involving certain residual Eisenstein series. In Section 8, we explain the definition of (residual) Eisenstein series and review \cite[Proposition 5.3]{IY} which manifests the correlation of analytic properties of Eisenstein series with the behavior of the associated Rankin-Selberg $L$-functions. In Section 9, we prove various lemmas required to establish Theorem \ref{rthm}, which claims the reciprocal non-vanishing of the Fourier--Jacobi periods. Guided by Theorem \ref{rthm}, the culmination of our endeavor is the proof of our main theorem, presented in Section 10.

We stress that separate approaches need to be taken to prove Lemma \ref{Jac} and for Lemma \ref{l1} according to whether $r=0$ or $r>0$ because the relevant twisted Jacquet modules are totally different. Threrefore, we cannot reduce the proof of our main theorem to a single one for all ranks. This is in contrast to the case of the local GGP conjecture, where one can reduce the local GGP for $r>0$ to that of $r=0$ (cf. \cite[Theorem 19.1]{Gan2}).

Upon the completion of this paper, we became aware that Gan, Gross and Prasad extended their original conjecture to non-tempered cases in their recent paper. In \cite[Conjecture 9.1]{Gan1}, they formulated a comprehensive global (resp. local) GGP conjecture for automorphic representations within global (resp. local) $A$-packets. However, automorphic representations in non-tempered $A$-packet need not necessarily be cuspidal. Therefore, for the formulation of the all-encompassing GGP conjecture including non-tempered representations, one needs to regularize both $\mathcal{B}_{\chi'}$ and $\mathcal{FJ}_{\psi,\mu}$. For this reason, they formulated \cite[Conjecture 9.1]{Gan1} assuming the existence of certain regularized period integrals in each case. With the introduction of our regularized Bessel and Fourier--Jacobi period integrals as outlined in Definition \ref{rdef}, we are now equipped to present their overarching conjecture in a more explicit manner (as outlined in Remark \ref{well}). We anticipate that the properties of the established regularized periods expounded in this paper will play a pivotal role in substantiating the general GGP conjecture.

\subsection{Some related problems} The methods used in this paper may be applicable to other problems. We discuss some of them here.

\begin{enumerate}

\item 
In \cite{H3}, following a method similar to that of the current paper, we were able to establish the (i) $\to$ (ii) direction of the original GGP conjecture for higher co-rank metaplectic-symplectic groups without relying on the stability assumption found in \cite{GJR}. Given that we've introduced the concept of regularized Bessel periods in this paper, it's plausible that a similar approach could be employed to address the Bessel cases as well. Notably, our approach avoids the necessity of calculating unramified local zeta integrals, thereby potentially enhancing the outcome achieved by D. Jiang and L. Zhang in \cite[Theorem 5.7]{JZ}. Further details on this topic can be found in our subsequent work \cite{H1}.

\item In \cite{BPCZ}, R. Beuzart-Plessis, P-H. Chaudouard and M. Zydor calculated a certain term appearing in the coarse spectral expansion of the Jacquet--Rallis trace formula for linear groups in terms of the Rankin-Selberg periods of Eisenstein series and the Flicker--Rallis intertwining periods. In their computation, the Ichino--Yamana's mixed trunction for ($GL(n+1) \times GL(n)$) in \cite{IY} is crucially used to define certain truncated integrals involving a relevant kernel function and Eisenstein series. Since this computation is one of the key innovation of their paper, to extend their result to attack the higher co-rank GGP conjectue for unitary groups $(U(n+2r+1) \times U(n))$, the mixed truncation for ($GL(n+2r+1) \times GL(n)$) involving some unipotent subgroup seems to be necessary. Though such a mixed truncation for higher co-rank linear groups can also be defined similarly, its very definition and proof are slightly different from those of classical groups. For this reason, we deal with the construction of the mixed truncation for linear groups separately in our future work \cite{H4}.

\item We anticipate that our approach could also be applied to certain significant non-tempered instances of the GGP conjecture. In an upcoming work, \cite{H2}, we will delve into the investigation of regularized Bessel periods incorporating specific Eisenstein series. Our aim will be to establish the direction (i) $\to$ (ii) of the non-tempered GGP conjecture within this context.

\end{enumerate}
\subsection*{Acknowledgements} The author wishes to extend heartfelt gratitude to Professor Hiraku Atobe for suggesting this intriguing problem and for engaging in numerous beneficial discussions. Without his guidance, this work would not have come to fruition. The concept of automorphic forms on non-reductive groups in Section 3 originated from him. Furthermore, he also found critical errors in the proof of Lemma \ref{Jac} and Lemma \ref{l1} in the initial draft. It is with immense appreciation that we acknowledge his significant contributions to this paper.

A deep debt of gratitude is owed to Professor Shunsuke Yamana. The profound influence of his paper \cite{Y} on this work is evident. Aside from being the author of the inspiring paper that served as the foundation for this research, he generously answered numerous questions posed by the author regarding his work. We also wish to express our sincere thanks to Professor Erez Lapid for clarifying certain aspects of his proof concerning \cite[Lemma 8.2.1(1)]{LR}. To Professors Dipendra Prasad and Sug Woo Shin, we extend our thanks for their insightful comments on our draft.

Our appreciation extends to Professors Yeansu Kim and Sanghoon Kwon, with whom the author conducted a five-month reading seminar on the "Arthur trace formula." This seminar facilitated familiarity with the Arthur truncation and generated ideas on how to prove various lemmas in Section 5. We also acknowledge the valuable discussions with Dr. Wan Lee and Dr. Chulhee Lee.

We wish to extend our thanks to Professors Dongho Byeon, Youn-Seo Choi, Youngju Choie, Wee Teck Gan, Haseo Ki and Sug Woo Shin. Their patience, guidance and unwavering supports have been instrumental in the completion of this project.

Last but not least, we sincerely express our gratitude to Prof. Wei Zhang for his interest in the results of this paper, and to the referee for their thorough review and invaluable comments, which undoubtedly enhanced the clarity of the paper's presentation.

Most part of this paper was written when the author was a postdoc at KIAS. We extend our gratitude to KIAS and Yonsei University for providing an excellent working environment and generous support. The author was supported by KIAS Individual Grant(SP037002) at Korea Institute for Advanced Study, the POSCO Science Fellowship of POSCO TJ Park Foundation and the National Research Foundation of Korea (NRF) grant funded by the Korea government (MSIT) (No. 2020R1F1A1A01048645).

\section{\textbf{General notation}}
In this section, we establish and adhere to a set of general notations and conventions that will be used throughout the paper.

Let $F$ be a number field with adele ring $\A_F$. For a place $v$ of $F$, denote by $F_v$ the localization of $F$ at $v$. 
Let $G$ be a connected reductive algebraic group over $F$ and fix a minimal $F$-parabolic subgroup $P_0$ of $G$ with a Levi decomposition $P_0=M_0U_0$. We call any $F$-parabolic subgroup $P=MU$ of $G$ which contains $P_0$ \emph{standard}. 
Fix a maximal compact subgroup $K=\prod_v K_v$ of $G(\A_F)$ which satisfies $G(\A_F)=P_0(\A_F)K$ and for every standard parabolic subgroup $P=MU$ of $G$, 
\[ P(\A_F)\cap K=(M(\A_F) \cap K)(U(\A_F)\cap K)\]
and $M(\A_F) \cap K$ is still maximal compact in $M(\A_F)$ (see \cite[I.1.4]{Mo}). 

We give the Haar measure on $U(\A_F)$ so that the volume of $[U]$ is 1 and give the Haar measure on $K$ so that the total volume of $K$  is $1$. 
We also choose Haar measures on $M(\A_F)$ for all Levi subgroups $M$ of $G$ compatibly 
with respect to the Iwasawa decomposition.

Let $\textbf{G}$ be a topological group that is a finite central covering of $G(\A_F)$, i.e., there exists a surjective morphism $pr_G:\textbf{G} \to G(\A_F)$, whose kernel is a finite subgroup of the center of $\textbf{G}$ and $pr_G$ a topological covering. 
It is known that $U_0(\A_F)$ lifts canonically into $\textbf{G}$  (see \cite[Appendix I]{Mo}) and a result of Weil \cite{We} says that $G(F)$ lifts uniquely into $\textbf{G}$. 
For the unipotent subgroup $U$ of any standard parabolic subgroup of $G$ (or a subgroup $S$ of $G(F)$), we use the same notation $U(\A_F)$ (or $S$) to denote their image of such lifting. For a subgroup $J$ of $G(\A_F)$, we use the boldface $\textbf{J}$ to denote $pr_G^{-1}(J)$. Note that parabolic subgroups of $\textbf{G}$ are $\textbf{P}=\textbf{M}U(\A_F)$. Sometimes, we view a function defined on $J$ as a function on $\textbf{J}$ by composing them with $pr_G$.  

 Let $\As_P(\textbf{G})$ be the space of automorphic forms on $U(\A_F) P(F) \backslash \textbf{G}(\A_F)$, 
i.e., smooth, $\textbf{K}$-finite and $\mathfrak{z}$-finite functions on $U(\A_F) P(F) \backslash \textbf{G}(\A_F)$ 
of moderate growth, 
where $\mathfrak{z}$ is the center of the universal enveloping algebra of the complexified Lie algebra 
of $G(F\otimes_{\mathbb{Q}}\mathbb{R})$.
When $P=G$, we simply write $\As(\textbf{G})$ for $\As_G(\textbf{G})$.
For a cuspidal automorphic representation $\rho$ of $\textbf{M}$, 
write $\As_P^{\rho}(\textbf{G})$ for the subspace of functions $\phi \in \As_P(\textbf{G})$ 
such that for all $k \in \textbf{K}$, 
the function $m \to \delta_P(m) ^{-\frac{1}{2}} \cdot  \phi(mk)$ belongs to the space of $\rho$. 
(Here, $\delta_P$ is the modulus function of $\textbf{P}$ (see \cite[I.2.17]{Mo}).)
\par

\section{\textbf{Automorphic forms on $N \rtimes \textbf{G}'$}}
In this section, we extend the definition of automorphic forms from reductive groups to certain non-reductive groups. 
Let $G' $ be a connected reductive group defined over $F$. We fix a minimal $F$-parabolic subgroup $P_0'$ of $G'$ and a maximal compact subgroup $K'=\prod_v K_v'$ of $G'(\A_F)$. We will adopt similar notations for subgroups of $G'$ by appending a prime $^{\prime}$ to those of $G$. Let $N$ be a unipotent $F$-group which admits a $G'$-action 
in the category of algebraic groups over $F$. 
We denote this action by $\sigma \colon G' \to \Aut(N)$. Using $\sigma$, we form the semi-direct product $N(\A_F) \rtimes \textbf{G}'$ and introduce automorphic forms on $N(\A_F) \rtimes \textbf{G}'$, generalizing the properties of Jacobi forms as presented in \cite[Proposition 4.1.2]{Sch}).

Let $| \cdot |$ denote a norm function on $G'(\A_F)$, as described in \cite[Section I.2.2]{Mo}. We will adopt the same notation $| \cdot |$ to represent the norm function on $\textbf{G}'$ obtained through the projection $pr_{G'}$. 
\par

\begin{defn} Let $P'=M'U'$ be a parabilic subgroup of $G'$. For a function $\varphi \colon N(\A_F) \rtimes \textbf{G}' \to \mathbb{C}$, we say that $\varphi$ is 
an \textit{$P'$-automorphic form} 
if 

\begin{itemize}

\item 
$\varphi \left( (\delta,\gamma) \cdot (n,ug) \right) = \varphi(n,g)$ for $u \in U'(\A_F)$ and $(\delta,\gamma)\in N(F) \rtimes P'(F)$;
\item 
$\varphi$ is smooth;
\item 
$\varphi$ is of moderate growth;
\item 
$\varphi$ is right $\textbf{K}'$-finite; 

\item 
$\varphi$ is $\mathfrak{z}'$-finite.
\end{itemize}
Here, $\varphi$ is said to be \textit{of moderate growth} 
if there is a positive integer $m$ and a constant $C$ such that 
\[ |\varphi(n,g)|< C \|g\| ^{m}
\]
for all $(n,g)  \in N(\A_F) \rtimes \textbf{G}'$.

\end{defn}
We denote by $\As_{P'}(N \rtimes \textbf{G}')$ 
the space of automorphic forms on $N (\A_F) \rtimes \textbf{G}'$ satisfying the above properties. When $P'=G'$, we simply write $\As(N \rtimes \textbf{G}')$ for $\As_{G'}(N \rtimes \textbf{G}')$. Note that if $N$ is the trivial group $\bf{1}$, $\As_{P'}(\bf{1}$ $ \rtimes \ \textbf{G}')$ is just $\As_{P'}(\textbf{G}')$. 

For $\varphi \in \As(N \rtimes \textbf{G}')$, define $\varphi_{P'} \co N (\A_F) \rtimes \textbf{G}'\to \CC$ by 
\[ 
\varphi_{P'}(n,g) \col \int_{[U']} \varphi (n,ug)du \ \  \text{for} \  (n,g)\in N (\A_F) \rtimes \textbf{G}'.
\] It is obvious that $\varphi_{P'} \in \As_{P'}(N \rtimes \textbf{G}')$. If we regard $\varphi \in \As(\mathbf{1} \rtimes \textbf{G}')$ as $\varphi \in \As(\textbf{G}')$, then $ \varphi_{P'}$ is the usual constant term of $\varphi \in \As(\textbf{G}')$ along $P'$ defined by
\[
\varphi_{P'}(g)=\int_{[U']} \varphi(ug) du.
\] 
Like automorphic forms on $\textbf{G}'$, we can decompose automorphic forms on $N(\A_F) \rtimes \textbf{G}'$. To describe it precisely, we introduce some notations.

Write $X(G')$ for the group of $F$-rational characters of $G'$. 
Let $(\aa_0')^{*}$ be the $\mathbb{R}$-vector space spanned by the lattice $X(M_0')$ 
and $\aa_0'=\Hom((\aa_0')^{*},\RR)$ its dual space. 
The canonical pairing on $(\aa_0')^* \times \aa_0'$ is denoted by $\la \ , \ \ra$.
Let $\Delta_0'$ and $(\Delta_0')^{\vee}=\{ \alpha^{\vee} \co \alpha \in \Delta_0' \}$ 
be the sets of simple roots and simple coroots in $(\aa_0')^{*}$ and $\aa_0'$, respectively. 
Write $(\hat{\Delta}_0')^{\vee}$ and $\hat{\Delta}_0'$ for 
the dual bases of $\Delta_0'$ and $(\Delta_0')^{\vee}$, respectively. 
(In other words, $(\hat{\Delta}_0')^{\vee}$ and $\hat{\Delta}_0'$ are 
the set of coweights and weights, respectively.)
For a standard parabolic subgroup $P'=M_{P'}U_{P'}$ of $G'$, 
write $A_{M_{P'}}$ for the maximal split torus in the center of $M_{P'}$ 
and $\mathfrak{a}^{*}_{P'}=X(M_{P'}) \otimes_{\mathbb{Z}} \mathbb{R}$ and $\aa_{P'}$ for its dual space.
\par

Let $Q' \subset P'$ be a pair of standard parabolic subgroups of $G'$.
There is a canonical injection $\aa_{P'} \hookrightarrow \aa_{Q'}$ 
and a canonical surjection $\aa_{Q'} \twoheadrightarrow \aa_{P'}$ 
induced by two inclusion maps $A_{M_{P'}} \hra A_{M_{Q'}}$ and $M_{Q'} \hra M_{P'}$. 
So we have a canonical decomposition
\[
\aa_{Q'}=\aa_{Q'}^{P'} \oplus \aa_{P'}, \quad \quad \aa_{Q'}^{*}=(\aa_{Q'}^{P'})^* \oplus \aa_{P'}^{*}.
\]
In particular, if we take $Q'=P_0'$, we have a decomposition
\[
\aa_0'=(\aa_0')^{P'} \oplus \aa_{P'}, \quad \quad (\aa_0')^{*}=(\aa_0'^{P'})^* \oplus \aa_{P'}^{*}.
\]

Let $\Delta_{P'}$ be the subset of $\Delta_0'$ consisting of simple roots whose restriction to $\aa_{P'}^{*}$ is non-trivial. We denote by $\Delta_{Q'}^{P'}$ the subset of $ \Delta_{Q'}$ appearing in the root decomposition 
of the Lie algebra of the unipotent radical $U_{Q'} \cap M_{P'}$. 
Then for $H \in \aa_{P'}$, $\la \alpha ,H \ra =0$ for all $\alpha \in \Delta_{Q'}^{P'}$ 
and so $\Delta_{Q'}^{P'} \subset (\aa_{Q'}^{P'})^*$. 
Note that $\Delta_{P'}^{G'}=\Delta_{P'}$.
For any $\alpha \in \Delta_{Q'}^{P'}$, 
there is an $\tilde{\alpha}\in \Delta_0'$ whose restriction to $(\aa_{Q'}^{P'})^*$ is $\alpha$. 
Write $\alpha^{\vee}$ for the projection of $\tilde{\alpha}^{\vee}$ to $\aa_{Q'}^{P'}$.
Define 
\[
({\Delta}_{Q'}^{P'})^{\vee}=\{ \alpha^{\vee} |  \ \alpha \in \Delta_{Q'}^{P'} \}.
\] 
Define $\hat{\Delta}_{Q'}^{P'} \subset (\aa_{Q'}^{P'})^*$ and $(\hat{\Delta}^{\vee})_{Q'}^{P'} \subset \aa_{Q'}^{P'}$ 
to be the dual bases of $(\Delta_{Q'}^{P'})^{\vee}$ and $\Delta_{Q'}^{P'}$,  respectively. Thus, $\hat{\Delta}_{Q'}^{P'}$ is the set of simple weights and $(\hat{\Delta}^{\vee})_{Q'}^{P'}$  is the set of coweights.
We simply write $\hat{\Delta}_{P'}^{\vee}$ for $(\hat{\Delta}^{\vee})_{P'}^{G'}$ 
and $\hat{\Delta}_{P'}$ for $\hat{\Delta}_{P'}^{G'} $, respectively.
\par

Let $\tau_{Q'}^{P'}$ be the characteristic function of the subset  
\[ 
\{H \in \aa_0' : \la \alpha,H \ra >0 \text{ for all } \alpha \in \Delta_{Q'}^{P'}\} \subset \aa_0'
\]
and let $\hat{\tau}_{Q'}^{P'}$ be the characteristic function of the subset 
\[ 
\{H \in \aa_0' : \la \varpi,H \ra >0 \text{ for all } \varpi \in \hat{\Delta}_{Q'}^{P'}\} \subset \aa_0'.
\]
Since these two functions depend only on the projection of $\aa_0'$ to $\aa_{Q'}^{P'}$, 
we also regard them as functions on $\aa_{Q'}^{P'}$. 
For simplicity, write $\aa_{P'}$  for $\aa_{P'}^{G'}$ and $\tau_{P'}$, 
$\hat{\tau}_{P'}$ for $\tau_{P'}^{G'}$, $\hat{\tau}_{P'}^{G'}$, respectively.
\par

For each parabolic subgroup $P'=M'U'$ of $G'$, we have a height map
\[
H_{P'} : G'(\A_F) \to \aa_{P'}
\]
characterized by the following properties (see \cite[page 917]{A1}):
\begin{itemize}
\item 
$|\chi|(m)=e^{\la \chi, H_{P'}(m)\ra}$ for all $\chi \in X(M')$ and $m\in M'(\A_F)$;
\item 
$H_{P'}(umk)=H_{P'}(m)$ for all $u \in U'(\A_F), m\in M'(\A_F), k \in K'$.
\end{itemize}
\par

The restriction of $H_{P'}$ on $M'(\A_F)$ is a surjective homomorphism. 
For a pair of parabolic subgroups $Q' \ss P'$ of $G'$, 
it is easy to check that the projection of $H_{Q'}(g) \in \aa_{Q'}$ to $\aa_{P'}$ is $H_{P'}(g)$ for all $g \in G'(\A_F)$.
\par

Denote the kernel of $H_{P'} |_{M'(\A_F)}$ by $M'(\A_F)^1$ 
and the connected component  of $1$ in $A_{M'}(\RR)$ by $A_{M'}(\RR)^0$. 
Write $A_0'$ for $A_{M_0'}(\RR)^0$. 
Then $M'(\A_F)$ is the direct product of the normal subgroup $M'(\A_F)^1$ with $A_{M'}(\RR)^0$, 
and $H_{P'}$ gives an isomorphism between $A_{M'}(\RR)^0$ and $\aa_{P'}$. 
Denote  the inverse of this map by $X \to e^X$. We fix a Haar measure on $\aa_{P'}$.
Write $\rho_0'$ for the half sum of positive roots in $(\aa_{0}')^{*}$ 
and denote by $\rho_{P'}, \rho_{Q'}^{P'}$ the projection of $\rho_0'$ to $(\mathfrak{a}_{P'})^{*}, (\mathfrak{a}_{Q'}^{P'})^{*}$, 
respectively. Recall that $e^{2\la \rho_{P'} ,H_{P'}(p)\ra}=\delta_{P'}(p)$ for $p\in P'(\A_F)$. 
We use the similar notation for $G$ by removing the prime $'$.

Write $\textbf{M}'^1=pr_{G'}^{-1}(M'(\A_F))$. The following is a generalization of \cite[Lemma I.3.2]{Mo} and can be proved in the same way.
\begin{prop}

An automorphic form $\phi \in \As_{P'}(N \rtimes \textbf{G}')$ admits a finite decomposition 
\beq \label{dec}
\phi((n,ue^Xmk))=\sum_i Q_i(X) \phi_i((n,mk)) e^{\la \lambda_i+\rho_{P'},X \ra}
\eeq
for $u \in U'(\A_F), X\in \mathfrak{a}_{P'}, m \in \textbf{M}'^1$ and $k\in \textbf{K}'$, 
where $\lambda_i\in \mathfrak{a}_{P'}^* \otimes_{\RR} \CC$, $Q_i \in \CC[\mathfrak{a}_{P'}]$ 
and $\phi_i \in \As_{P'}(N \rtimes\textbf{G}')$ satisfies $\phi_i(n,e^Xg)=\phi_i(n,g)$ for $X \in \mathfrak{a}_{P'}$ and $g \in \textbf{G}'$.
\end{prop}
We denote the finite set of exponents $\lambda_i$ appearing in the above decomposition of $\phi$ by $\mathcal{E}_{P'}(\phi)$. When $N=\textbf{1}$, the above proposition is nothing but \cite[Lemma I.3.2]{Mo}.


As we shall see in the next section, we are mainly interested in the situation $N \rtimes G' \ss G$. In this case, we will consider the restriction of the constant terms of functions in $\As(\textbf{G}) \otimes \As(N \rtimes \textbf{G}')$ to $N(\A_F) \rtimes \textbf{G}'$. Therefore, we introduce the following notation to denote it.

For $\phi \in \As(\textbf{G}), \ \phi' \in \As(N \rtimes \textbf{G}')$ and a parabolic subgroup $\tilde{S}=S_1 \times S_2'$ of $G \times G'$,  a function $(\phi \otimes \phi')_{\tilde{S}}: N(\A_F) \rtimes \textbf{G}' \to \CC$ is defined by
\[ 
(\phi \otimes \phi')_{\tilde{S}}(n,g) =  \phi_{S_1}(ng) \phi'_{S_2'}(n, g).
\]

\section{\textbf{The GGP triplets}}
From now on, we focus on special triplets of algebraic groups $(G,G',N)$ which appear in the setting of the GGP conjecture. 

Let $E$ be either $F$ or a quadratic extension of $F$ with adele ring $\A_E$. For a place $v$ of $F$, denote by $E_v$ the localization of $E$ at $v$.  Fix a non-trivial additive character $\psi$ of $F \bs \A_F$ and the quadratic character $\omega$ of $\A_F^{\times}$ associated to $E/F$ by the global class field theory. (When $E=F$, $\omega$ is just the trivial character.) Fix a character $\mu$ of $E^{\times} \bs \A_E^{\times}$ such that $\mu|_{\A_F^{\times}}=\omega$. Write $|\cdot|_{E_v}$ for the normalized absolute valuation on $E_v$ and for $x\in \A_E^{\times}$, put $|x|_E=\prod_v |x|_{E_v}$. When $E=F$, we simply write $|\cdot|_{v}$ for $|\cdot|_{F_v}$ and $|\cdot|$ for $|\cdot|_F$. 
Occasionally, we regard $\psi$ as the character of $E \bs \A_E$ by composition with the trace map $\mathrm{Tr}_{E/F}$. We also view $\mu$ and $|\cdot |_E$ as the characters of $GL_n(\A_E)$ by composition with the determinant map. Sometimes, we denote $\mu \circ \det_{GL_n}$ by $\mu(n)$.

Fix $\epsilon \in \{\pm\}$. Let $( W_{n},(\cdot , \cdot)_n )$ be a non-degenerate $n$-dimensional $\epsilon$-Hermitian space over $E$ and $W_m$ its a $m$-dimensional subspace. Let $W_m^{\perp}$ be the orthogonal complement of $W_m$ in $W_n$ and $(\cdot , \cdot)_m$ be the restriction of $(\cdot , \cdot)_n$ to $W_m$. We further assume that $(\cdot , \cdot)_m$ is also non-degenerate and 

\begin{itemize}
\item $\epsilon \cdot (-1)^{\text{dim}W_m^{\perp}}=-1$;
\item $W_m^{\perp}$ in $W_n$ is a split space.
\end{itemize}

When $n-m=2r$,  the second condition means that there exist $r$-dimensional isotropic subspaces $X,X^{*}$ of $W_m^{\perp}$ such that $W_m^{\perp}=X  \oplus X^{*}$. When $n-m=2r+1$, it means that there exist $r$-dimensional isotropic subspaces $X,X^{*}$ and a non-isotropic hermitian line $E$ in $W_m^{\perp}$ such that $E$ is orthogonal to $X \oplus X^{*}$.

Let $k,k'$ be Witt indices of $W_n, W_m$, respectively. When $n-m=2r$ (resp. $n-m=2r+1$), it is easy to check $k = k'+r$ (resp. $k'+r\le k \le k'+r+1$). Let $G_n$ be the isometry group of $( W_{n},(\cdot , \cdot)_n )$ and $G_m$ that of $( W_{m}, (\cdot , \cdot)_m )$. We identify $G_m$ as the subgroup of $G_n$ which acts trivially on $W_m^{\perp}$. We fix a norm $\| \cdot \|_{G_n}$ on $G_n(\A_F)$ and define a norm $\| \cdot \|_{G_m}$ on $G_m(\A_F)$ by 
\[\| g' \|_{G_m} \col \| g' \|_{G_n}, \quad \quad g' \in G_m(\A_F). 
\]
Until section 6, we fix $n,m$ and use the notation $G$ and $G'$ to denote $G_n$ and $G_m$, respectively. 

We fix maximal totally isotropic subspaces $X_{k'}',Y'_{k'}$ of $W_m$ which are in duality with respect to $(\cdot , \cdot)_m$. We fix also their complete flags 
\begin{align*}
& 0=X_0' \subset X_1' \subset \cdots \subset X_{k'}',
\\&0=Y_0' \subset Y_1' \subset \cdots \subset Y_{k'}',
\end{align*}
in order that for all $1\le i \le k'$, $X_i'$ and $Y_i' $ are in duality with respect to $(\cdot , \cdot)_m$. Let $P_0'=M_0'U_0'$ be the parabolic subgroup of $G'$ defined as the stabilizer of the above complete flag of subspaces within $X_{k'}'$ with the Levi subgroup $M_0'$, which stabilizes the complete flag of subspaces within $Y_{k'}'$. Let $A_{M_0'}$ be the maximal split torus in $M_0'$ and call any $F$-parabolic subgroup $P'$ of $G'$ which contains $P_0'$ \emph{standard}. We also fix a maximal compact subgroup $K'=\prod_v K_v'$ of $G'(\A_F)$ which satisfies $G'(\A_F)=P_0'(\A_F)K'$ and for every standard parabolic subgroup $P'=M'U'$ of $G'$, 
\[ P'(\A_F)\cap K'=(M'(\A_F) \cap K')(U'(\A_F)\cap K')\]
and $M'(\A_F) \cap K'$ is still maximal compact in $M'(\A_F)$ (see \cite[I.1.4]{Mo}). 

Let $P(\overline{X_{k'}'})$ be the parabolic subgroup of $G$ defined as the stabilizer of the same complete flag of subspaces within $X_{k'}'$. Fix a complete flag $\overline{X}$ of subspaces within $X$ and let $N_{n,r}$ be the unipotent radical of the parabolic subgroup of $G$ defined as the stabilizer of $\overline{X}$. Let $P_0=M_0U_0$ be a minimal $F$-parabolic subgroup of $G$ contained in $P(\overline{X'_{k'}})$ and stabilizes $\overline{X}$. Note that $N_{n,r} \ss U_0$. We also fix a maximal compact subgroup $K$ of $G(\A_F)$ containing $K'$ as we choose $K'$ in $G'(\A_F)$.
Note that the map $P \to P \cap G'$ gives a bijection between the set of $F$-parabolic subgroups of $G$ that contain $P(\overline{X'_{k'}})$ and the set of standard parabolic subgroups of $G'$. Throughout the rest of the paper, a parabolic subgroup of $G$ (resp. a parabolic subgroup of $G'$) always means a standard parabolic subgroup of $G$ containing $P(\overline{X'_{k'}})$ (resp. a standard parabolic subgroup of $G'$). When $P$ is a parabolic subgroup of $G$, we use the letter $P'$ to denote $P \cap G'$ and vice versa. In particular, $\left(P(\overline{X'_{k'}})\right)'=P_0'$.

 Since $\aa_{0}'$ (resp. $\aa_{0}$) is isomorphic to $\RR^{k'}$ (resp. $\RR^{k}$), we regard functions on $\aa_{0}'$ as those of $\aa_{0}$ by composing them with the projection. Additionally, since both $P$ and $P'$ contain $P(\overline{X_{k'}'})$ and $P_0'$, respectively, $\aa_P$ and $\aa_{P'}$ are isomorphic. Therefore, we will identify them and establish fixed Euclidean norms $|\cdot |_{P}$ on $\aa_P$ and $|\cdot |_{P'}$ on $\aa_{P'}$ so that they become the same through the identification of $\aa_{P}$ and $\aa_{P'}$.

Let $I$ be either $G$ or $G'$. Let $I(\A_F)^1=\cap_{\chi \in X(I)} \ker (|\chi|)$ and $I(F \otimes_{\mathbb{Q}}\mathbb{R})^1=I(\A_F)^1 \cap I(F \otimes_{\mathbb{Q}}\mathbb{R})$, where $X(I)$ is the additive group of homomorphisms from $I$ to $GL(1)$ defined over $F$.
Then $I(F) \subset I(\A_F)^1$. When $E=F$ and $\epsilon=-1$, $I$ is the symplectic group. Then it has a non-trivial double cover, which is called the metaplectic group. We define the topological group $\textbf{I}$ and the surjective morphism $pr: \textbf{I} \to I(\A_F)$  according to $(E,\epsilon)$ as follows:
\begin{align*}
& \textbf{I}=\begin{cases} \text{a metaplectic double cover of $I(\A_F)$},  \ \ \quad \quad \quad \quad  \quad \quad \quad \quad E=F, \ \epsilon=-1; \\ I(\A_F),  \quad \quad \quad \quad \quad \quad \quad \quad \quad \quad \quad \quad \quad \quad \quad \quad \quad \quad \quad \quad \  \ \ \ \ \  \text{otherwise}, \end{cases}
\\&
pr=\begin{cases} \text{a double covering map whose kernel is } \{\pm 1\},   \ \quad \quad \quad \quad E=F, \ \epsilon=-1; \\ \text{identity map}, \quad \quad  \quad \quad \quad \quad \quad \quad \quad \quad \quad \quad   \ \quad \quad \quad \quad \quad \quad \text{otherwise.}
\end{cases}
\end{align*}

The letters $P,Q$ (resp. $P',Q'$) are reserved for parabolic subgroups of $G$ (resp. $G'$) and $M,L$ (resp. $M',L'$) for their Levi subgroups and $U,V$ (resp. $U',V'$) for their unipotent radicals. 
We fix the minimal $F$-parabolic subgroup $P_0 \times P_0'$ of $G \times G'$ and call $P \times Q'$ 
a parabolic subgroup of $G \times G'$ if $P$ and $Q'$ are parabolic subgroup of $G$ and $G'$ respectively. 
To distinguish the notation of parabolic subgroups of $G \times G'$ with those of $G$, we use the letter $\tilde{S}$ for parabolic subgroups of  $G \times G'$. 
Let $\theta$ be the involution on the set of parabolic subgroups of $G \times G'$ sending $P \times Q'$ to $Q \times P'$. 
Then the map $P \to P \times P'$ gives a bijection between parabolic subgroups of $G$ 
and $\theta$-invariant parabolic subgroups of $G \times G'$. We use the letter $\bar{P}$ to denote $P \times P'$.
Let $N$ be a unipotent $F$-subgroup of $G$ which admits a $G'$-action $\sigma:G' \to \text{Aut}(N)$ such that $N \rtimes G'$ is a $F$-subgroup of $G$.

For $\phi \in \As(\textbf{G}), \ \phi' \in \As(N \rtimes \textbf{G}')$, we define a function $(\phi \otimes \phi')_{P,N} \co  \textbf{G}' \to \CC$ as 
\[ 
(\phi \otimes \phi')_{P,N} (g) \col \int_{[N]} (\phi \otimes \phi')_{\overline{P}} (n,g) dn. 
\] 
\begin{rem}\label{con}
When $N$ is the trivial group $\mathbf{1}$, for every $\phi \in \As(\textbf{G})$ and $\phi' \in \As(\mathbf{1} \rtimes \textbf{G}')$, $(\phi \otimes \phi')_{P,\mathbf{1}}=\phi_P \cdot \phi'_{P'}.$ 
\end{rem}

\begin{prop} 
For $\phi \in \As(\textbf{G}), \phi' \in \As(N \rtimes \textbf{G}')$, 
$(\phi \otimes \phi')_{P,N}$ is a $U'(\A_F)P'(F)$-invariant smooth function on $\textbf{G}'$ for any parabolic subgroup $P$ of $G$.
\end{prop}

\begin{proof} The smoothness readily follows from those of $\phi$ and $\phi'$. To prove the left $U'(\A_F)P'(F)$-invariance,
choose an arbitrary $v \in U'(\A_F)$ and $p\in P'(F)$. 
Then $(\phi \otimes \phi')_{P,N}(apg)$ equals 
\[ 
\int_{[N]} \int_{[U']} \int_{[U]} \phi(nu_1vpg) \phi'(n,u_2vpg) du_1du_2dn 
= \int_{[N]} \int_{[U']}  \int_{[U]} \phi(nu_1pg) \phi'(n,u_2pg)  du_1du_2dn. 
\] 
(We used the change of variables $u_1 \to u_1v^{-1}$, $u_2 \to u_2v^{-1}$.)
For each $i=1,2$, write $u_i'=p^{-1}u_ip$. 
Since $p \in P'(F)\ss P'(\A_F)^1 \ss P(\A_F)^1$, we have $du_i'=du_i$ because $\delta_{P'}(p^{-1})=\delta_P(p^{-1})=1$.
We also write $n'=\sigma(p^{-1})n$. 
Then $dn'=\delta_{\sigma} (p^{-1}) dn$, where 
\[
\delta_\sigma(p^{-1}) = |\det(d\sigma(p^{-1}); \mathrm{Lie}(N))|
\]
with $d\sigma(p^{-1}) \colon \mathrm{Lie}(N) \to \mathrm{Lie}(N)$ being the differential of $\sigma(p^{-1})$.
Since $p \in G'(F) \subset G'(\A_F)^1$, we have $\delta_{\sigma} (p^{-1}) =1$. 
Thus 
\[
\int_{[N]} \int_{[U']} \int_{[U]} \phi(nu_1pg) \phi'(n,u_2pg)  du_1du_2dn
\]
equals 
\[ 
\int_{[N]} \int_{[U']}\int_{[U]}   \phi(p n'u_1'g)) \phi'((1,p) \cdot (n',u_2'g)) du_1' du_2'dn' 
\] 
and since $\varphi$ is $N(F) \rtimes G'(F)$-invariant, the last integral equals
\[ 
\int_{[N]} \int_{[U']}\int_{[U]}  \phi (n'u_1'g) \phi (n',u_2'g) du_1' du_2'  dn' = (\phi \otimes \phi')_{P,N} (g). 
\] 
This shows that $(\phi \otimes \phi')_{P,N} $ is a left $U'(\A_F)P'(F)$-invariant function. 
\end{proof}

\section{\textbf{Mixed truncation}}
In this section, we introduce the mixed truncation which is an essential ingredient to define regularized periods. 

\noindent 
Let $T' \in \aa_0'$. 
For $\phi_1 \in \As(\textbf{G}), \phi_2 \in \As(N \rtimes \textbf{G}')$, we define their mixed  truncation by
\[
\Lambda_m^{T'}(\phi_1 \otimes \phi_2)(g)
=\sum_{P'} (-1)^{\dim \mathfrak{a}_{P'}}\sum_{\gamma \in {P'}(F)\bs {G'(F)}} 
(\phi_1 \otimes \phi_2)_{P,N}(\gamma g)  \hat{\tau}_{P'}(H_{P'}(\gamma g)-T') \  , \ g\in \textbf{G}'.
\]
More generally, we define a partial mixed truncation by
\[ 
\Lambda_m^{T',P'}(\phi_1 \otimes \phi_2 )(g)
=\sum_{Q'\subset P'} (-1)^{\dim \mathfrak{a}_{Q'}^{P'}}\sum_{\delta \in {Q'}(F)\bs {P'}(F)}  
(\phi_1 \otimes \phi_2)_{Q,N}(\delta g) \hat{\tau}^{P'}_{Q'}(H_{Q'}(\delta g)-T').
\]
Then $\Lambda_m^{T',P'}(\phi_1 \otimes \phi_2 )$ is a function on $P'(F) \bs \textbf{G}'$.
\par

For any two parabolic groups $Q' \ss P'$, Langlands' combinatorial lemma asserts that 
\begin{equation} \label{ll0}
\sum_{Q' \ss R' \ss P'} (-1)^{\dim\aa_{Q'}^{R'}} \hat{\tau}_{Q'}^{R'} \tau_{R'}^{P'}
=
\begin{cases} 
1 \quad \text{ if } P'=Q' \\ 
0 \quad \text{ otherwise.} 
\end{cases}
\end{equation}
\par
We put
\beqn 
\Gamma_{Q'}^{P'}(H',X)
= \sum_{Q' \ss R' \ss P'} (-1)^{\dim\aa_{R'}^{P'}} \cdot \tau_{Q'}^{R'}(H')\hat{\tau}_{R'}^{P'}(H'-X), \quad H',X \in \aa_0'.
\eeqn
Note that this function depends only on the projection of $H'$ and $X$ onto $\aa_{Q'}^{P'}$ and for a fixed $X$, it is compactly supported function of $H' \in \aa_{Q'}^{P'}$.

The following lemma is a consequence of  Langlands' combinatorial lemma.
\begin{lem} For $H',X \in \aa_0',$
\begin{align}
\label{ll1}
\hat{\tau}_{Q'}^{P'}(H'-X)
&=\sum_{Q' \ss R' \ss P'}(-1)^{\dim\aa_{R'}^{P'}} \cdot \hat{\tau}_{Q'}^{R'}(H) \Gamma_{R'}^{P'}(H',X),
\\
\label{ll2}
\tau_{Q'}^{P'}(H'-X)
&=\sum_{Q' \ss R' \ss P'} \Gamma_{Q'}^{R'}(H'-X,-X) \cdot \tau_{R'}^{P'}(H').
\end{align}

\end{lem} 

\begin{proof}Note that 
\begin{align*}&\sum_{Q' \ss R' \ss P'}(-1)^{\dim\aa_{R'}^{P'}} \cdot \hat{\tau}_{Q'}^{R'}(H') \Gamma_{R'}^{P'}(H',X)=\sum_{Q' \ss R' \ss P'}\sum_{R' \ss S' \ss P'}(-1)^{\dim\aa_{R'}^{S'}} \cdot \hat{\tau}_{Q'}^{R'}(H') \tau_{R'}^{S'}(H') \hat{\tau}_{S'}^{P'}(H'-X) 
\\&=\sum_{Q' \ss S' \ss P'}(-1)^{\dim\aa_{Q'}^{S'}}\cdot \hat{\tau}_{S'}^{P'}(H'-X)\cdot \left( \sum_{Q' \ss R' \ss S'}(-1)^{\dim\aa_{Q'}^{R'}} \cdot \hat{\tau}_{Q'}^{R'}(H') \tau_{R'}^{S'}(H') \right).
\end{align*}
By (\ref{ll0}), it equals to  $\hat{\tau}_{Q'}^{P'}(H'-X)$. This proves  (\ref{ll1}).

On the other hand, 
\begin{align*}&\sum_{Q' \ss R' \ss P'} \Gamma_{Q'}^{R'}(H'-X,-X) \cdot \tau_{R'}^{P'}(H')
\\&=\sum_{Q' \ss R' \ss P'} \sum_{Q' \ss S' \ss R'} (-1)^{\dim\aa_{S'}^{R'}} \cdot \tau_{Q'}^{S'}(H'-X)\cdot \hat{\tau}_{S'}^{R'}(H') \tau_{R'}^{P'}(H')
\\&=\sum_{Q' \ss S' \ss P'} \tau_{Q'}^{S'}(H'-X) \sum_{S' \ss R' \ss P'} (-1)^{\dim\aa_{S'}^{R'}} \cdot \hat{\tau}_{S'}^{R'}(H') \tau_{R'}^{P'}(H').
\end{align*}
By (\ref{ll0}), it equals to  $\tau_{Q'}^{P'}(H'-X)$. This proves  (\ref{ll2}).

\end{proof}
The mixed truncation has the following property.

\begin{lem} \label{mt} 
Let $\phi \in \As(\textbf{G}), \phi' \in \As(N \rtimes \textbf{G}')$. Then for all $g \in \textbf{G}'$,
\[ 
\Lambda_m^{T'+S',P'}(\phi \otimes \phi')(g)
=\sum_{Q'\ss P'} \sum_{\delta \in Q'(F) \bs P'(F)}
\Lambda_m^{T',Q'}(\phi \otimes \phi')(\delta g) \cdot \Gamma_{Q'}^{P'}(H_Q(\delta g)-T',S').
\]
\end{lem}

\begin{proof} 
Using (\ref{ll1}), 
$\Lambda_m^{T'+S',P'}(\phi \otimes \phi')(g)$ equals
\begin{align*}
&\sum_{Q'\subset P'} (-1)^{\dim \mathfrak{a}_{Q'}^{P'}} \cdot \sum_{\delta \in Q'(F)\bs P'(F)}  
(\phi_1 \otimes \phi_2)_{Q,N}(\delta g)  
\\\notag
&\times
\left( \sum_{Q' \ss R' \ss P'}(-1)^{\dim\aa_{R'}^{P'}} 
\cdot \hat{\tau}_{Q'}^{R'}(H_{Q'}(\delta g)-T') \Gamma_{R'}^{P'}(H_{R'}(\delta g)-T',S') \right).
\end{align*}
Write $\delta=\delta_1 \delta_2,$ where $\delta_1 \in Q'(F)\bs R'(F), \delta_2 \in R'(F) \bs P'(F)$. 
Then $H_{R'}(\delta g)=H_{R'}(\delta_2g)$ and thus 
\begin{align*}
&\Lambda_m^{T'+S',P}(\phi \otimes \phi')(g)
\\&= 
\sum_{R'\subset P'} \sum_{\delta_2 \in R'(F)\bs P'(F)} 
\left(  \sum_{Q' \ss R' } \sum_{\delta_1 \in Q'(F)\bs R'(F)} (-1)^{\dim\aa_{Q'}^{R'}}\cdot 
(\phi_1 \otimes \phi_2)_{Q,N}(\delta_1\delta_2 g)  
\cdot \hat{\tau}_{Q'}^{R'}(H_{Q'}(\delta_1\delta_2g)-T') \right) 
\\&\times
 \Gamma_{R'}^{P'}(H_{R'}(\delta_2 g)-T',S')
\\&=
\sum_{R'\subset P'} \sum_{\delta_2 \in R'(F)\bs P'(F)} \Lambda_m^{T',R'}(\phi \otimes \phi')(\delta_2 g) \cdot \Gamma_{R'}^{P'}(H_{R'}(\delta_2 g)-T',S').
\end{align*}
This completes the proof.
\end{proof}

For any two parabolic subgroups $Q'\ss P'$, put
\[
\sigma_{Q'}^{P'}=\sum_{Q' \ss R'} (-1)^{\dim\aa_{Q'}^{R'}} \cdot \tau_{P'}^{R'} \hat{\tau}_{R'}.
\]
Using the binomial theorem, it is easy to check that for any parabolic subgroup $P_1' \supseteq P'$,
\beq 
\label{bin} \tau_{P'}^{P_1'} \hat{\tau}_{P_1'}=\sum_{P_1' \ss R'} \sigma_{P'}^{R'}.
\eeq
\par
For a parabolic subgroup $P=UM$ (resp. $P'=U'M'$) of $G$ (resp. $G'$), write $\textbf{M}^1=pr^{-1}(M^1(\A_F
))$ (resp. $\textbf{M}'^1=pr^{-1}(M'^1(\A_F
))$). 
Let $w$ be a compact subset of $U_0(\A_F) \textbf{M}_{0}^1$ and $t_0 \in \aa_0$. Note that $\textbf{M}_0$ is the direct product of the normal subgroup $\textbf{M}_0^1$ with $A_{0}$. Put
\[A_0(t_0)=\{ a \in A_0 \mid \la \beta, H_{P_0}(a)-t_0 \ra >0\ \ \text{ for all } \beta \in \Delta_0^{P} \}.
\]
For any parabolic subgroup $P$  of $G$ and $T \in \aa_0$, define a Siegel set and a truncated Siegel set relative to $P$ by 
\begin{align*}
& \S^{P}=S^{P}(t_0)=\{x=pak \ | \ p \in w, a \in A_0(t_0), k \in \textbf{K}\}, 
\\&  \S^{P}(T)=\S^{P}(t_0,T)
=\{x \in \S^{P}(t_0) \ | \ \la \alpha,H_{P_0}(x)-T\ra \le 0 \ \ \text{ for all }  \alpha \in \hat{\Delta}_0^{P} \},
\end{align*}
respectively. 
When $P=G$, we simply write $\S= \S^{G}$ and $\S(T)= \S^{G}(T)$. Let $F^{P}(g,T)$ be the characteristic function on $P(F)\S^{P}(T)$ in $\textbf{G}$. We take $w', \S^{P'},\S^{P'}(T')$ and $F^{P'}(g,T')$ similarly. We choose $w,w',t_0,t_0'$ so as to $\textbf{G}=P(F)\S^P$ and $\textbf{G}'=P'(F)\S^{P'}$.
\par
 The following lemma, which is a generalization of \cite{IY}{,  Lemma 2.1], is crucial to prove the rapidly decreasing property of $\Lambda_m^{T',P'}$.
\begin{lem}\label{kl} For any $w',t_0'$, we can choose $w$ and $t_0$ so that $\S^{P'} \subset \S^P$.

\end{lem} 

\begin{proof}
When $k = k'+r$, the lemma is obvious because $P_0'=P_0 \cap G'$. Suppose $k > k'+r$ and we fix $w',t_0'$. Since $A_0' \ss A_0$, we can choose $t_0 \in \aa_0$ such that $\ A_0'(t_0') \ss A_0(t_0)$. We shall find a compact subset $w \ss U_0(\A_F) \textbf{M}_{0}^1$ such that $w'A_0'(t_0')\textbf{K}' \ss wA_0(t_0)\textbf{K}.$\\ Choose arbitrary $g \in w'A_0'(t_0')\textbf{K}'$ and decompose it $g=u'm'a'k'$, where $u'\in U_0',m' \in \textbf{M}_0'^1,a' \in A_0'(t_0'),k\in \textbf{K}'$.
Note that $U_0' \ss U_0, \ A_0'(t_0') \ss A_0(t_0)$ and $\textbf{K}' \ss \textbf{K}$, though $P_0'$ is not contained in $P_0$. Let $W_{m,k'}$ be the orthogonal complement of $X_{k'}'+Y_{k'}'$ in $W_m$ and $G_{k'}'$ its isometry group. Since $(M_0')^1=(A_{M_0'})^1 \times G_{k'}'$ and $A_{M_0'} \ss A_{M_0}$, we may assume $m' \in \textbf{G}_{k'}'(\A_F).$ Let $W_{n,k'}$ be the orthogonal complement of $X_{k'}'+Y_{k'}'$ in $W_n$ and put $G_{k'}$ its isometry group. Since $k>k'+r$, $G_{k'} \cap P_0$ is a proper parabolic subgroup of $G_{k'}$. Since $G_{k'}' \ss G_{k'}$, using the Iwasawa decomposition of $G_{k'}$ with respect to $G_{k'} \cap P_{0}$, we can decompose $m'$ in $\textbf{G}_{k'}$ as $u_0m_0k_0$, where $u_0\in \textbf{G}_{k'} \cap U_{0}, m_0 \in \textbf{G}_{k'} \cap \textbf{M}_0,k_0 \in \textbf{G}_{k'} \cap \textbf{K}$. Since $\textbf{G}_{k'}$ and $A_0'$ are commutative, $u'm'a'k'=u'u_0m_0k_0a'k'=u'u_0m_0a'k_0k$. This shows that we can find $w$ as desired.
\end{proof}

Now, we prove the rapidly decreasing property of $\Lambda_m^{T',P'}$.

\begin{lem} \label{rd}
Let $\phi \in \As(\textbf{G}), \phi' \in \As(N \rtimes \textbf{G}')$.
Then for every positive integer $N_1$ and $k \in \textbf{K}'$,
\[
\sup_{m \in \S^{P'} \cap \textbf{M}'^1} |\Lambda^{T',P'}_m(\phi \otimes \phi')(mk)| \cdot \|m\|_{G'}^{N_1}< \infty. 
\] 
\end{lem}

\begin{proof} 
It is enough to show  when $P'=G'$ and $k=1$.  Fix $T' \in \aa_0'$ and  for each parabolic subgroup $P'$ of $G'$, choose $\mathcal{S}^{P'}$ so that $\textbf{G}'=P'(F)\S^{P'}$. \cite[Lemma 6.4]{A1} says that
\beq \label{pu} 
\sum_{P_1' \ss P'} \sum_{\delta_1 \in P_1'(F) \bs P'(F)}
F^{P_1'}(\delta_1x,T') \tau_{P_1'}^{P'}(H_{P_1'}(\delta_1x)-T')=1.
\eeq
By substituting (\ref{pu}) into the definition of $\Lambda_m^{T'}(\phi_1 \otimes \phi_2)(g)$, we have
\begin{align*}
&\Lambda_m^{T'}(\phi_1 \otimes \phi_2)(g)
\\&=
\sum_{P'} (-1)^{\dim \mathfrak{a}_{P'}}\sum_{\delta \in P'(F)\bs G'(F)} 
(\phi_1 \otimes \phi_2)_{P,N}(\delta g)  \hat{\tau}_{P'}(H_{P'}(\delta g)-T') \  
\\&=
\sum_{P_1' \ss P'} (-1)^{\dim \mathfrak{a}_{P'}}\sum_{\delta \in P'(F)\bs G'(F)} 
\left( \sum_{\delta_1 \in {P_1'}(F) \bs P'(F)} F^{P_1'}(\delta_1\delta g,T') 
\tau_{P_1'}^{P'}(H_{P_1'}(\delta_1\delta g)-T') \right)
\\&\times 
(\phi_1 \otimes \phi_2)_{P,N}(\delta g)  \hat{\tau}_{P'}(H_{P'}(\delta g)-T')
\\&=
\sum_{P_1' \ss P'} (-1)^{\dim \mathfrak{a}_{P'}} \sum_{\delta \in {P_1'}(F)\bs G'(F)}
F^{P_1'}(\delta g,T')  (\phi_1 \otimes \phi_2)_{P,N}(\delta g)  
\tau_{P_1'}^{P'}(H_{P_1'}(\delta g)-T') \hat{\tau}_{P'}(H_{P'}(\delta g)-T').
\end{align*}
(In the last equality, we used the fact 
$(\phi_1 \otimes \phi_2)_{P,N}(\delta g)  \hat{\tau}_{P'}(H_{P'}(\delta g)-T')
=(\phi_1 \otimes \phi_2)_{P,N}(\delta_1\delta g)  \hat{\tau}_{P'}(H_{P'}(\delta_1 \delta g)-T')$.)
\par

Using (\ref{bin}), we have
\begin{align*}
\Lambda_m^{T'}(\phi_1 \otimes \phi_2)(g)
&=
\sum_{P_1' \ss P' \ss Q'} (-1)^{\dim {\mathfrak{a}_{P'}}}  \sum_{\delta \in P_1'(F)\bs G'(F)} 
F^{P_1'}(\delta g,T')  (\phi_1 \otimes \phi_2)_{P,N}(\delta g)  \sigma_{P_1'}^{Q'}(H_{P'}(\delta g)-T')
\\&= 
\sum_{P_1' \ss Q'} \sum_{\delta \in P_1'(F) \bs G'(F)}
F^{P_1'}(\delta g,T')\sigma_{P_1'}^{Q'}(H_{P_1'}(\delta g)-T')  \left(\int_{[N]} \phi_{P_1,Q}(n,\delta g )dn \right)
\end{align*}
where 
\[
\phi_{P_1,Q}(n,g)=\sum_{P_1 \ss R \ss Q}(-1)^{\dim\aa_R} \cdot (\phi \otimes \phi')_{\bar{R}}(n, g). 
\]
\par
By applying the similar argument in the proof of \cite[Lemma 8.2.1(1)]{LR}, one can show that 
\[  
(-1)^{\dim \aa_{P_1}} \phi_{P_1,Q}(n,g)
= \sum_{\bar{P_1} \ss \tilde{S} \ss \bar{Q}} (-1)^{\dim \aa_{\tilde{S}}} \phi_{\tilde{S},\bar{Q}}(n,g), 
\]
where $\tilde{S}=(S_1,S_2')$ runs over all parabolic subgroups of $G \times G'$ 
between $\bar{P_1}$ and $\bar{Q}$ such that $S_1 \cap S_2=P_1$
and
\[
\phi_{\tilde{S},\bar{Q}}(n,g)
=\sum_{\tilde{S} \ss \tilde{R} \ss \bar{Q}}(-1)^{\dim\aa_{\tilde{R}}} \cdot (\phi \otimes \phi')_{\tilde{R}}(n, g). 
\]
Note that 
\[ |  \phi_{\tilde{S},\bar{Q}}(n,g)  |= \left|\phi_{S_1,Q}(ng) \right| \cdot | \phi'_{S_2',Q}(n,g)  |
\]
where \[\phi_{S_1,Q}(h)=\sum_{S_1\ss R \ss Q} (-1)^{\dim\aa_{{R}}} \phi_R(h), \quad h \in \textbf{G}  \]
and 
\[ \phi'_{S_2',Q'}(n,g)=\sum_{S_2'\ss R' \ss Q'} (-1)^{\dim\aa_{{R'}}} \phi'_{R'}(n,g), \quad (n,g) \in N \rtimes \textbf{G}'.\]

Choose $\delta\in G'(F)$ and $g\in \S'$ such that 
$F^{P_1'}(\delta g,T')\sigma_{P_1'}^{Q'}(H_{P_1'}(\delta g)-T')\ne 0$. Note that $\sigma_{P_1'}^{Q'}$ is a characteristic function by \cite[Lemma 6.1]{A1}. Therefore, by Lemma $\ref{kl}$, we can take $T \in \aa_{0}$ positive enough such that the projection of $T$ to $\aa_{P_{k'}}\simeq \aa_0'$ is $T'$ and $P'(F)\S^{P'}(T') \ss P(F)\S^{P}(T)$ and $F^{P_1}(\delta g,T)\sigma_{P_1}^{Q}(H_{P_1}(\delta g)-T)\ne 0$. The property (iv) of the heights function in \cite[I.2.2]{Mo} shows that there are positive constants $ c_1,\epsilon$ such that 
\[
e^{\| H_{P_1}(a)\|_{P_1}}=e^{\| H_{P_1'}(a)\|_{P_1'}}> c_1 \|a\|^{\epsilon} \, \ \text{ for all } a\in A_{P_1'}(\RR)^0.
\]
If we apply the same arguments in \cite[page 92-96]{A2}  to $\textbf{G}'$ and $\textbf{G}$, respectively, there exists a positive constant $c_{P_1',Q'}$ such that for all $g\in \S'$ and $\delta \in  G'(F)$ with $F^{P_1'}(\delta g,T')\sigma_{P_1'}^{Q'}(H_{P_1'}(\delta g)-T')\ne 0$,
\[ 
 \| g\|_{G}^{N_2\epsilon} 
\cdot \int_{[N]} |\phi_{S_1,Q}(n\delta g)|dn <c_{P_1',Q'} \quad ,   \quad \| g\|_{G'}^{N_2\epsilon} 
\cdot \int_{[N]} |\phi'_{S_2',Q'}(n,\delta g)|dn < c_{P_1',Q'} \]
for all positive integer $N_2$.
Furthermore, there are positive constant $c_{T'}, N_3$ such that 
\[ 
\sum_{P_1' \ss Q'} \sum_{\delta \in P_1'(F) \bs G'(F)}
F^{P_1'}(\delta g,T')\sigma_{P_1'}^{Q'}(H_{P_1'}(\delta g)-T') <c_{T'} \| g\|_{G'}^{N_3}  \ , \ g\in \textbf{G}'^1.
\] 
(See \cite[page 96-97]{A2}.)
\par

Thus if we choose $N_2$ so that $2N_2 \epsilon \ge N_1+N_3$, 
it follows
\[
\sup_{g \in \S' \cap \textbf{G}'^1} |\Lambda^{T'}_m(\phi \otimes \phi')(g)| \cdot \|g\|_{G'}^{N_1}< \infty. 
\]
This completes the proof.
\end{proof}

\section{\textbf{Regularized trilinear periods}}
In this section, we shall define the regularized trilinear periods using the mixed truncation introduced in the previous section. We keep the same notations and conventions as in Section 2.

Let $\textbf{C}$ be the kernel of $pr$ and $\textbf{H}$ be either $\textbf{G}$ or $N \rtimes \textbf{G}'$ or $\textbf{G}'$. Then
\[\textbf{C}=\begin{cases}\{\pm1\}, \quad  \quad \text{ if } (E,\epsilon) = (F,-1),\\ \{1\},\quad \quad  \ \ \text{ if } (E,\epsilon) \ne (F,-1)\end{cases}.
\]
When $(E,\epsilon) = (F,-1)$, for $f\in  \As(\textbf{H})$, we call $f$ \emph{genuine} if \[f(\xi g)= \xi f(g) \quad \quad (\xi \in \textbf{C}, \ g\in \textbf{H}).\] \\ 
When $(E,\epsilon) \ne (F,-1)$, for any $f\in  \As(\textbf{H})$, we have
\[ f(\xi g)= f(g) \quad \quad \ (\xi \in \textbf{C}, \ g\in \textbf{H}),
\]
because $\textbf{C}=\{1\}$.

Put 
\[\As^{gen}(\textbf{H}):=\begin{cases}\As(\textbf{H}), \quad \  \quad \quad \quad \quad \quad \quad \quad \quad \text{ if $(E,\epsilon) \ne (F,-1)$} \\ \la f \in \As(\textbf{H}) \mid \text{$f$ is genuine} \ra, \quad \text{if } (E,\epsilon) = (F,-1)\end{cases}.\]
Then for any $f \in \As^{gen}(\textbf{H})$, there is a character $\chi_f:\textbf{C} \to \{\pm1\}$ such that 
\[ f(\xi g)= \chi_f(\xi)f(g) \quad \quad (\xi \in \textbf{C}, \ g\in \textbf{H}).
\]

\noindent Let $\mf{T}$ be the set of characters of $\textbf{C}^3$ satisfying the condition that for any $\eta=(\eta_1,\eta_2,\eta_3)\in \mf{T}$, the product of $\eta_1$, $\eta_2$, and $\eta_3$ is equal to 1, the trivial character of $\textbf{C}$.
Let $\As(\textbf{G},\textbf{G}')$ denote the space of triplets $(\phi_1,\phi_2,\phi_3) \in \As^{gen}(\textbf{G}) \times  \As^{gen}(N \rtimes \textbf{G}') \times  \As^{gen}(\textbf{G}')$ which satisfy 
\[ \phi_i(\xi g)= \eta_i(\xi) \phi_i(g) \quad \quad (\xi \in \textbf{C}, \ g\in \textbf{G}', \ i=1,2,3)
\]
for some $\eta=(\eta_1,\eta_2,\eta_3)\in \mf{T}$.

For $(\phi,\phi',\phi'') \in \As(\textbf{G},\textbf{G}')$, we consider the following integral
\begin{equation} \label{reg} 
\int_{G'(F)\bs G'(\A_F)^1} \Lambda_m^{T'}(\phi \otimes \phi')(g) \phi''(g) dg.
\end{equation}
By the assumption of $\chi$, the integrand is well defined on $G'(\A_F)$ and thanks to Lemma \ref{rd}, this integral converges absolutely.
\par

\begin{prop} \label{poly} 
The integral in $(\ref{reg})$ is a function of the form $\sum_{\lambda}p_{\lambda}(T')e^{\la \lambda,T'\ra}$, 
where $p_{\lambda}$ is a polynomial in $T'$ and $\lambda$ can be taken from the set 
\[
\bigcup_{P'} \{ \lambda_1 + \lambda_2 + \lambda_3 + \rho_{P} \ | \ 
\lambda_1 \in \mc{E}_{P}(\phi_{P}),\ \lambda_2 \in \mc{E}_{P'}(\phi_{P'}), \ \lambda_3 \in  \mc{E}_{P'}(\phi''_{P'}) \}.
\]
\end{prop}

\begin{proof} 
Since $\Gamma_{P'}^{G'}(X-T',S')$ is compactly supported on $X$, 
using Lemma \ref{rd}, 
we integrate the equation in Lemma \ref{mt} against $\phi''(g)$ over $G'(F) \bs G'(\A_F)^1$ to obtain 
\begin{align*}
&\int_{G'(F) \bs G'(\A_F)^1}\Lambda_m^{T'+S'}(\phi \otimes \phi')(g)\phi''(g) dg 
\\&=
\sum_{P'}  \int_{G'(F) \bs G'(\A_F)^1}  
\left(\sum_{\delta \in P'(F) \bs G'(F)}  
\Lambda_m^{T',P'}(\phi \otimes \phi')(\delta g)  \phi''(g) \cdot\Gamma_{P'}^{G'}(H_{P'}(\delta g)-T',S') 
\right)dg
\\&= 
\sum_{P'}  \int_{P'(F) \bs G'(\A_F)^1}  
\Lambda_m^{T',P'}(\phi \otimes \phi')( g)  \phi''(g) \cdot\Gamma_{P'}^{G'}(H_{P'}( g)-T',S') dg
\\&=
\sum_{P'}  \int_{K'}\int_{M'(F) \bs M'(\A_F)^1}  \int_{\aa_{P'}}^{\#}
\Lambda_m^{T',P'}(\phi \otimes \phi')(e^Xmk)  \phi_{P'}''(e^Xmk) 
\cdot\Gamma_{P'}^{G'}(X-T',S')e^{-2\la \rho_{P'},X\ra} dXdmdk.
\end{align*}

(In the second equality, we used $\phi''(g)=\phi''(\delta g)$ for all $\delta \in G(F)$. 
The symbol $\int_{\aa_{P'}}^{\#}$ in the last equality represents 
the regularized integral of a polynomial exponential function over a cone in a vector space. 
See \cite[Sec 2]{JLR}.)
\par
 
The inner integral over $\aa_{P'}$ can be expressed in terms of 
the Fourier transform of $\Gamma_{P'}^{G'}(\cdot,T')$ evaluated at $\lambda_1+\lambda_2+\lambda_3 +\rho_{P}$  
for $
\lambda_1 \in \mc{E}_{P}(\phi_{P}),\ \lambda_2 \in \mc{E}_{P'}(\phi_{P'}), \ \lambda_3 \in  \mc{E}_{P'}(\phi''_{P'})$. 
Thus by \cite[Lemma 2.2]{A2}, 
the $\#$-integral is a polynomial exponential function in $S'$ 
whose exponents are 
$\{ \lambda_1 + \lambda_2 + \lambda_3 + \rho_{P} \ | \ 
\lambda_1 \in \mc{E}_{P}(\phi_{P}),\ \lambda_2 \in \mc{E}_{P'}(\phi_{P'}), \ \lambda_3 \in  \mc{E}_{P'}(\phi''_{P'}) \}_{Q' \supset P'}$.
\end{proof}

\begin{defn} 
Let $\As_0(\textbf{G},\textbf{G}')$ be the subspace of triplets in $\As(\textbf{G},\textbf{G}')$
such that the polynomial corresponding to the zero exponent of (\ref{reg}) is constant. 
For $(\phi,\phi',\phi'') \in \As_0(\textbf{G},\textbf{G}')$, 
we define its regularized period $\mc{P}(\phi,\phi',\phi'')$ as its value $p_0(T')$. 
We also write 
\[ 
\mc{P}(\phi,\phi',\phi'')
=  \int_{G'(F) \bs G'(\A_F)^1}^{*} (\phi \otimes \phi')_{G,N}(g)\phi''(g) dg.
\]
\end{defn}

For a parabolic subgroup $P'$ of $G'$, we set the $\#$-integral  
\[
\mc{P}_{P'}^{T'}(\phi,\phi',\phi'')
=\int_{P'(F) \bs G'(\A_F)^1}^{\#} \Lambda_m^{T',P'}(\phi \otimes \phi')(g)\phi_{P'}''(g)\tau_{P'}(H_{P'}(g)-T') dg
\] 
as the triple integral
\[ 
\int_{K'} \int_{M'(F) \bs M'(\A_F)^1} \int_{\mf{a}_{P'}}^{\#}
\Lambda_m^{T',P'}(\phi \otimes \phi')(e^Xmk)\phi_{P'}''(e^Xmk)\tau_{P'}(X-T') e^{-2\la \rho_{P'},X\ra }dX dm dk.
\]
Since the functions $m\longrightarrow \Lambda_m^{T',P'}(\phi \otimes \phi')(mk)$ are rapidly decreasing, 
$\mc{P}_{P'}^{T'}(\phi,\phi',\phi'')$ exists if and only if 
\[ 
\la \lambda_1+\lambda_2+\lambda_3 +\rho_{P}, \omega^{\vee} \ra \ne 0
\ \ \text{for all} \ \omega^{\vee} \in (\hat{\Delta}^{\vee})_{P'}, \  
\lambda_1 \in \mc{E}_{P}(\phi_{P}),\ \lambda_2 \in \mc{E}_{P'}(\phi_{P'}), \ \lambda_3 \in  \mc{E}_{P'}(\phi''_{P'}) 
\]
by \cite[Lemma 3]{JLR}.
\par

Let $\As(\textbf{G},\textbf{G}')^*$ be the space of all triplets $(\phi,\phi',\phi'') \in \As(\textbf{G},\textbf{G}')$ 
such that 
\[ 
\la \lambda_1+\lambda_2+\lambda_3+\rho_{P}, \omega^{\vee} \ra \ne 0
\ \ \text{for all} \ \omega^{\vee} \in (\hat{\Delta}^{\vee})_{P'}, \  
\lambda_1 \in \mc{E}_{P}(\phi_{P}),\ \lambda_2 \in \mc{E}_{P'}(\phi_{P'}), \ \lambda_3 \in  \mc{E}_{P'}(\phi''_{P'}) 
\]
for all proper parabolic subgroups $P'$ of $G'$. 
\par

\begin{prop} \label{inv} 
Let $(\phi,\phi',\phi'') \in \As(\textbf{G},\textbf{G}')^*$. 
Then the following statements hold. 
\begin{enumerate}
\item 
$\sum_{P'} \mc{P}_{P'}^{T'}(\phi,\phi',\phi'')$ is independent of $T'$.
\item 
$\mc{P}(\phi,\phi',\phi'')=\sum_{P'} \mc{P}_{P'}^{T'}(\phi,\phi',\phi'')$.
It says $\As(\textbf{G},\textbf{G}')^*\subset \As_0(\textbf{G},\textbf{G}')$.

\item 
The regularized period $\mathcal{P}$ is a right $N(\A_F) \rtimes G'(\A_F)^1$-invariant linear functional on $\As(\textbf{G},\textbf{G}')^*$.
\item
If $g \to (\phi \otimes \phi')_{G,N}(g) \phi''(g)$ is integrable over $G'(F) \bs G'(\A_F)^1$, then
\[\mc{P}(\phi,\phi',\phi'')=\int_{G'(F) \bs G'(\A_F)^1} (\phi \otimes \phi')_{G,N}(g) \cdot \phi''(g) dg.
\]
\end{enumerate}
\end{prop}

\begin{proof} 
The proof of (i) is somewhat similar with \cite[Proposition 3.3]{IY1}. 
By the assumption, $\sum_{P'} \mc{P}_{P'}^{T'}(\phi,\phi',\phi'')$ is well-defined. 
Choose a regular $S'\in \aa_0'$. 
Using the Lemma \ref{mt} and the decomposition of automorphic forms (\ref{dec}), 
we can write $ \mc{P}^{T'+S'}_{P'}(\phi,\phi',\phi'')$ as 
\begin{align*}
&\int_{P'(F) \bs G'(\A_F)^1}^{\#}\sum_{\substack{ Q' \\ Q' \ss P'}} \sum_{\delta \in Q'(F) \bs P'(F)} \Lambda_m^{T',Q'}(\phi \otimes \phi')(\delta g) \phi_{P'}''(g)\Gamma_{Q'}^{P'}(H_{Q'}(\delta g)-T',S')\tau_{P'}(H_{P'}(g)-T'-S') dg 
\\&=\int_{Q'(F) \bs G'(\A_F)^1}^{\#}\sum_{\substack{ Q' \\ Q' \ss P'}} \Lambda_m^{T',Q'}(\phi \otimes \phi')( g) \phi_{P'}''(g)\Gamma_{Q'}^{P'}(H_{Q'}( g)-T',S')\tau_{P'}(H_{P'}(g)-T'-S') dg
\\&=\sum_{\substack{ Q' \\ Q' \ss P'}}   \int_{K'}\int_{(Q'\cap M')(F) \bs M'(\A_F)^1} \int_{\aa_{P'}}^{\#} \left(\Lambda_m^{T',Q'}(\phi \otimes \phi')\cdot \phi_{P'}''\right)(e^Xmk)\Gamma_{Q'}^{P'}(H_{Q'}(m)-T',S')
\\&\times
\tau_{P'}(X'-T'-S') e^{-2\la \rho_{P'},X \ra }dX dm dk.
\end{align*}
Using the Iwasawa decomposition of $M'$ with respect to $Q'\cap M'=U_{Q'}'M'_{Q'}$, write $m=ve^Ylk'$ where 
\[ 
v \in (U_{Q'} \cap M')(\A_F), \ Y \in \aa_{Q'}^{P'}, \ l \in M'_{Q'}(\A_F)^1, \ k' \in (K \cap M')(\A_F).
\]
Since $U'_{Q'}=U'(U'_{Q'} \cap M)$ and $\phi_P''(e^Xmk)=\phi_P''(ue^Xmk)$ for all $u\in U'(\A_F)$,
\[
\int_{[U'_{Q'} \cap M']}\phi_{P'}''(e^Xve^Ylk' k) dv
= \int_{[U']} \int_{[U'_{Q'} \cap M']}\phi_{P'}''(uv'e^{X+Y}lk' k) dv'du
=\phi_{Q'}''(e^{X+Y}lk'k).
\] 
Thus $ \mc{P}^{T'+S'}_{P'}(\phi,\phi',\phi'')$ equals the sum over $Q' \ss P'$ of
\begin{align*}
&\int_{K'}\int_{[M'_{Q'}]}  \int_{\aa_{Q'}^{P'}}^{\#} \int_{\aa_{P'}}^{\#}
\Lambda_m^{T',Q'}(\phi \otimes \phi')( e^{X+Y}lk) \phi_{Q'}''(e^{X+Y}lk) 
\Gamma_{Q'}^{P'}(Y-T',S')\tau_{Q'}(X-T'-S') e^{-2\la \rho_{Q'},X+Y \ra }dX dYdm dk
\\&=
\int_{K'}\int_{[M'_{Q'}]}  \int_{\aa_{Q'}}^{\#}\Lambda_m^{T',Q'}(\phi \otimes \phi')( e^{X}lk) \phi_{Q'}''(e^{X}lk) 
\Gamma_{Q'}^{P'}(X-T',S')\tau_P(X-T'-S') e^{-2\la \rho_{Q'},X\ra }dX dm dk.
\end{align*}
(We incorporated the double $\#$-integral into a single $\#$-integral over $\aa_{Q'}$ using \cite[(14)]{JLR}.)
\par

Thus if we change two sums over $P'$ and $Q'$ in $\sum_{P'} \mc{P}_{P'}^{T'+S'}(\phi,\phi',\phi'')$ and  use (\ref{ll2}), 
\begin{align*}
\sum_{P'} \mc{P}_{P'}^{T'+S'}(\phi,\phi',\phi'')
&=
\sum_{Q'}\int_{K'}\int_{[M'_{Q'}]}  \int_{\aa_{Q'}}^{\#}\Lambda_m^{T',Q'}(\phi \otimes \phi')( e^{X}lk) \phi_{Q'}''(e^{X}lk) \tau_{Q'}(X-T')e^{-2\la \rho_{Q'},X\ra }dX dm dk
\\&=
\sum_{Q'}\mathcal{P}_{Q'}^{T'}(\phi,\phi',\phi'')=\sum_{P'}\mathcal{P}_{P'}^{T'}(\phi,\phi',\phi'').
\end{align*}
This proves (i).
\par

Next we prove (ii). 
Applying the similar argument in the proof of Proposition \ref{poly} to $M'$ instead of $G'$, 
we see that $\mc{P}_{P'}^{T'}(\phi,\phi',\phi'')$ is also a polynomial exponential function in $T'$. 
Among all $ \mc{P}_{P'}^{T'}(\phi,\phi',\phi'')$,
the zero exponent appears only in the term $P' = G'$ by the assumption. 
Since $\sum_{P'}\mc{P}_{P'}^{T'}(\phi,\phi',\phi'')$ is independent of $T'$,  
other exponents are cancelled out except for the zero exponents. 
This proves (ii).
\par

Next we prove (iii). The linearity and the $N(\A_F)$-invariance is obvious from the definition of $\mathcal{P}$. Thus we are sufficient to check $G'(\A_F)^1$-invariance of $\mathcal{P}$. 
For any $x\in \textbf{G}'$ and a function $f$ on $\textbf{G}$ or $N(\A_F) \rtimes \textbf{G}'$ or $\textbf{G}'$, write $ f_x$ for the right translation of $f$ by $x$.

For $(\phi,\phi',\phi'')\in \As(\textbf{G},\textbf{G}')^*$, it is easy to see that 
\[(\phi_x \otimes \phi'_x)_{P,N}=\left((\phi \otimes \phi')_{P,N}\right)_x.
\]
The same argument just before \cite[Theorem 9]{JLR} shows that for arbitrary $\varphi \in \As_{P'}(\textbf{G}')$, \[\mc{E}_{P'}(\varphi)=\mc{E}_{P'}(\varphi_x).\]
Thus $\mathcal{E}_{P'}((\phi_x \otimes \phi'_x)_{P,N})=\mathcal{E}_{P'}\left((\phi \otimes \phi')_{P,N}\right)$ and so $(\phi_x,\phi'_x,\phi''_x)\in \As(\textbf{G},\textbf{G}')_{\chi}^*.$

By (ii), it is enough to show that for any $x \in G'(\A_F)^1$,
\beq \label{sinv} \mc{P}(\phi,\phi',\phi'')=\sum_{P'}\int_{P'(F) \bs G'(\A_F)^1}^{\#} \Lambda_m^{T',P'}(\phi_x \otimes \phi'_x)(g)(\phi_{x})_{P'}''(g)\tau_{P'}(H_{P'}(g)-T') dg.
\eeq
Fix $x \in G'(\A_F)^1$ and set $F_{P'}(g)=\Lambda_m^{T',P'}(\phi_{x^{-1}} \otimes \phi'_{x^{-1}})(gx)$. 
By applying the same argument in  \cite[page 193]{JLR}, we have
\begin{align*} &\int_{P'(F) \bs G'(\A_F)^1}^{\#} \Lambda_m^{T',P'}(\phi_x \otimes \phi'_x)(g)\phi_{x,P'}''(g)\tau_{P'}(H_{P'}(g)-T') dg
\\& \int_{P'(F) \bs G'(\A_F)^1}^{\#} F_{P'}(g)(\phi_{x^{-1}})_{P'}''(gx)\tau_{P'}(H_{P'}(gx)-T') dg=\int_{P'(F) \bs G'(\A_F)^1}^{\#} F_{P'}(g)\phi''_{P'}(g) \tau_{P'}(H_{P'}(gx)-T') dg
\end{align*} because $(\phi_{x^{-1}})_{P'}''(gx)=\phi''_{P'}(g)$. 

For $h \in G'(\A_F),$ let $K(h)\in K'$ be any element such that $hK(h)^{-1}\in P_0'(\A_F).$  Then using the same argument in \cite[page 193-194]{JLR}, 
we have
\[F_{P'}(g)=\underset{Q' \subset P'}{\sum_{Q'}}\sum_{\delta \in Q'(F) \bs P'(F)}\Lambda_m^{T',Q'}(\phi \otimes \phi')(\delta g) \Gamma_{Q'}^{P'}(H_{Q'}(\delta g)-T',-H_{Q'}(K(\delta g)x)).
\]
Applying the similar argument in the proof of (ii), we have
\begin{align*} & \int_{P'(F) \bs G'(\A_F)^1}^{\#}\sum_{\delta \in Q'(F) \bs P'(F)}\Lambda_m^{T',Q'}(\phi \otimes \phi')(\delta g) \phi''_{P'}(g)\Gamma_{Q'}^{P'}(H_{Q'}(\delta g)-T',-H_{Q'}(K(\delta g)x)) \tau_{P'}(H_{P'}(gx)-T') dg
\\& =\int_{Q'(F) \bs G'(\A_F)^1}^{\#} \Lambda_m^{T',Q'}(\phi \otimes \phi')( g)  \phi''_{Q'}(g)\Gamma_{Q'}^{P'}(H_{Q'}( g)-T',-H_{Q'}(K( g)x)) \tau_{P'}(H_{P'}(gx)-T') dg.
\end{align*}
Note that $H_{Q'}(gx)-H_{Q'}(g)=H_{Q'}(K(g)x)$. By putting $H=H_{Q'}(gx)-T', X=H_{Q'}(K(g)x)$ in (\ref{ll2}),
\begin{align*} &\underset{Q' \ss P' }{\sum_{P'}} \int_{Q'(F) \bs G'(\A_F)^1}^{\#} \Lambda_m^{T',Q'}(\phi \otimes \phi')( g)  \phi''_{Q'}(g)\Gamma_{Q'}^{P'}(H_{Q'}( g)-T',-H_{Q'}(K( g)x)) \tau_P(H_{P'}(gx)-T') dg
\\&=\int_{Q'(F) \bs G'(\A_F)^1}^{\#} \Lambda_m^{T',Q'}(\phi \otimes \phi')( g)  \phi''_{Q'}(g)\tau_{Q'}(H_{Q'}(g)-T') dg.
\end{align*}

By changing the order of integration and summation, we have
\begin{align*}
& \sum_{P'}\int_{P'(F) \bs G'(\A_F)^1}^{\#} \Lambda_m^{T',P'}(\phi_x \otimes \phi'_x)(g)(\phi_{x})_{P'}''(g)\tau_{P'}(H_{P'}(g)-T') dg
\\&= \sum_{P'}\int_{P'(F) \bs G'(\A_F)^1}^{\#} F_{P'}(g)\phi''_{P'}(g) \tau_{P'}(H_{P'}(gx)-T') dg
\\&=\sum_{Q'}  \underset{Q' \ss P' }{\sum_{P'}} \int_{Q'(F) \bs G'(\A_F)^1}^{\#} \Lambda_m^{T',Q'}(\phi \otimes \phi')( g)  \phi''_{Q'}(g)\Gamma_{Q'}^{P'}\left(H_{Q'}( g)-T',-H_{Q'}(K( g)x)\right) \tau_{P'}(H_{P'}(gx)-T') dg
\\&=\sum_{Q'} \int_{Q'(F) \bs G'(\A_F)^1}^{\#} \Lambda_m^{T',Q'}(\phi \otimes \phi')( g)  \phi''_{Q'}(g)\tau_{Q'}(H_{Q'}(g)-T') dg
\\&=\mc{P}(\phi,\phi',\phi'')
\end{align*}
and so we proved $(\ref{sinv})$ as required. (Changing the order of integration and summation in the third and fourth equality is justified by \cite[Lemma 2.1]{A3}. We refer the reader to \cite[page 194, 196]{JLR} for the detail.)

The last assertion (iv) essentially follows from the fact that \[\int_{G'(F) \bs G'(\A_F)^1}(\phi \otimes \phi')_{G,N}(g) \phi''(g) F^{G'}(g,T') dg-\int_{G'(F)\bs G'(\A_F)^1} \Lambda_m^{T'}(\phi \otimes \phi')(g) \phi''(g) dg \to 0\]
as $T' \to \infty$ in the positive Weyl chamber. (Here $F^{G'}(g,T')$ is the function defined just before Lemma \ref{kl}.) Since the proof is almost same with that of \cite[Proposition 3.8, Cor. 3.10]{IY}, we omit the detail. 
\end{proof}

\par 
 Let $\As(\textbf{G},\textbf{G}')^{**}$ be the subspace of all triplets 
$(\phi,\phi',\phi'') \in \As(\textbf{G},\textbf{G}')^*$ such that 
\[
\la \lambda_1+\lambda_2+\lambda_3+\rho_P, \omega^{\vee} \ra\ne 0
\quad \text{for all} \ 
\omega^{\vee} \in (\hat{\Delta}^{\vee})^{P'}_{Q'}, \ 
\lambda_1 \in \mc{E}_{Q}(\phi_{Q}),\ \lambda_2 \in \mc{E}_{Q'}(\phi_{Q'}), \ \lambda_3 \in  \mc{E}_{Q'}(\phi''_{Q'})
\] 
for all pairs of parabolic subgroups $Q'\subset P'$ of $G'$. 
Clearly $\As(\textbf{G},\textbf{G}')^{**} \subset \As(\textbf{G},\textbf{G}')^{*}$.
\par

If $(\phi,\phi',\phi'')\in \As(\textbf{G},\textbf{G}')^{**}$, then the regularized integral
\begin{align*}
&\int_{P'(F)\bs G'(\A_F)^1}^* (\phi \otimes \phi')_{P,N}(g)  \phi_{P'}''(g) \hat{\tau}_{P'}(H_{P'}(g)-T')  dg
\\&=\int_{K'} \int_{M'(F) \bs M'(\A_F)^1} \int_{\mf{a}_{P'}}^{\#} (\phi \otimes \phi')_{P,N}(e^Xmk) \phi_{P'}''( e^Xmk) \hat{\tau}_{P'}(X-T') e^{-2\la \rho_{P'},X\ra } dXdmdk
\end{align*}
is well-defined for all proper parabolic subgroups $P'$ of $G'$ by \cite[Lemma 3]{JLR}.

\begin{prop} \label{dec2} 
If $(\phi,\phi',\phi'') \in \As(\textbf{G},\textbf{G}')^{**}$, then
\begin{align*}
&\int_{G'(F) \bs G'(\A_F)^1} \Lambda_m^{T'}(\phi \otimes \phi')(g) \phi''(g)dg 
\\&=\sum_{P'} (-1)^{\dim \mathfrak{a}_{P'}}\int_{P'(F)\bs G'(\A_F)^1}^*(\phi \otimes \phi')_{P,N}(g)\cdot \phi''_{P'}(g)\hat{\tau}_{P'}(H_{P'}(g)-T')dg.
\end{align*}

\end{prop}

\begin{proof} 
We prove this by induction on the split rank of $G'$. 
Assume that it holds for groups $M_{P'}$, 
where $M_{P'}$ is the Levi subgroup of a proper parabolic subgroup $P'$ of $G'$. 
\par

By Proposition \ref{inv} (ii), we can write 
\[
\int_{G'(F) \bs G'(\A_F)^1} \Lambda_m^{T'}(\phi \otimes \phi')(g) \phi''(g)dg =
\mc{P}(\phi,\phi',\phi'') -\sum_{P' \subsetneq G'} \mc{P}_{P'}^{T'}(\phi,\phi',\phi'').
\]
Using the similar argument in \cite[Theorem 10]{JLR}, 
we can show that $\mc{P}_{P'}^{T'}(\phi,\phi',\phi'')$ is equal to
 \[ 
 \sum_{R' \subset P'} (-1)^{\dim \mathfrak{a}_{R'}^{P'}}\int_{R'(F)\bs G'(\A_F)^1}^* (\phi \otimes \phi')_{R,N}(g)\cdot \phi''_{R'}(g)\hat{\tau}_{R'}^{P'}(H_{R'}(g)-T')\tau_{P'}(H_{P'}(g)-T')dg
 \]
from our induction hypothesis. 
(See also \cite[Proposition 3.5]{Y2}.) 
\par

Thus
\begin{align*}
&\ \sum_{P' \subsetneq G'} \mc{P}_{P'}^{T'}(\phi,\phi',\phi'') \\
&= 
\sum_{R' \subset P' \subsetneq G'} (-1)^{\dim \mathfrak{a}_{R'}^{P'}}\int_{R'(F)\bs G'(\A_F)^1}^* (\phi \otimes \phi')_{R,N}(g)\cdot \phi''_{R'}(g)\hat{\tau}_{R'}^{P'}(H_{R'}(g)-T')\tau_{P'}(H_{P'}(g)-T')dg
\\&=\sum_{R' \subsetneq G'}- (-1)^{\dim \mathfrak{a}_{R'}}\int_{R'(F)\bs G'(\A_F)^1}^*(\phi \otimes \phi')_{R,N}(g)\cdot \phi''_{R'}(g)\hat{\tau}_{R'}(H_{R'}(g)-T')dg
\end{align*}
by Langlands' combinatorial lemma. 
Thus we proved  
\[ 
\mc{P}(\phi,\phi',\phi'') -\sum_{P' \subsetneq G'} \mc{P}_{P'}^{T'}(\phi,\phi',\phi'')
= \sum_{P'} (-1)^{\dim \mathfrak{a}_{P'}}\int_{P'(F)\bs G'(\A_F)^1}^*(\phi \otimes \phi')_{P,N}(g)\cdot  \phi''_{P'}(g)\hat{\tau}_{P'}(H_{P'}(g)-T')dg,
\] 
as desired.
\end{proof}

We conclude this section by providing the definitions of the \emph{regularized Bessel periods} and the \emph{regularized Fourier-Jacobi periods} for all classical and metaplectic groups.

Let $\nu_{\psi^{-1},\mu^{-1}}$ be a global generic Weil representation of $N_{n,r}(\A_F) \rtimes \textbf{G}'$ and $\Theta_{\psi^{-1},\mu^{-1},W_m}(\cdot )$ a theta series functional which gives a realization of $\nu_{\psi^{-1},\mu^{-1}}$ in $\As(N_{n,r} \rtimes \textbf{G}')$. Let $\chi$ be a generic automorphic character of $N_{n,r}(\A_F)$, that is stable under the conjugate action of $G'$ in $G$. We view $\chi$ as an automorphic character of $N_{n,r}(\A_F) \rtimes \textbf{G}'$, which is trivial on $\textbf{G}'$. 

\begin{defn} \label{rdef}Let $\vi_1 \in \As(\textbf{G}), \vi_2 \in \As(\textbf{G}')$ and $f \in \nu_{\psi^{-1},\mu^{-1},W_m}$. 
Suppose that $(\vi_1,\chi,\vi_2)$ and $(\vi_1,\Theta_{\psi^{-1},\mu^{-1},W_m}(f),\vi_2)$ belong to $\As(\textbf{G},\textbf{G}')^*.$

When $n-m=2r+1$, the \textbf{regularized Bessel period} for $(\phi_1,\phi_2)$ is defined by
\[\mathcal{B}_{\chi}(\vi_1,\vi_2) \col \mc{P}(\vi_1,\chi,\vi_2).
\]

When $n-m=2r$, the \textbf{regularized Fourier-Jacobi period} for $(\phi_1,\phi_2,f)$ is defined by 
\[
\mathcal{FJ}_{\psi,\mu}(\vi_1,\vi_2,f) \col \mc{P}(\vi_1,\Theta_{\psi^{-1},\mu^{-1},W_m}(f),\vi_2).
\]

\begin{rem}\label{well} When $N=\mathbf{1}$, Remark \ref{con} and Proposition \ref{inv} (ii) show that $\mathcal{P}$ is exactly the same regularized period functional defined in  \cite{IY} and \cite{Y} in each case. Proposition \ref{inv} (iv) shows that the above  definitions coincide with those of the Bessel period and Fourier-Jacobi period defined on cusp forms. Consequently, we can provide a more explicit formulation of \cite[Conjecture 9.1]{Gan1} using our regularized periods.
\end{rem}

\end{defn}

\section{\textbf{Jacquet modules corresponding to the Fourier-Jacobi characters}}
Starting this section, we prove one implication in the global GGP conjecture for skew hermitian spaces. In this section, we shall prove a lemma which plays a key role in proving Lemma \ref{l1}, which asserts the vanishing of the regularized Fourier-Jacobi periods involving residual representations. 

Throughout the rest of the paper, we assume that $E/F$ is a quadratic extension of number fields. Then for a place $v$ of $F$,
\[ E_v =\begin{cases}F_v \oplus F_v \ \ , \quad \quad \quad \quad \quad \quad \quad \quad  \text{when $v$ splits in $E$}\\
\text{quadratic extension of }F_v, \quad  \text{when $v$ non-splits in $E$}.
\end{cases}
\]

Let $\sigma \in \text{Gal}(E_v/F_v)$ be the non-trivial element. If $E_v=F_v \oplus F_v$, $\sigma$ is defined by \[\sigma(x,y)=(y,x).\]

Since we shall only consider the local situation, we suppress $v$ from the notation until the end of this section.

Let $\psi:F \to \CC^{\times}$ be an additive character and $\mu : E^{\times} \to \CC^{\times}$ a multiplicative character defined similarly as in the global case.
Let $W_{n}$ be a free left $E$-module of rank $n$ which has a skew-hermitian structure $(\cdot ,\cdot )$ and $G_n$ be its unitary group. 
Let $k$ be the dimension of a maximal totally isotropic subspace of $W_n$ 
and we assume $k>0$. 
We fix maximal totally isotropic subspaces $X$ and $X^*$ of $W_n$ which are in duality with respect to $(\cdot , \cdot)$. 
Fix a complete flag in $X$
\[
0=X_0 \subset X_1 \subset \cdots \subset X_k=X,
\]
and choose a basis $\{e_1,e_2,\cdots,e_j\}$ of $X_j$ 
such that $\{e_1,\cdots,e_j\}$ is a basis of $X_j$ for $1\le j \le k$.
Let $\{f_1,f_2,\cdots ,f_k\}$ be a basis of $X^*$ which is dual to the fixed basis $\{e_1,\cdots,e_k\}$ of $X$, 
i.e., $(e_i,f_j)=\delta_{ij}$ for $1\le i,j \le k$, 
where $\delta_{i,j}$ denotes the Kronecker delta. 
We write $X^*_j$ for the subspace of $X^*$ spanned by $\{f_1, f_2,\cdots ,f_j\}$ 
and $W_{n-2j}$ for the orthogonal complement of $X_j + X^*_j$ in $W_n$. 
\par

For each $1\le j \le k$, denote by $P_{n,j}$ the parabolic subgroup of $G_n$ stabilizing $X_j$, 
by $U_{n,j}$ its unipotent radical and by $M_{n,j}$ the Levi subgroup of $P_{n,j}$ stabilizing $X_j^*$. 
Then $M_{n,j} \simeq GL(X_j) \times G_{n-2j}$. 
(Here, we regard $GL(X_j)\simeq GL_j(E)$ 
as the subgroup of $M_{n,j}$ which acts as the identity map on $W_{n-2j}$.) Let $P_0$ be the minimal parabolic subgroup of $G_n$ which stabilizes the above complete flag of subspaces within $X$. We also fix a good maximal compact subgroup $K_n$ of $G_n$ with respect to $P_0$ (see \cite[I. 1.4]{Mo}).
\par

For $0\le j \le k $, write $N_{n,j}$ (resp. $\mathcal{N}_{j}$) 
for the unipotent radical of the parabolic subgroup of $G_n$ (resp., $GL(X_j)$) 
stabilizing the flag $\{0\}=X_0 \subset X_1 \subset \cdots \subset X_j$. 
If we regard $\mathcal{N}_{j}$ as a subgroup of $M_{n,j}\simeq GL(X_j) \times G_{n-2j}$, 
it acts on $U_{n,j}$ and so $N_{n,j}=U_{n,j} \rtimes \mathcal{N}_{j}$. When $j=0$, $\mathcal{N}_0$ denotes the trivial group.
\par

For any $0\le j<k$, $(\mathrm{Res}_{E/F}(W_{n-2j}), \mathrm{Tr}_{E/F}(\cdot ,\cdot ))$ is 
a nondegenerate symplectic $F$-space of $F$-dimension $2(n-2j)$. 
Let $\mathcal{H}_{W_{n-2j}}=\mathrm{Res}_{E/F}(W_{n-2j}) \oplus F$ 
be the Heisenberg group associated to $(\mathrm{Res}_{E/F}(W_{n-2j}), \mathrm{Tr}_{E/F}(\cdot ,\cdot ))$ 
and $\Omega_{\psi^{-1},\mu^{-1}W_{n-2j}}$ be the Weil representation of $ \mathcal{H}_{W_{n-2j}} \rtimes G_{n-2j}$ 
with respect to $\psi^{-1},\mu^{-1}$. 
Then since $N_{n,j-1}\backslash N_{n,j} \simeq \mathcal{H}_{W_{n-2j}}$ and $\mathcal{N}_j \ss N_{n,j-1}$, 
we can pull back  $\Omega_{\psi^{-1},\mu^{-1},W_{n-2j}}$ to $ N_{n,j}\rtimes G_{n-2j}$ 
and  denote it by the same symbol $\Omega_{\psi^{-1},\mu^{-1},W_{n-2j}}$. When $E$ is a field, $\mathrm{Tr}_{E/F}$ and $\mathrm{N}_{E/F}$ denote the usual trace and norm map, respectively. When $E=F \oplus F$, define \[\mathrm{Tr}_{E/F}(x,y)=x+y, \quad \quad \mathrm{N}_{E/F}(x,y)=xy   ,  \quad (x,y) \in E.\]
When $j \ge 2$, we define a character $\lambda_j:\mathcal{N}_j \to \CC^{\times}$ by
\[ 
\lambda_j(n)=\psi( \mathrm{Tr}_{E/F}((ne_2,f_1)+(ne_3,f_2)+ \cdots + (ne_{j},f_{j-1})) ), \quad n \in \mathcal{N}_{j}.
\] When $j=0,1$, $\lambda_j$ denotes the trivial character.
Denote $H_{n,j}=N_{n,j} \rtimes G_{n-2j}$. Using the projection $H_{n,j}$ to $\mathcal{N}_{j}$, we regard $\lambda_j$ as a character of $H_{n,j}$.

Put $\nu_{\psi^{-1},\mu^{-1},W_{n-2j}}=\lambda_j^{-1} \otimes \Omega_{\psi^{-1},\mu^{-1},W_{n-2j}}$. We can embed $H_{n,j}$ into $G_n \times G_{n-2j}$ by inclusion on the first factor and projection on the second factor.
Then $\nu_{\psi^{-1},\mu^{-1},W_{n-2j}}$ is a smooth representation of $H_{n,j}=N_{n,j} \rtimes G_{n-2j}$ 
and up to conjugation of the normalizer of $H_{n,j}$ in $G_n \times G_{n-2j}$, 
it is uniquely determined by $\psi$ modulo $\mathrm{N}_{E/F}(E)^{\times}$ and $\mu$. We shall denote by $\omega_{\psi^{-1},\mu^{-1},W_{n-2j}}$ the restriction of $\nu_{\psi^{-1},\mu^{-1},W_{n-2j}}$ to $G_{n-2j}$.
\par

For $0 \le l \le k-1$, we define a character $\psi_l$ of $N_{n,l+1}$, 
which factors through the quotient $n \colon N_{n,l+1} \to U_{n,l+1} \backslash N_{n,l+1}\simeq \mathcal{N}_{l+1}$, 
by setting
\[
\psi_l(u)=\lambda_{l+1}(n(u)).
\]
\par
Denote by $q$ the cardinality of the residue field of $F$ and write $q_E=q^2$.

For an irreducible smooth representation $\pi'$ of $G_n$, 
 write $J_{\psi_l}(\pi' \otimes \Omega_{\psi^{-1},\mu^{-1},W_{n-2l-2}})$ 
for the Jacquet module of $\pi' \otimes \Omega_{\psi^{-1},\mu^{-1},W_{n-2l-2}}$ 
with respect to the group $N_{n,l+1}$ and its character $\psi_l$. It is a representation of the unitary group $G_{n-2l-2}$. We can also write this as a two-step Jacquet module as follows (see \cite[Chapter 6]{gjr}.) 
\\ For $1\le s \le k$, 
let $C_{s-1} \simeq F$ be the subgroup of $N_{n,s}$ corresponding to the center of the Heisenberg group $\mathcal{H}_{W_{n-2s}}$ through the isomorphism
$ N_{n,s-1}\backslash N_{n,s} \simeq \mathcal{H}_{W_{n-2s}}$. Note that the center of $\mathcal{H}_{W_{n-2s}}$ acts on $\Omega_{\psi^{-1},\mu^{-1},W_{n-2s}}$ by the character $\psi^{-1}$. For $1 \le l \le k-1$, put $N_{n,l+1}^0=N_{n,l}\cdot C_{l}$ and define a character $\psi_{l}^0:N_{n,l+1}^0(F) \to \CC$ as $\psi_{l}^0(nc)=\psi_{l-1}(n)\psi(c)$ for $n \in N_{n,l}, c \in C_{l}$.

Let $J_{\psi_{l}^0}(\pi')$ be the Jacquet module of $\pi'$ with respect to $N_{n,l+1}^0$ and $\psi_{l}^0$. We also denote by $J_{\mc{H}_{W_{n-2(l+1)}}/C_l}$ the Jacquet functor with respect to $W_{n-2(l+1)}/C_l$ and the trivial character.

 Then we have 
\beq \label{two}
J_{\psi_{l}}(\pi' \otimes \Omega_{\psi^{-1},\mu^{-1},W_{n-2(l+1)}})\simeq J_{\mc{H}_{W_{n-2(l+1)}}/C_l}(J_{\psi_{l}^0}(\pi')\otimes \Omega_{\psi^{-1},\mu^{-1},W_{n-2(l+1)}}).
\eeq
We use the symbol $\text{Ind}$ to denote the normalized induction functor, $\text{ind}$ for the unnormalized induction functor and $\text{c-ind}$ for compactly supported induction.

\begin{lem} \label{Jac} 
Let $E$ be a field and $n,m,a$ positive integers such that $2r=n-m\ge 0$. Let $\mathcal{E}$ be a smooth representation of $G_{m+2a}$ of finite length and $\sigma,\pi$ and $\Sigma$ be irreducible smooth representations of $GL(X_a),G_n$ and $G_{n+2a}$, respectively. 
Then 

\begin{enumerate}
\item  If $r>0,$ \begin{align*}& \Hom_{N_{n+2a,r} \rtimes G_{m+2a}} 
\left( \Sigma \otimes  \nu_{\psi^{-1},\mu^{-1}, W_{m+2a}} 
\otimes  \mathcal{E},
\CC \right)
\\& \simeq \Hom_{G_{n+2a}} \left( 
J_{\psi_{r-1}}(\Sigma  \otimes \Omega_{\psi^{-1},\nu^{-1},W_{m+2a}})
 , \mathcal{E}^{\vee} \right) \end{align*}
 
 \item If $r=0$,
 \begin{align*}
& \dim_{\mathbb{C}}\Hom_{G_{m+2a}} 
\left(\Ind_{P_{m+2a,a}}^{G_{m+2a}}(\sigma |\cdot|_E^s \boxtimes \pi) \otimes  \nu_{\psi^{-1},\mu^{-1}, W_{m+2a}} 
\otimes \mathcal{E}   ,
\CC \right)
\\& \le \dim_{\mathbb{C}}\Hom_{G_m}
(J_{\psi_{a-1}}(  \mathcal{E} \otimes \Omega_{\psi^{-1},\mu^{-1}, W_m}),\pi^{\vee})
\end{align*}
\end{enumerate}
except for finitely many $q_E^{-s}$. 
\end{lem}

\begin{proof} \noindent 
We first prove (i). By the Frobenius reciprocity,
 \begin{align*}
 & \Hom_{N_{n+2a,r} \rtimes G_{m+2a}} 
\left( \Sigma \otimes  \nu_{\psi^{-1},\mu^{-1}, W_{m+2a}} 
\otimes  \mathcal{E} ,
\CC \right)
\\&=\Hom_{N_{n+2a,r} \rtimes G_{m+2a}} 
\left(\Sigma \otimes  \lambda_r^{-1} \otimes \Omega_{\psi^{-1},\mu^{-1}, W_{m+2a}} 
\otimes  \mathcal{E} ,
\CC \right)
\\& \simeq \Hom_{G_{m+2a}} \left( 
J_{\psi_{r-1}}( \Sigma\otimes \Omega_{\psi^{-1},\nu^{-1},W_{m+2a}})
 \otimes \mathcal{E},\CC \right)
 \\& \simeq \Hom_{G_{m+2a}} \left( 
J_{\psi_{r-1}}( \Sigma \otimes \Omega_{\psi^{-1},\nu^{-1},W_{m+2a}})
 , \mathcal{E}^{\vee} \right).
\end{align*}
Next, we prove (ii). The proof is quite similar with \cite[Lemma 4.1]{Y}.  
By the Frobenius reciprocity,
\[
\Hom_{G_{m+2a}} \left( 
 \Ind_{P_{m+2a,a}}^{G_{m+2a}}(\sigma |\cdot|_E^s \boxtimes \pi) \otimes \nu_{\psi^{-1},\mu^{-1}, W_{m+2a}}
\otimes \mathcal{E}  ,
\CC \right) 
\]
is isomorphic to
\beq \label{hom} 
\Hom_{P_{m+2a,a}}\big( (\sigma |\cdot|_E^{s-s_0} \boxtimes \pi) 
\otimes  \omega_{\psi^{-1},\mu^{-1}, W_{m+2a}}
\otimes \mathcal{E} ,
\CC \big)
\eeq
where $s_0$ is defined so that $\delta_{P_{m+2a,a}}^{\frac{1}{2}} = |\cdot|_E^{s_0}$.
For $1\le i \le a$, write $\mathcal{P}_{a,i}$ for the subgroup of $GL(X_a)$ 
which stabilizes the flag $X_{a-i} \subset X_{a-i+1}\subset \cdots \subset X_{a-1}$ 
and fixes $e_j$ modulo $X_{j-1}$ for $a-i+1 \le j \le a$. 
Put $Q_{a,i}=(\mathcal{P}_{a,i} \times G_{m}) \ltimes U_{m+2a,a}$. 
Using a similar argument to the proof of \cite[Theorem 16.1]{Gan2}, we have an exact sequence
\begin{equation}\label{ex}
\begin{CD}
  0 \to   \text{c-ind}_{Q_{a,1}}^{P_{m+2a,a}} \mu^{-1}|\cdot|_E^{\frac{1}{2}} \boxtimes \Omega_{\psi^{-1},\mu^{-1}, W_{m}} \to \omega_{\psi^{-1},\mu^{-1}, W_{m+2a}}|_{P_{m+2a,a}}  \to \mu^{-1}|\cdot|_E^{\frac{1}{2}} \boxtimes \omega_{\psi^{-1},\mu^{-1}, W_{m}} \to 0,
\end{CD}
\end{equation}
where $\text{c-ind}$ is the unnormalized compact induction functor. 
Since $\mathcal{E}$ is admissible, we have 
\begin{align}
\notag
&\Hom_{P_{m+2a,a}}\left(
(\sigma |\cdot|_E^{s-s_0} \boxtimes \pi) \otimes ( \mu^{-1}|\cdot|_E^{\frac{1}{2}} \boxtimes \omega_{\psi^{-1},\mu^{-1}, W_{m}} )
\otimes\mathcal{E}, \CC \right)
\\\label{h1} 
&\simeq \Hom_{M_{m+2a,a}}\left(
\left(\sigma \mu^{-1}|\cdot|_E^{\frac{1}{2}+s} \boxtimes( \pi \otimes \omega_{\psi^{-1},\mu^{-1}, W_{m}} )\right) \otimes J_{P_{m+2a,a}}(\mathcal{E})  ,
\CC \right),
\end{align}
where $J_{P_{m+2a,a}}$ is the Jacquet functor with respect to $U_{m+2a,a}$ and its trivial character. 
Since the central character of any irreducible subquotient of 
$J_{P_{m+2a,a}}(\mathcal{E})|_{GL(X_a)} \otimes \sigma \mu^{-1}|\cdot|_E^{\frac{1}{2}+s} $ 
is not trivial for almost all $q_E^{-s}$,  
the $\Hom$ space in (\ref{h1}) is zero for almost all $q_E^{-s}$. 
By tensoring the exact sequence (\ref{ex}) with 
$\mathcal{E} \otimes (\sigma |\cdot|_E^{s-s_0} \boxtimes \pi)$ 
and applying the $\Hom_{P_{m+2a,a}}(-,\CC)$ functor, 
we can regard (\ref{hom}) as a subspace of
\begin{align*}
&\Hom_{P_{m+2a,a}}\left(  (\sigma |\cdot|_E^{s-s_0} \boxtimes \pi)
\otimes (\text{c-ind}_{Q_{a,1}}^{P_{m+2a,a}} \mu^{-1}|\cdot|_E^{\frac{1}{2}} 
\boxtimes \Omega_{\psi^{-1},\mu^{-1}, W_{m}})\otimes \mathcal{E} ,
\CC \right)
\\&\simeq 
\Hom_{Q_{a,1}}\left(
 (\sigma |\cdot|_E^{s-s_0} \boxtimes \pi) \otimes (\delta_{P_{m+2a,a}} \delta_{Q_{a,1}}^{-1}\mu^{-1}|\cdot|_E^{\frac{1}{2}} 
\boxtimes \Omega_{\psi^{-1},\mu^{-1}, W_{m}}) \otimes \mathcal{E} ,
\CC \right)
\\&\simeq 
\Hom_{Q_{a,1}}\left( 
(\sigma \mu^{-1}|\cdot|_E^{\frac{1}{2}+s-s_1} 
\boxtimes (\pi \otimes \Omega_{\psi^{-1},\mu^{-1}, W_{m}}) ) \otimes \mathcal{E} ,
\CC \right)
\end{align*}
for almost all $s$,
where $s_1$ is defined so that $\delta_{P_{m+2a,a}} \delta_{Q_{a,1}}^{-1}|\cdot|_E^{-s_0} = |\cdot|_E^{-s_1}$. 
(Here we have used the Frobenius reciprocity law again.)
\par

Note that $\mathcal{P}_{a,1}$ is a mirabolic subgroup of $GL(X_a)$. 
By the Bernstein and Zelevinsky's result (See \cite[Sec. 3.5]{BZ}), 
the restriction of $\sigma$ to $\mathcal{P}_{a,1}$ has a filtration
\[
\{0\}=\sigma_{a+1}  \subset \sigma_a \subset \dots \subset \sigma_2 \subset \sigma_{1}=\sigma
\]
such that $\sigma_j / \sigma_{j+1} \simeq (\Phi^{+})^{j-1}\Psi^+\sigma^{(j)}$ for $1\le j \le a$. 
(For the definition of $\Phi^{+}$ and $\Psi^{+}$, see \cite[Sec. 3]{BZ}.) 
Let $X_{a,i}$ be the subspace of $X_a$ generated by $e_{a-i+1},e_{a-i+2},\dots,e_a$ 
and put $\mathcal{N}_{a,i}$ by $\mathcal{N}_a \cap GL(X_{a,i})$. 
Then for $1\le j \le a$, 
\[
(\Phi^{+})^{j-1}\Psi^+\sigma^{(j)} 
\simeq \text{c-ind}_{\mathcal{P}_{a,j}}^{\mathcal{P}_{a,1}}
|\cdot|_E^{\frac{j}{2}}\sigma^{(j)}\boxtimes \lambda_{a}|_{\mathcal{N}_{a,j}}.
\]
By applying the Frobenius reciprocity law, we have
\begin{align}
\notag
&\Hom_{Q_{a,1}}\left(
  \left( ( (\Phi^{+})^{j-1}\Psi^+\sigma^{(j)} \mu^{-1}|\cdot|_E^{\frac{1}{2}+s-s_1}) 
\boxtimes (\pi\otimes \Omega_{\psi^{-1},\mu^{-1}, W_{m}}) \right) \otimes \mathcal{E},
\CC \right)
\\\label{h4}
&\simeq 
\Hom_{GL(X_{a-j}) \times (\mathcal{N}_{a,j} \times G_m) \rtimes U_{m+2j,j}}\left(
(|\cdot|_E^{\frac{j+1}{2}+s-s_1'}\mu^{-1} \sigma^{(j)} 
\boxtimes \lambda_{a}|_{\mathcal{N}_{a,j}} \boxtimes (\pi
\otimes \Omega_{\psi^{-1},\mu^{-1}, W_{m}})) \otimes J_{P_{m+2a,a-j}}(\mathcal{E}) ,
\CC \right)
\end{align}
for some $s_1' \in \RR$ (depending on $j$).
For $1\le j \le a-1$, the central character of any irreducible subquotient of 
$(|\cdot|_E^{\frac{j+1}{2}+s-s_1'}\mu^{-1} \sigma^{(j)}) \otimes J_{P_{m+2a,a-j}}(\mathcal{E})|_{GL(X_{a-j})}$ 
is not trivial for almost all $q_E^{-s}$. 
Thus the Hom space in (\ref{h4}) is zero for almost all $q_E^{-s}$ unless $j=a$, 
in which case it is isomorphic to
\begin{align*}
&\Hom_{Q_{a,a}}( ( \psi_{a-1}\boxtimes \pi ) \otimes
 \Omega_{\psi^{-1},\mu^{-1}, W_{m}} 
\otimes \mathcal{E}, \CC 
)
\\&\simeq 
\left(\Hom_{G_m}(
J_{\psi_{a-1}^{-1}}(\mathcal{E}\otimes \Omega_{\psi^{-1},\mu^{-1}, W_{m}}),\pi^{\vee}
)\right)^{\oplus t},
\end{align*}
where $t=\dim_{\CC}\sigma^{(a)}.$ 
Since $t \leq 1$, our claim is proved.

\end{proof}

\begin{rem} Lemma \ref{Jac} can be used to control the size of the Hom space
\[ \Hom_{N_{n+2a,r} \rtimes G_{m+2a}} 
\left( \Ind_{P_{n+2a,a}}^{G_{n+2a}}(\sigma |\cdot|_E^s \boxtimes \pi) \otimes  \nu_{\psi^{-1},\mu^{-1}, W_{m+2a}} 
\otimes  \mathcal{E},
\CC \right)
\]
in terms of the relevant Jacquet modules. When $r>0$, the existence of $N_{n+2a,r}$ in this Hom space becomes essential since its absence would lead to an infinite-dimensional Hom space. Consequently, we need to handle the cases of $r>0$ and $r=0$ separately.

\end{rem}


\section{\textbf{Eisenstein series and residual representation}}
In this section, we introduce the concept of residual Eisenstein series representations. For an irreducible cuspidal automorphic representation $\pi$ of $G_{n}(\A_F)$ and $\sigma$ of $GL_a(\A_E)$, 
 write $L(s,\sigma\times \pi)$ for the Rankin-Selberg $L$-function $L(s,\sigma\times BC(\pi))$. 
We also write $L(s,\sigma,As^{+})$ for the Asai $L$-function of $\sigma$ 
and $L(s,\sigma,As^{-})$ for the $\mu$-twisted Asai $L$-function $L(s,\sigma \otimes \mu, As^{+})$ 
(cf., \cite[Section 7]{Gan2}).
\par

For $g\in G_{n+2a}(\A_F)$, 
write $g=m_gu_gk$ for $m_g\in M_{n+2a,a}(\A_F), \ u_g\in U_{n+2a,a}(\A_F), \ k_g\in K$. 
Since $M_{n+2a,a}\simeq GL(X_a) \times G_n$, 
we decompose $m_g=m_1 m_2$, 
where $m_1 \in GL(X_a)(\A_E) , m_2 \in G_n(\A_F).$ 
Define $d(g)=|\det{m_1}|_{\A_E}$. 
Then for $\phi \in \As_{P_{n+2a,a}}^{\sigma \boxtimes \pi}(G_{n+2a})$, the associated Eisenstein series with $\phi$ is defined by
\[
E(g,\phi,z)=\sum_{P_{n+2a,a}(F) \bs G_{n+2a}(F)}\phi(\gamma g)d(\gamma g)^z.
\] 
From the theory of Eisenstein series, it converges absolutely when $\Re(z)$ is sufficiently large and admits a meromorphic continuation to the whole complex plane. \cite[IV.1.8.]{Mo}. 

\begin{rem} \label{I_z}For $\phi \in \As_{P_{n+2a,a}}^{\sigma \boxtimes \pi}(G_{n+2a})$, we can uniquely extend it to a flat section $\phi_z \in \As_{P_{n+2a,a}}^{\sigma|\cdot|^z \boxtimes \pi}(G_{n+2a})$ which satisfying $\phi_z(k)=\phi(k)$ for $k \in K$. Then $E(g,\phi,z)=E(g,\phi_z,0)$ and the map $\phi_z \to E(g,\phi_z,0)$ gives an automorphic realization of the induced representation $\Ind_{P_{n+2a,a}(\A_F)}^{G_{n+2a}(\A_F)}(\sigma  |\cdot|^z_{E} \ \boxtimes \pi)$.\end{rem}
One of the properties of cuspidal representations with generic $A$-parameters (i.e., Remark \ref{ec} (iii)) plays an essential role in the proof of the  following proposition.

\begin{prop}(\cite[Proposition~5.2]{JZ}, \cite[Proposition 5.3]{IY}) \label{p1} 
Let $\pi$ be an irreducible cuspidal automorphic representation of $G_{n}(\A_F)$ with generic $A$-parameter
and $\sigma$ an irreducible unitary cuspidal automorphic representation of $GL_a(\A_E)$. 
For $\phi \in \As_{P_{n+2a,a}}^{\sigma \boxtimes \pi}(G_{n+2a})$, 
the Eisenstein series $E(\phi,z)$ has at most a simple pole at $z=\frac{1}{2}$ and $z=1$. 
Moreover, it has a pole at $z=\frac{1}{2}$ as $\phi$ varies 
if and only if 
$L(s,\sigma\times \pi^{\vee})$ is non-zero at $s=\frac{1}{2}$ 
and $L(s,\sigma,As^{(-1)^{n-1}})$ has a pole at $s=1$. 
Furthermore, it has a pole at $z=1$ as $\phi$ varies 
if and only if $L(s,\sigma\times \pi^{\vee})$ has a pole at $s=1$. 
\end{prop}

Thanks to Proposition \ref{p1}, we can define the residues of the Eisenstein series to be the limits 
\[ 
\mathcal{E}^{0}(\phi)=\lim_{z \to \frac{1}{2}}(z-\frac{1}{2})E(\phi,z), 
\quad 
\mathcal{E}^{1}(\phi)=\lim_{z \to 1}(z-1)E(\phi,z) \quad  \text{ for } \phi \in \As_{P_{n+2a,a}}^{\sigma \boxtimes \pi}(G_{n+2a}).
\] 
For $i=0,1$, 
let $\mathcal{E}^{i}(\sigma,\pi)$ be the residual representation of $G_{n+2a}(\A_F)$ 
generated by $\mathcal{E}^{i}(\phi)$ for $\phi \in \As_{P_{n+2a,a}}^{\sigma \boxtimes \pi}(G_{n+2a})$. Since the residual representation is isomorphic to the image of some relevant intertwining operator, the residue map $\mc{E}^i : \As_{P_{n+2a,a}}^{\sigma \boxtimes \pi}(G_{n+2a}) \to \mathcal{E}^{i}(\sigma,\pi)$ is $G_{n+2a}-$equivariant.
\par

\begin{rem}  \label{r1}  As we mentioned in Remark \ref{ec} (ii), we can write $BC(\pi)$ as an isobaric sum of the form $\sigma_1  \boxplus \cdots \boxplus \sigma_t $, 
where $\sigma_1,\cdots,\sigma_t$ are 
distinct irreducible unitary cuspidal automorphic representations of some general linear groups 
such that the (twisted) Asai $L$-function $L(s,\sigma_i,As^{(-1)^{n-1}})$ has a pole at $s=1$. Since $L(s,\sigma\times \pi^{\vee})=\Pi_{i=1}^t L(s,\sigma \times \sigma_i^{\vee})$, 
Proposition \ref{p1} implies that $\mathcal{E}^{1}(\sigma,\pi)$ is non-zero 
if and only if $\sigma \simeq \sigma_i$ for some $1 \le i \le t$. 
\end{rem}

\begin{rem}\label{r2}
Let $c \colon E \to E$ be the non-trivial element in $\text{Gal}(E/F)$ and use the same notation for the automorphism of $GL_n(E)$ induced by $c \colon E \to E$. For a representation $\sigma$ of $GL_n(\A_E)$, define $\sigma^{c}:=\sigma \circ c$. 
Note that $L(s,\sigma,As^{\pm})$ are nonzero at $s=1$ by \cite[Theorem 5.1]{Sha}. 
Thus if $L(s,\sigma,As^{(-1)^{n-1}})$ has a pole at $s=1$, the Rankin--Selberg $L$-function 
\[
L(s,\sigma \times \sigma^c)=L(s,\sigma,As^{+})\cdot L(s,\sigma,As^{-})
\]
has a simple pole at $s=1$ and so $\sigma^c \simeq \sigma^{\vee}$.
\end{rem}

\begin{rem} \label{r3} Let $\pi$ be an irreducible cuspidal automorphic representation of $G_{n}(\A_F)$ with generic $A$-parameter
and $\sigma$ an irreducible unitary cuspidal automorphic representation of $GL_a(\A_E)$. Let $\mc{E}$ be an automorphic representation of $G_{m+2a}$. Then for $\phi \in \Ind_{P_{n+2a,a}}^{G_{n+2a}}(\sigma\boxtimes \pi), f \in \nu_{\psi^{-1},\mu^{-1},W_{m+2a}}$ and $\vi \in \mc{E}$, $ \mathcal{P}(E(\phi,z),\Theta_{\psi^{-1},\mu^{-1},W_{m+2a}}(f),\vi)$ is well defined for all $z\in \CC$ but the poles of $E(\phi,z)$. By the meromorphy of the Eisenstein series $E(\phi,z)$ and Proposition \ref{inv} (ii), $\mathcal{P}(E(\phi,z),\Theta_{\psi^{-1},\mu^{-1},W_{m+2a}}(f),\vi)$ extends to a meromorphic function on the whole complex plane with possible poles contained in the set of poles of $E(\phi,z)$. 
\end{rem}

\section{\textbf{Reciprocal non-vanishing of the Fourier--Jacobi periods}}
n this section, we present the proof of Theorem \ref{rthm}, which asserts the reciprocal non-vanishing of the Fourier--Jacobi periods. This result is of significant importance in establishing our main theorem. The proof relies on several lemmas that we introduce and prove along the way.

Let $m,a$ be positive integers and $r$ non-negative integer. Write $n=m+2r$ and let $\left(W_{n+2a},(\ ,\ )\right)$ be a  skew-hermitian spaces over $E$ of dimension $n+2a$. Let $X,X^*$ be both $r$-dimensional isotropic subspaces of $W_{n+2a}$ which are in duality with respect to $(\cdot , \cdot )$ and $\left( W_{m+2a},(\ , \ )\right)$ be its orthogonal complement. Then one has the polar decomposition $W_{n+2a}=X \oplus W_{m+2a} \oplus X^*$.

Suppose that $\left( W_{m+2a},(\ , \ )\right)$ also has the polar decomposition $W_{m+2a}=Y_a \oplus W_m \oplus Y_a^{*},$ 
where $Y_a, Y_a^{*}$ are both $a$-dimensional isotropic subspaces of $W_{m+2a}$ which are in duality with respect to $(\cdot , \cdot )$.  
\par For $h \in GL(Y_a(\A_E)), \ b \in \text{Hom}(W_m,Y_a), \ c \in \text{Hom}(Y_a^*,Y_a)$, we define $h^* \in GL(Y_a^{*}(\A_E)), \ b^* \in \text{Hom}(Y_a^*,W_m), \ c^*\in \text{Hom}(Y_a,Y_a^*)$ by requiring that
\begin{align*} &(h^*x,y)=(x,hy)  ,  \quad  (b^*x,z)=(x,bz)  , \quad (c^*y,x)=(y,cx)\ \ \text{ for all } x\in Y_a^*(\A_E), \ y\in Y_a(\A_E), \ z\in W_m(\A_E).
\end{align*}
Let $\big(\rho_a,\mathcal{S}(Y_a^* (\A_E))\big)$ be the Heisenberg representation of $\mathcal{H}_{Y_a \oplus Y_a^*}(\A_F)$ with respect to $\psi^{-1}$. By restriction of scalars, we view $W_m$ as a symplectic space over $F$ with a symplectic form $\Tr_{E/F}\big( ( \ ,\ ) \big)$. Let $W_m=Y\oplus Y^*$ be a Witt decomposition of $W_m$ so that $Y,Y^*$ are maximal isotropic subspaces of $W_m$ which are in duality with respect to $\Tr_{E/F}\big( ( \ ,\ ) \big)$. Denote $\mathcal{S}(Y(\A_F))$ by $\mathcal{S}^0$ and let $(\rho^0,\mathcal{S}^0)$ be the Heisenberg representation of $\mathcal{H}_{W_m}(\A_F)$ with respect to $\psi^{-1}$. The action of $\mathcal{H}_{W_m}(\A_F)$ on $\mathcal{S}^0$ is given as follows: for $f \in \mathcal{S}^0$ and $y\in Y(\A_F)$,
\begin{itemize}
\item $(\rho^0(0,z)(f))(y)=\psi^{-1}(z)\cdot f(y)$, \quad \quad \quad \ \ for $z\in \A_F$
\item $(\rho^0(y_1,0)(f))(y)=f(y_1+y)$, \quad \quad \quad \quad \ \ \ for $y_1 \in Y(\A_F)$
\item $(\rho^0(y^*,0)(f))(y)=\psi^{-1}((y^*,y))\cdot f(y)$, \quad for $y^* \in Y^*(\A_F)$.
\end{itemize}

Then $\mathcal{S}=\mathcal{S}(Y_a^*(\A_E)) \otimes \mathcal{S}^0$ provides a mixed model for both the Heisenberg representation $\rho$ of $\mathcal{H}_{W_{m+2a}}(\A_F)$ and the global Weil representation $\nu_{\psi^{-1},\mu^{-1},W_{m+2a}}=\otimes \nu_{\psi_v^{-1},\mu_v^{-1},W_{m+2a}}$ of $H_{n+2a,r}(\A_F)$. (Here, we regard $\mathcal{S}$ as a space of functions on $Y_a^*$ whose values in $\mathcal{S}^0$.)

Using [\cite{Ku}, Theorem 3.1], we can describe the (partial) $N_{n,r}(\A_F) \rtimes P_{m+2a,a}(\A_F) -$action of the global Weil representation $\nu_{\psi^{-1},\mu^{-1},W_{m+2a}}$ on $\mathcal{S}$ as follows: (see also [\cite{gi}, page 35])

For $f \in \mathcal{S}$ and $y \in Y_a^*(\A_E)$,
\begin{align}
& \label{w1} \left(\nu_{\psi^{-1},\mu^{-1},W_{m+2a}}(n)f\right)(y)=\nu_{\psi^{-1},\mu^{-1},W_{m}}(n)\left( f(y) \right), \ \ \quad \quad \quad \quad \quad  \quad  \quad \ \ n \in N_{n,r}(\A_F),
\\& \label{w2} \left(\nu_{\psi^{-1},\mu^{-1},W_{m+2a}}(g_0)f\right)(y)=\nu_{\psi^{-1},\mu^{-1},W_{m}}(g_0)\left(f(y)\right), \ \ \ \quad \quad \quad \quad \quad \quad \ \ g_0\in U(W_m)(\A_F),
\\& \label{w3}  \left(\nu_{\psi^{-1},\mu^{-1},W_{m+2a}}(h_0)f\right)(y)=\mu^{-1}(h_0)|\det(h_0)|_{\A_E}^{\frac{1}{2}}\cdot f(h_0^*y), \ \ \quad \quad \quad \quad \ \ h_0\in GL(Y_a(\A_E)),
\\& \label{w4} \left(\nu_{\psi^{-1},\mu^{-1},W_{m+2a}}(b)f\right)(y)= \rho^0(b^*y,0)\left(f(y)\right), \ \quad \quad \quad \quad \quad \quad \quad \quad \quad \quad \ \ b\in  \text{Hom}(W_m,Y_a)(\A_E),
\\& \label{w5}  \left(\nu_{\psi^{-1},\mu^{-1},W_{m+2a}}(c)f\right)(y)=\psi^{-1}(\frac{\mathrm{Tr}_{E/F} (cy,y) }{2}) \cdot f(y), \ \quad \quad \quad \quad \quad  \quad\ \ \ c\in  \text{Herm}(Y_a^*,Y_a)(\A_E),
\\& \label{w6} \left( \rho(x+x',0)f \right)(y)=\psi^{-1}(\mathrm{Tr}_{E/F}(y,x)+\frac{\mathrm{Tr}_{E/F}(x',x)}{2}) \cdot f(y+x'), \quad \quad  x\in Y_a(\A_E),  x' \in Y_a^*(\A_E)
\end{align}
where \[ \text{Herm}(Y_a^*,Y_a)=\{c \in \text{Hom}(Y_a^*,Y_a) \mid c^*=-c\}
\]
and we identified $\text{Herm}(Y_a^*,Y_a)$ and $\Hom(W_m,Y_a)$ as subgroups of $U_{m+2a,a}$ via the canonical isomorphism $U_{m+2a,a}\simeq\Hom(W_m,Y_a) \ltimes \text{Herm}(Y_a^*,Y_a)$.

By choosing a theta functional $\theta_1:\mathcal{S}^0 \to \CC$, we define theta functions associated to $\mathcal{S}$ as follows:\par For $f \in \mathcal{S}(Y_a^*(\A_E)) \otimes \mathcal{S}^0$, 
 its associated theta function $\Theta_{\psi^{-1},\mu^{-1},W_{m+2a}}(\cdot ,f)$ is defined by
\[
\Theta_{\psi^{-1},\mu^{-1},W_{m+2a}}(h,f)
=\sum_{y \in Y_a(E)} \theta_{1}\left(\left(\nu_{\psi^{-1},\mu^{-1},W_{m+2a}}(h)f\right)(y)\right),
\] where $h=((u,n),g) \in \left( (U_{n+2a, r}\rtimes \mathcal{N}_r)\rtimes G_{m+2a}\right) (\A_F)=(N_{n+2a,r} \rtimes G_{m+2a})(\A_F)=H_{n+2a,r}(\A_F)$. \\ \noindent Then $\Theta_{\psi^{-1},\mu^{-1},W_{m+2a}}(f) \in \As(H_{n+2a,r})$ and 
the space of these theta functions $\{ \Theta_{\psi^{-1},\mu^{-1},W_{m+2a}}( f) \ | \ f\in \mathcal{S} \}$ 
is a realization of the global Weil representation $\nu_{\psi^{-1},\mu^{-1},W_{m+2a}}$ as an automorphic representation of $H_{n+2a,r}(\A_F)$. 
By one of main results of Weil \cite{We}, the global Weil representation has a unique (up to scailing) automorphic realization. Thus we can choose $\theta_1$ so that this theta function coincides with the one we defined in Section 1.
\par

Since we have fixed $\mu, \psi$, we simply write $\nu_{W_{m+2a}}$ for $\nu_{\psi^{-1},\mu^{-1},W_{m+2a}}$ 
and its associated theta function $\Theta_{\psi^{-1},\mu^{-1},W_{m+2a}}( f)$ as $\Theta_{W_{m+2a}}( f)$. Write $\Theta_{P_{m+2a,a}}( f)$ for the constant term of $\Theta_{W_{m+2a}}( f)$ along $P_{m+2a,a}$. 
\par

From now on,  $G_{k}$ denotes the isometry group of $W_{k}$ for various $k$.
\begin{lem}[cf. \cite{Y}{,  Lemma 6.2]}] \label{l0} For $h \in H_{n+2a,r}(\A_F)$, $f \in \mathcal{S}(Y_a^*(\A_E)) \otimes \mathcal{S}^0$, 
\[
\Theta_{P_{m+2a,a}}\left(h,f\right)
=\theta_1\left( \left(\nu_{W_{m+2a}}(h)f\right)(0)\right)
.\]

\end{lem}

\begin{proof}
Using the isomorphism $U_{m+2a,a}\simeq \Hom(W_m,Y_a) \ltimes \text{Herm}(Y_a^*,Y_a)$,  we take Haar measure $du=db \ dc$ on $U_{m+2a,a}$ for $u=bc$ with $b \in \Hom(W_m,Y_a), \ c \in \text{Herm}(Y_a^*,Y_a) $. 
For $y \in Y_a^*(\A_E)$, $\psi^{-1}(\frac{\mathrm{Tr}_{E/F} (cy,y) }{2})$ is the trivial character of $c \in \text{Herm}(Y_a^*,Y_a)$ if and only if $y=0$. Therefore by (\ref{w5}),
\[ \int_{\text{Herm}(Y_a^*,Y_a)}\Theta_{W_{m+2a}}((1,c)h,f) dc=\text{meas}(\text{Herm}(Y_a^*,Y_a)) \cdot \theta_1\left( \left(\nu_{W_{m+2a}}(h)f\right)(0)\right).
\] 
By ($\ref{w4}$), $\left(\nu_{W_{m+2a}}((1,b)h)(f)\right)(0)=\nu_{W_{m+2a}}(h)(f)(0)$ for all $b\in \Hom(W_m,Y_a)$ and so the claim follows.
\end{proof}

\begin{rem}\label{eval}Let $e$ be the identity element of $H_{n+2a,r}(\A_F) = N_{n+2a, r}(\A_F) \rtimes G_{m+2a}(\A_F)$. Observe that 
\[\Theta_{P_{m+2a,a}}\left(e,f\right)
=\theta_1\left( f(0)\right)
\]
and every Schwartz function in $\mathcal{S}^0$ can be obtained as evaluation at $0\in Y_a^*(\A_F)$ of some Schwartz function in $\mathcal{S}$. So we can regard theta functions in $\nu_{W_m}$ as the evaluation at $e$ of 
the constant terms of theta functions in $\nu_{W_{m+2a}}$ along $P_{m+2a,a}$. This observation will be used in the proof of Lemma \ref{l3}.
\end{rem}

\begin{rem}\label{ind}From (\ref{w3}),  we have
\[ 
\Theta_{P_{m+2a,a}}\left((u,hg),f\right)
=\mu^{-1}(h)|\det(h)|_{\A_E}^{\frac{1}{2}}\cdot \Theta_{P_{m+2a,a}}\left((u,g),f\right)
\] 
for $h \in GL(Y_a(\A_E))$. 
Thus from the the Lemma \ref{l0} and $(\ref{w1}), (\ref{w2})$, we see that the constant terms of theta functions $\{\Theta_{P_{m+2a,a}}(\cdot, f)\}_{f\in  \mathcal{S}}$ 
belong to the induced representation 
\[
\mathrm{ind}
_{GL_a(\A_E) \times (N_{n,r}(\A_F) \rtimes G_{m}(\A_F))}
^{N_{n+2a,r}(\A_F) \rtimes G_{m+2a}(\A_F)}
\mu^{-1} |\cdot|_{\A_E}^{\frac{1}{2}} \boxtimes \nu_{W_m}.
\] 
\end{rem}

The following vanishing lemma plays an essential role in proving Theorem \ref{rthm}.
\begin{lem}\label{l1}
Let $\sigma$ be an irreducible cuspidal automorphic representation of $GL_a(\A_E)$, 
$\pi_1,\pi_2$ irreducible cuspidal automorphic representations with generic $A$-parameters
of $G_n(\A_F)$ and  $G_m(\A_F)$, respectively. 
Write $BC(\pi_2)$ as an isobaric sum $\sigma_1 \boxplus \cdots \boxplus \sigma_t$, 
where $\sigma_1,\cdots,\sigma_t$ are distinct irreducible unitary cuspidal automorphic representations 
of the general linear groups such that 
the (twisted) Asai $L$-function $L(s,\sigma_i,As^{(-1)^{m-1}})$ has a pole at $s=1$.
Assume that $\sigma\simeq \sigma_i$ for some $1\le i \le t$. Then \[ \mathcal{P}(E(\phi,z), \Theta_{W_{m+2a}}(f),\vi)=0\] 
for all $ \phi\in \As_{P_{n+2a,a}}^{\mu \cdot \sigma^{\vee}\boxtimes\pi_1}(G_{n+2a}), \vi\in \mathcal{E}^1(\sigma,\pi_2)$ 
and $f\in \nu_{W_{m+2a}}$.
\end{lem}
\begin{proof}
Since  $G_{m+2a}(\A)=G_{m+2a}(\A)^1$ 
and $\{E(\phi,z) \mid \phi \in \As_{P_{n+2a,a}}^{\sigma \boxtimes \pi_1}(G_{n+2a})\}$ is an automorphic realization of $ \Ind_{P_{n+2a,a}}^{G_{n+2a}}(\sigma|\cdot|^{z} \boxtimes \pi_1)$, 
we can regard the functional $\mathcal{FJ}(E(\phi,z), \Theta_{W_{m+2a}}(f),\vi)$ 
as an element of 
\begin{equation}\label{ob1}
\Hom_{N_{n+2a,r}(\A) \rtimes G_{m+2a}(\A)} \left( 
 \Ind_{P_{n+2a,a}}^{G_{n+2a}}(\sigma|\cdot|^{z} \boxtimes \pi_1) \otimes  \nu_{W_{m+2a}}
\otimes  \mathcal{E}^1(\sigma,\pi_2), \CC 
\right)
\end{equation}\label{ob2}
by Proposition \ref{inv} (iii).

When $\sigma\simeq \sigma_i$, 
the residue $\mathcal{E}^1(\sigma,\pi_2)$ is non-zero by Remark \ref{r1}. 
We first prove the irreducibility of $\mathcal{E}^1(\sigma,\pi_2)$. 
Since the cuspidal support of the residues in $\mathcal{E}^1(\sigma,\pi_2)$ 
consists only of $\sigma |\cdot|^{-1} \boxtimes \pi_2$, 
the residues are square integrable by \cite[Lemma I.4.11]{Mo}. 
Thus $\mathcal{E}^1(\sigma,\pi_2)$ is a unitary quotient of 
$\Ind_{P_{n+2a,a}(\A)}^{G_{n+2a}(\A)}(\sigma |\cdot| \boxtimes \pi_2)$. 
Since the Langlands quotient of $\Ind_{P_{m+2a,a}(\A)}^{G_{m+2a}(\A)}(\sigma|\cdot| \boxtimes \pi_2)$ is its unique semisimple quotient, 
it is isomorphic to $\mathcal{E}^1(\sigma,\pi_2)$. 
\par
We fix a finite place $v$ of $F$ such that the local $v$-components of $\psi, \sigma, \pi_1, \pi_2,\mathcal{E}^1(\sigma,\pi_2)$ are all unramified. For the moment, we consider only the local situation and suppress the subscript $_v$ and drop the field $E_v,F_v$, when it is clear from the context.
Let $B_a$ and $B_k'$ denote the standard Borel subgroups of $GL_a$ and $G_k$, respectively. 
Recall that an irreducible generic unramified representation of $GL_a$ is an irreducible principal series representation.

\noindent By \cite[Theorem 1.5.2]{Ae}, we can write $\pi_2$ as the irreducible unramified constituent of 
\[
\Ind_{B_m'}^{G_m}
(\chi_{1} \boxtimes \cdots \boxtimes \chi_{m} ),
\] 
where $\chi_i$'s are unramified characters of $E^{\times}$ such that $\chi_i=\chi_i'\cdot |\cdot|^{s_i}$. (Here, $\chi_i'$'s are unitary character and $0\le s_i<\frac{1}{2}$.)

Since $\sigma$ is an isobaric summand of $BC(\pi_2)$, by Lemma \ref{r1}, we may assume \[\sigma \simeq \begin{cases}\Ind_{B_a}^{GL_a}(\chi_1  \boxtimes \cdots \boxtimes \chi_b \boxtimes \chi_b^{-1}  \boxtimes \cdots \boxtimes \chi_1^{-1}), \ \quad \quad \quad \quad \quad \quad \quad \quad \text{if} \ \ a=\text{even}, \ \omega_{\sigma}=\text{trivial}  \\
\Ind_{B_a}^{GL_a}(\chi_1  \boxtimes \cdots \boxtimes \chi_{b-1} \boxtimes 1 \boxtimes \chi_0 \boxtimes \chi_{b-1}^{-1}  \boxtimes \cdots \boxtimes \chi_1^{-1}), \quad \quad \quad \text{if} \ \ a=\text{even}, \ \omega_{\sigma}=\text{non-trivial}\\
 \Ind_{B_a}^{GL_a}(\chi_1  \boxtimes \cdots \boxtimes \chi_{b-1} \boxtimes 1 \boxtimes \chi_{b-1}^{-1} \boxtimes \cdots \boxtimes \chi_1^{-1}), \ \quad \quad \quad \quad  \quad   \text{if} \ \ a=\text{odd}, \ \omega_{\sigma}=\text{trivial}\\
  \Ind_{B_a}^{GL_a}(\chi_1  \boxtimes \cdots \boxtimes \chi_{b-1} \boxtimes \chi_0 \boxtimes \chi_{b-1}^{-1} \boxtimes \cdots \boxtimes \chi_1^{-1}), \ \quad \quad \quad \quad   \  \text{if} \ \ a=\text{odd}, \ \omega_{\sigma}=\text{non-trivial}
\end{cases}
\]
where  $\chi_0$ is the unique non-trivial unramified quadratic character of $E^{\times}$.

For a sequence $\mathbf{a}=(a_1,a_2,\ldots,a_t)$ of positive integers whose sum is equal or less than $k$, 
we denote by $\mathcal{P}_{\mathbf{a}}=\mathcal{M}_{\mathbf{a}}\mathcal{U}_{\mathbf{a}}$ (resp. $P_{\mathbf{a}}=M_{\mathbf{a}}U_{\mathbf{a}}$ )
the parabolic subgroup of $GL_k$ (resp. $G_k$) whose Levi subgroup $\mathcal{M}_{\mathbf{a}}$ is isomorphic to $GL_{a_1} \times \dots \times GL_{a_t} \times GL_{k-2\cdot \sum_{i=1}^ta_i}$ (resp. $GL_{a_1} \times \dots \times GL_{a_t} \times G_{k-2\cdot \sum_{i=1}^ta_i}$).
By conjugation of a certain Weyl element, it is not difficult to show that $\mathcal{E}^1(\sigma,\pi_2)$ is the irreducible unramified constituent of
\[ 
\Sigma:=\begin{cases}\Ind_{P_{(1,1)}}^{G_{m+2a}}
\left( \chi_1 |\cdot| \boxtimes \chi_1 |\cdot|^{-1}  \boxtimes \tau \right), \quad \quad \quad \text{if} \ \ a=\text{even}, \ \omega_{\sigma}=\text{trivial},  \\
\Ind_{P_{m+2a,1}}^{G_{m+2a}}
\left( \chi' |\cdot| \boxtimes \tau \right), \quad \quad \quad \quad \quad  \quad \quad  \text{ other cases,}
\end{cases}
\] 
where $\chi'$ is a quadratic character of $E^{\times}$ and $\tau$ is an irreducible smooth representation of $G_{m+2a-2}$.

Now we claim that \beq \label{ob3} \Hom_{N_{n+2a,r} \rtimes G_{m+2a}} \left( 
 \Ind_{P_{n+2a,a}}^{G_{n+2a}}(\sigma|\cdot|^{z} \boxtimes \pi_1) \otimes  \nu_{W_{m+a}}
\otimes  \mathcal{E}^1(\sigma,\pi_2), \CC 
\right) \eeq is zero for almost all $q^{-z}$. By Lemma \ref{Jac}, 
\[ (\ref{ob3}) \simeq \Hom_{G_{m+2a}} \left( 
J_{\psi_{r-1}}\left(\Ind_{P_{n+2a,a}}^{G_{n+2a}}(\sigma|\cdot|^{z} \boxtimes \pi_1) \otimes  \Omega_{W_{m+2a}} \right) ,{\mathcal{E}^1(\sigma,\pi_2)}^{\vee} \right).\]

Therefore, we will show  
\beq \label{ob4}
\Hom_{ G_{m+2a}} \left( J_{\psi_{r-1}}\left(\Ind_{P_{n+2a,a}}^{G_{n+2a}}(\sigma|\cdot|^{z} \boxtimes \pi_1) \otimes  \Omega_{W_{m+2a}} \right) 
, \Sigma^{\vee} 
\right)=0
\eeq
for almost all $q^{-z}$.

For a character $\chi$ of $F^{\times}$, the Bernstein-Zelevinski derivative of $\chi(k):=\chi (\det_{GL_k})$ is given by
\[ 
\left( \chi (k) \right)^{(t)}
=
\begin{cases} 
\chi(k-t)  &\text{ if } t \le 1, \\ 
0  & \text{ if } t \ge 2.
\end{cases}
\]

For a partition $(m_1,\cdots,m_t)$ of $k$ and characters $\mu_i$'s, let $\tau= \text{Ind}_{\mc{P}_{(m_1,\cdots,m_t)}}^{GL_k} \mu_1(m_1) \otimes \cdots \otimes \mu_t(m_t)$. Using the Leibniz rule of the Bernstein-Zelevinsky derivative, it is easy to see that for every non-negative integer $j$,
\beq \label{ee1}
 |\cdot|^{\frac{1-j}{2}}\cdot \tau^{(j)} \simeq \bigoplus_{  \substack{ i_1+\cdots + i_t=j \\ 0 \le i_1, \cdots, i_t \le 1}} \text{Ind}_{\mc{P}_{(m_1-i_1,\cdots,m_t-i_t)}}^{GL_{k-j}} |\cdot|^{\frac{1-i_1}{2}}\cdot \mu_1(m_1-i_1) \otimes \cdots \otimes|\cdot|^{\frac{1-i_t}{2}}\cdot  \mu_t(m_t-i_t)
\eeq 
up to semisimplication.

Using (\ref{two}), write $J_{\psi_{r-1}}\left(\Ind_{P_{n+2a,a}}^{G_{n+2a}}(\sigma|\cdot|^{z} \boxtimes \pi_1) \otimes  \Omega_{W_{m+2a}} \right)$ as

\beq \label{jac} J_{\mc{H}_{W_{m+2a}}/C_{r-1}}(J_{\psi_{r-1}^0}(\Ind_{P_{n+2a,a}}^{G_{n+2a}}(\sigma|\cdot|^{z} \boxtimes \pi_1) )\otimes \Omega_{W_{m+2a}}).\eeq
Now we apply \cite[Theorem 6.1]{gjr} with $j=a, \tau=\sigma$ and $l=r-1$.
Then by (\ref{jac}), \cite[Proposition 6.6]{gjr} and \cite[Proposition 6.7]{gjr}, 
\[
J_{\psi_{r-1}}\left(\Ind_{P_{n+2a,a}}^{G_{n+2a}}(\sigma|\cdot|^{z} \boxtimes \pi_1) \otimes  \Omega_{W_{m+2a}} \right)
\]

is decomposed as
\begin{equation} \label{e1} 
\bigoplus_{\substack{0\le t \le \min\{a,r-1\}  }}\mathrm{Ind}^{G_{m+2a}}_{P_{m+2a,a-t}} 
\left(
 \mu^{-1} |\cdot|^{\frac{1-t}{2}}\cdot (\sigma |\cdot|^z)^{(t)} \boxtimes J_{\psi_{r-1-t}}(\pi_1 \otimes \Omega_{W_{m+2t}})
\right)
\end{equation}

\begin{equation} \label{e2} \bigoplus \begin{cases} \mathrm{Ind}^{G_{m+2a}}_{P_{m+2a,a-r}}\left(\mu^{-1} \rvert \cdot \lvert^{\frac{1-r}{2}}(\sigma |\cdot|^z)^{(r)} \boxtimes (\pi_1 \otimes \omega_{W_{n}})\right),\quad \quad \ \ \text{ if } r\le a \\  0, \ \  \quad\quad\quad\quad\quad\quad\quad\quad\quad\quad\quad\quad\quad\quad\quad\quad\quad\quad \quad \quad \quad \quad \quad \text{ otherwise.} 
\end{cases}
\end{equation}
up to semisimplicaition. (Here, we reflected some typos in \cite[Proposition 6.7]{gjr}.)

Write $\pi_1=\Ind_{B_{n}'}^{G_{n}}
(\lambda_{1} \boxtimes \cdots \boxtimes \lambda_{[\frac{n}{2}]})$ for some unramified characters $\lambda_i$'s of $E^{\times}$ such that $\lambda_i=\lambda_i'\cdot |\cdot|^{s_i'}$. (Here, $\lambda_i'$ is a unitary character and $0\le s_i'<\frac{1}{2}$.)

Consider the map $T:\pi_1 \otimes \omega_{W_{n}} \to \text{Ind}_{B_{n}'}^{G_n}\big(\mu^{-1}|\cdot|^{\frac{1}{2}}\cdot (\lambda_{1} \boxtimes \cdots \boxtimes \lambda_{[\frac{n}{2}]})\big)$
 given by
\[T(f \otimes \phi)(h)=f(h) \otimes (\omega_{W_{n}}(h)\phi)(0).
\]

Then using arguments similar to \cite[Proposition 6.6]{gjr}, we can show that $T$ is well-defined and an isomorphism. Therefore, an irreducible unramified subquotient of $\mathrm{Ind}^{G_{m+2a}}_{P_{m+2a,a-r}}\left(\mu^{-1} \rvert \cdot \lvert^{\frac{1-r}{2}} \cdot(\sigma |\cdot|^z)^{(r)} \boxtimes (\pi_1 \otimes \omega_{W_{n}})\right)$ is that of 
 \begin{equation} \label{e3}\mathrm{Ind}^{G_{m+2a}}_{B'_{m+2a}}\left(\mu^{-1} ( |\cdot|^{z+\frac{1}{2}}\cdot \chi_{i_1}\boxtimes \cdots \boxtimes|\cdot|^{z+\frac{1}{2}}\cdot  \chi_{i_{a-r}}\boxtimes |\cdot|^{\frac{1}{2}} \cdot \lambda_{1} \boxtimes \cdots \boxtimes |\cdot|^{\frac{1}{2}}\cdot \lambda_{[\frac{n}{2}]} ) \right),\end{equation}
 where $\{i_1,\cdots,i_{a-r}\} \ss \{1,2,\cdots,a\}$.
However, \cite[Theorem 2.9]{BZ} implies that $\Sigma^{\vee}$ has no common composition factor with (\ref{e3}) for almost all $q^{-z}$. Therefore,
\[\Hom_{G_{m+2a}} \left( (\ref{e2})
, \Sigma^{\vee} 
\right)=0\]
for almost all $q^{-z}$ and so we are sufficient to show that \[\Hom_{G_{m+2a}} \left( (\ref{e1})
, \Sigma^{\vee} 
\right)=0.\]

For $0 \le i \le n-1$, put
\[
\pi_{1,i}\simeq \Ind_{B_{n-2i}'}^{G_{n-2i}}
(\lambda_{i+1} \boxtimes \cdots \boxtimes \lambda_{[\frac{n}{2}]} ).\]

We apply \cite[Theorem 6.1]{gjr} with $j=1,\tau=\lambda_1,l=r-1-t$. Then by \cite[Proposition 6.6]{gjr} and \cite[Proposition 6.7]{gjr},
\[
J_{\psi_{r-1-t}}\left(\pi_{1}\otimes  \Omega_{W_{m+2t}} \right)
\]
is decomposed as
\begin{equation} \label{e4} 
\bigoplus_{\substack{0\le t_1 \le 1 \\  -\frac{m}{2}-t<t_1 \le r-1-t }}\mathrm{Ind}^{G_{m+2t}}_{P_{m+2t,1-t_1}} 
\left(
  \mu^{-1} |\cdot|^{\frac{1-t_1}{2}}\cdot (\lambda_1)^{(t_1)} \boxtimes J_{\psi_{r-1-t-t_1}}(\pi_{1,1} \otimes \Omega_{W_{m+2(t+t_1-1)}})
\right)
\end{equation}
\begin{equation} \label{e5} \bigoplus \begin{cases} \pi_{1,1}\otimes \omega_{W_{m+2t}},\quad \ \ \text{ if } r-1-t=0 \\  0, \ \  \quad\quad\quad\quad\quad\quad \ \text{   otherwise.} 
\end{cases}
\end{equation}

When $r-1-t=0$, using the similar argument in the above, we have
\[\Hom_{G_{m+2a}} \left( \mathrm{Ind}^{G_{m+2a}}_{P_{m+2a,a-t}} 
\left(
 \mu^{-1}  |\cdot|^{\frac{1-t}{2}}\cdot (\sigma |\cdot|^z)^{(t)} \boxtimes  (\ref{e5})
, \Sigma^{\vee} \right)
\right)= 0\]
for almost all $q^{-z}$.

Therefore, we consider only when $r-1-t>0$ and it is enough to show that
\[\mathrm{Ind}^{G_{m+2a}}_{P_{(a-t,1-t_1)}} 
\left( \mu^{-1} \left( |\cdot|^{\frac{1-t}{2}}\cdot (\sigma |\cdot|^z)^{(t)} \boxtimes  |\cdot|^{\frac{1-t_1}{2}}\cdot (\lambda_1)^{(t_1)}\right) \boxtimes J_{\psi_{r-1-t-t_1}}(\pi_{1,1} \otimes \Omega_{W_{m+2(t+t_1-1)}}) \right)
\]
has no irreducible subquotient $\Sigma^{\vee}$ for almost all $q^{-z}$.

If we continue in this way, we are reduced to show that
\[\mathrm{Ind}^{G_{m+2a}}_{P_{(a-t,1-t_1,\cdots,1-t_k)}} 
\left(  \mu^{-1}\left( |\cdot|^{\frac{1-t}{2}}\cdot (\sigma |\cdot|^z)^{(t)} \boxtimes  |\cdot|^{\frac{1-t_1}{2}}\cdot (\lambda_1)^{(t_1)} \boxtimes \cdots \boxtimes  |\cdot|^{\frac{1-t_k}{2}}\cdot (\lambda_k)^{(t_k)}\right) \boxtimes J_{\psi_{0}}(\pi_{1,k} \otimes \Omega_{W_{m+2(r-1-k)}}) \right)
\]
has no irreducible subquotient $\Sigma^{\vee}$ for almost all $q^{-z}$.

To compute $J_{\psi_{0}}(\pi_{1,k} \otimes \Omega_{W_{m+2(r-1-k)}})$, write $\tau_k=\Ind_{B_{[\frac{n}{2}]-k-1}}^{GL_{[\frac{n}{2}]-k-1}}(\lambda_{k+1} \boxtimes \cdots \boxtimes \lambda_{[\frac{n}{2}]-1})$. We apply \cite[Theorem 6.1]{gjr} with $j=[\frac{n}{2}]-k-1, \tau=\tau_k$ and $l=0$. Then by \cite[Proposition 6.6]{gjr} and \cite[Proposition 6.7]{gjr}, $J_{\psi_{0}}(\pi_{1,k} \otimes \Omega_{W_{m+2(r-1-k)}})$ is decomposed as
\begin{equation} \label{e6} \begin{aligned}
&\Ind_{P_{(1,\cdots,1)}}^{G_{n-2k-2}}\left(\mu^{-1} |\cdot|^{\frac{1}{2}}\cdot (\lambda_{k+1} \boxtimes \cdots \boxtimes \lambda_{[\frac{n}{2}]-1})\boxtimes J_{\psi_{0}}(\Ind_{B_2'}^{G_2} \lambda_{[\frac{n}{2}]}  \otimes \Omega_{W_{0}})\right) \\
 &\bigoplus_{k+1\le s \le [\frac{n}{2}]-1} \Ind_{P_{(1,\cdots,1)}}^{G_{n-2k-2}}\left(\mu^{-1} |\cdot|^{\frac{1}{2}}\cdot( \lambda_{k+1} \boxtimes \cdots \boxtimes \lambda_{s-1} \boxtimes  \lambda_{s+1} \boxtimes \cdots \boxtimes \lambda_{[\frac{n}{2}]-1})\boxtimes (\Ind_{B_2'}^{G_2} \lambda_{[\frac{n}{2}]} \otimes \omega_{W_2})\right).
\end{aligned}
\end{equation}

Using arguments similar to \cite[Lemma 5.2]{gi1}, we can show that $\Ind_{B_2'}^{G_2} \lambda_{[\frac{n}{2}]}$ is irreducible. Further, it is $\psi$-generic by \cite[Theorem 1.1]{Mu}. Since $\Omega_{W_0}=\psi^{-1}$, we see that $J_{\psi_{0}}(\Ind_{B_2'}^{G_2} \lambda_{[\frac{n}{2}]}  \otimes \Omega_{W_{0}})=\textbf{1}$. Furthermore, as we have done previously, we see that \[ \Ind_{B_2'}^{G_2} \lambda_{[\frac{n}{2}]} \otimes \omega_{W_2}=\Ind_{B_2'}^{G_2}(\mu^{-1} |\cdot|^{\frac{1}{2}}\cdot \lambda_{[\frac{n}{2}]}).\] Therefore, \cite[Theorem 2.9]{BZ} implies that \[\mathrm{Ind}^{G_{m+2a}}_{P_{(a-t,1-t_1,\cdots,1-t_k)}} 
\left(  \mu^{-1} \left( |\cdot|^{\frac{1-t}{2}}\cdot (\sigma |\cdot|^z)^{(t)} \boxtimes  |\cdot|^{\frac{1-t_1}{2}}\cdot (\lambda_1)^{(t_1)} \boxtimes \cdots \boxtimes  |\cdot|^{\frac{1-t_k}{2}}\cdot (\lambda_k)^{(t_k)}\right) \boxtimes (\ref{e6}) \right)
\]
has no irreducible subquotient $\Sigma^{\vee}$ for almost all $q^{-z}$ and it proves our claim.

From what we have discussed so far, we see that
\[\mathcal{FJ}(E(\phi',z), \vi,f)=0\] at least when $\Re(z) \gg 0$. 
Since $z \mapsto \mathcal{FJ}(E(\phi',z), \vi,f)$ is meromorphic, $\mathcal{FJ}(E(\phi',z), \vi,f)$ is identically zero.
\end{proof}


From now on, we simply write $P$ (resp. $P_a,M_a, K_a$) for $P_{n+2a,a}$ (resp. $P_{m+2a,a},M_{m+2a,a}, K_{m+2a}$). Note that $N_{n+2a,r}\simeq  (\Hom(Y_a,X) \times \Hom(Y_a^*,X)) \rtimes N_{n,r}$ and $N_{n+2a,r} \cap P= \Hom(Y_a^*,X) \rtimes N_{n,r}$. Denote $\Hom(Y_a,X) \times \Hom(Y_a^*,X)$ by $V$ and let $dv$ be the Haar measure of $V$ such that $dn'=dvdn$, where $dn'$ (resp. $dn$) is the Haar measures of $N_{n+2a,r}$ (resp. $N_{n,r}$). The following remark is an easy observation which will be used in the proof of Lemma \ref{l2},
\begin{rem}\label{rJ}
For $x \in M_{m+2a,a}$, $xV x^{-1}=V$ and the Jacobian $\frac{d(xvx^{-1})}{dv}=1$. 
\end{rem}

\begin{lem}\label{l2} 
With the same notation as in Lemma \ref{l1}, assume $\sigma \simeq \sigma_i$ for some $1 \le i \le t$. 
If $ \phi\in \As_{P}^{\mu \cdot \sigma^{\vee}\boxtimes\pi_1}(G_{n+2a}), \vi\in \mathcal{E}^1(\sigma,\pi_2)$ 
and $f\in \nu_{W_{m+2a}}$, then
\begin{align*}
\notag
&\mathcal{FJ}_{\psi,\mu}(\mathcal{E}^0(\phi), \varphi,f)
=
\\
& \int_{K_a}\int_{V(\A_F)} \left( \int_{M_{a}(F)\bs M_{a}(\A_F)^1}  \left( \int_{[N_{n,r}]}\phi(nmvk) \cdot \Theta_{P_{a}}(f)\left((n,m)\cdot (v,k)\right) dn  \right) \varphi_{P_{a}}(mk) dm \right) \cdot e^{\frac{H_{P}(vk)}{2}}   dv dk.
\end{align*}
\end{lem}

\begin{proof}The proof is similar with \cite[Lemma 6.5]{Y}.
Using the residue theorem, write 
\[
\mathcal{E}^0(\phi)(g)
=\frac{1}{2\pi i} \int_{|z-\frac{1}{2}|=\epsilon}E(g,\phi,z)dz
\] 
for some small $\epsilon>0$. By employing arguments similar to those introduced by Arthur in \cite[pp. 47-48]{Af}, we can interchange the complex line integral with the mixed truncation operator. Then by Fubini's theorem, we have 
\begin{align}
&\int_{G_{m+2a}(F) \bs G_{m+2a}(\A_F)} 
\Lambda_m^T(\mathcal{E}^0(\phi) \otimes \Theta_{W_{m+2a}}(f))(g) \cdot \varphi(g)dg  \label{zc}
\\&=
\frac{1}{2\pi i }\cdot \int_{|z-\frac{1}{2}|=\epsilon} \int_{G_{m+2a}(F) \bs G_{m+2a}(\A_F)} 
\Lambda_m^T(E(\phi,z) \otimes \Theta_{W_{m+2a}}(f))(g) \cdot \varphi(g)dg dz  \nonumber
\\&=
\lim_{z\to \frac{1}{2}}\left(z-\frac{1}{2}\right) \cdot\int_{G_{m+2a}(F) \bs G_{m+2a}(\A_F)} 
\Lambda_m^T(E(\phi,z) \otimes \Theta_{W_{m+2a}}(f))(g) \cdot \varphi(g) dg.  \nonumber
\end{align}
Since $ \mathcal{FJ}_{\psi,\mu}(\mathcal{E}^0(\phi), \varphi,f)$ is the zero coefficient of the left-hand side, 
we compute the zero coefficient of the right-hand side. 
By Proposition \ref{dec2} and Lemma \ref{l1}, 
\begin{align*}
&\int_{G_{m+2a}(F) \bs G_{m+2a}(\A_F)} \Lambda_m^T\left( E(\phi,z) \otimes \Theta_{W_{m+a}}(f)\right)(g) \cdot \varphi(g) dg
\\&=
-\int_{P_{a}(F)\bs G_{m+2a}(\A_F)}^{*} 
\left(E(\phi,z) \otimes \Theta_{W_{m+2a}}(f)\right)_{P_{a},N_{n+2a,r}} \cdot  \varphi_{P_{a}}(g)\cdot \hat{\tau}_{P_{a}}(H_{P_{a}}(g)-T)dg.
\end{align*}
Note that $E_{P}(g,\phi,z)=\phi(g)d(g)^z+(M(z)\phi)(g)d(g)^{-z}$, where $M(z)$ is the relevant intertwining operator. Therefore,
\begin{align*}
&\int_{[N_{n+2a,r}]}E_{P}(\phi,z)(lg) \cdot \Theta_{P_{a}}(f)(l,g)  dl 
\\&= \int_{[N_{n+2a,r}]}\phi(lg) \cdot \Theta_{P_{a}}(f)(l,g)  \cdot d(lg)^z  dl + \int_{[N_{n+2a,r}]} (M(z)\phi)(lg) \cdot \Theta_{P_{a}}(f)(l,g) \cdot d(lg)^{-z} dl.
\end{align*}
Since the second term does not contribute to the zero coefficient in (\ref{zc}), we only consider the first term. 


Then by Lemma \ref{l0} and Remark \ref{rJ},
\begin{align*}
& -\int_{P_{a}(F)\bs G_{m+2a}(\A_F)}^{*}\left( \int_{[N_{n+2a,r}]}\phi(lg) \cdot \Theta_{P_{a}}(f)(l,g)  \cdot d(lg)^z  dl \right)\cdot \varphi_{P_{a}}(g) \cdot \hat{\tau}_{P_{a}}(H_{P_a}(g)-T)dg
\\& =-\int_{K_a} \int_{M_{a}(F)\bs M_{a}(\A_F)^1} \int_{T}^{\infty} \int_{[V]}\int_{[ N_{n,r}]}\phi(nvmk) \cdot \Theta_{P_{a}}(f)(nv,mk)   e^{z(X+H_{P}(vmk))}    \varphi_{P_{a}}(mk)   \cdot e^{-\frac{X}{2}} dn dv dX dm dk 
\\& =-\int_{T}^{\infty}e^{(z-\frac{1}{2})X}dX \cdot \int_{K_a}  \int_{M_{a}(F)\bs M_{a}(\A_F)^1}  \int_{[V]} \int_{[N_{n,r}]}\phi(nmvk) \cdot \Theta_{P_{a}}(f)\left((n,m)(v,k)\right)  \varphi_{P_{a}}(mk)   e^{zH_{P}(vk)}   dn dv dm dk.
\end{align*}
Since 
\[
-\int_{T}^{\infty}e^{(z-\frac{1}{2})X}dX =\frac{e^{(z-\frac{1}{2})T}}{z-\frac{1}{2}},
\] 
the zero coefficient of 
\[ 
\lim_{z\to \frac{1}{2}}\left(z-\frac{1}{2}\right) 
\cdot\int_{G_{m+2a}(F) \bs G_{m+2a}(\A_F)} \Lambda_m^T(E(g,\phi,z) \otimes \Theta_{W_{m+2a}}(f))(g) \cdot  \varphi(g) dg 
\] 
is equal to
\begin{align*}
& \lim_{z\to \frac{1}{2}}\left(z-\frac{1}{2}\right) \cdot \mathcal{P}(E(\phi,z) , \Theta_{W_{m+2a}}(f), \varphi)  \quad +
\\&  \int_{K_a}\int_{[V]} \left( \int_{M_{a}(F)\bs M_{a}(\A_F)^1}  \left( \int_{[N_{n,r}]}\phi(nmvk) \cdot \Theta_{P_{a}}(f)\left((n,m)\cdot(v,k)\right) dn  \right) \varphi_{P_{a}}(mk) dm \right) \cdot e^{\frac{H_{P}(vk)}{2}}   dv dk.
\end{align*}
Since $\mathcal{P}(E(\phi,z) , \Theta_{W_{m+2a}}(f), \varphi)=0$ by Lemma \ref{l1}, it completes the proof.
\end{proof}

\begin{lem}\label{l3}
With the same notation as in Lemma \ref{l1}, 
we assume $\sigma \simeq \sigma_i$ for some $1 \le i \le t$. Then the following conditions are equivalent.
\begin{enumerate}
\item
The Fourier--Jacobi period functional $\mathcal{FJ}_{\psi,\mu}(\pi_1, \pi_2,\nu_{W_m})$ is nonzero 

\item There exist $ \phi\in \As_{P_a}^{\mu \cdot \sigma^{\vee}\boxtimes\pi_1}(G_{m+2a}), \vi\in \mathcal{E}^1(\sigma,\pi_2)$ and $f\in \nu_{W_{m+2a}}$ such that 
\begin{equation} \label{rec}
\int_{K_a}\int_{[V]} \left( \int_{M_{a}(F)\bs M_{a}(\A_F)^1}  \left( \int_{[N_{n,r}]}\phi(nmvk) \cdot \Theta_{P_{a}}(f)\left((n,m)\cdot (v,k)\right) dn  \right) \varphi_{P_{a}}(mk) dm \right) \cdot e^{\frac{H_{P}(vk)}{2}}dvdk  \ne 0
\end{equation}
\end{enumerate}
\end{lem}

\bigbreak

\begin{rem}\label{ir}The above lemma seems similar to \cite[Proposition 5.3]{GJR}, which appears in the proof of one direction of the global GGP conjecture for the metaplectic-symplectic case. Experts discovered a technical gap in the proof of \cite[Proposition 5.3]{GJR} when the co-rank is not zero. To be more precise, they mistakenly assumed that the image of a certain residual intertwining operator would be the full induced representation, and consequently, they chose a vector in the image of the operator with good support. However, since the image of such an operator is not the full induced representation, it becomes highly non-trivial to choose a good vector within the operator's image. To circumvent such difficulty, we employed the isobaric sum of $BC(\pi_2)$ instead of $BC(\pi_1)$ in Lemma \ref{l1}. This approach has an advantage and a disadvantage. The advantage is that it eliminates the need to choose a special vector within the image of the residual intertwining operator in the proof of Lemma \ref{l3}. The disadvantage is that, compared to the approach using the isobaric sum of $BC(\pi_1)$, the proof of Lemma \ref{l1} becomes more complex because one must explicitly compute the relevant twisted Jacquet-modules corresponding to Fourier-Jacobi characters, even for non-first occurrences. Having proven Lemma \ref{l1}, we can then demonstrate the following lemma with reasonable ease by utilizing the inductive structure of the constant terms of theta series.
\end{rem}


\begin{proof} 
The proof of the (i) $\to$ (ii) direction essentially follows the reasoning presented in the proof of \cite[Lemma 6.5]{Y}. However, our situation is slightly more complex due to the inclusion of an integration over $V$. As a result, a more meticulous treatment becomes necessary to effectively manage this additional integral.

\noindent Put 
\[ 
\Pi=\mu  |\cdot|_E^{\frac{1}{2}}\cdot \sigma^{\vee} \boxtimes \pi_1, 
\quad
\Pi'=\sigma |\cdot|_E^{-1} \boxtimes \pi_2, 
\quad
\Pi''=\mu^{-1} |\cdot|_E^{\frac{1}{2}} \boxtimes \nu_{\psi^{-1},\mu^{-1},W_m}.
\] 
We define a functional on $\Pi \boxtimes \Pi' \boxtimes \Pi''$ by 
\[
l(\eta \boxtimes \eta' \boxtimes \eta'')
=\int_{M_{a}(F)\bs M_{a}(\A_F)^1}\left(\int_{[N_{n,r}]}\eta(nm)\eta''\left((n,m)\right)dn\right) \eta'(m) dm.
\]
\par

Let $(\Pi \boxtimes \Pi' \boxtimes \Pi'')^{\infty}$ be 
the canonical Casselman-Wallach globalization of $\Pi \boxtimes \Pi' \boxtimes \Pi''$ 
realized in the space of smooth automorphic forms 
without the $K_{M_{n+2a}} \times K_{M_{m+2a}} \times K_{M_{m+2a}}$-finiteness condition, 
where $K_{M_{i+2a}}=K_{i+2a} \cap M_{i+2a,a}$ for $i=n,m$ 
(cf. \cite{Ca}, \cite[Chapter 11]{Wa}). 
Since cusp forms are bounded, 
$l$ can be uniquely extended to a continuous functional on $(\Pi \boxtimes \Pi' \boxtimes \Pi'')^{\infty}$ 
and denote it by the same notation. 
Our assumption enables us to choose $\eta \in \Pi, \eta^{\prime} \in \Pi^{\prime}$ and $\eta^{\prime \prime} \in \Pi''$ 
so that $l(\eta \boxtimes \eta' \boxtimes \eta'')\ne 0$. 
We may assume that $\eta, \eta'$ and $\eta''$ are pure tensors. 
By \cite{LS,BS}, the functional $l$ is a product of local functionals 
$l_v \in \mathrm{Hom}_{M_{a,v}}
((\Pi_{v} \boxtimes \Pi^{\prime}_v \boxtimes \Pi^{\prime \prime}_v)^{\infty},\mathbb{C})$, 
where we set 
$(\Pi_v \boxtimes \Pi_v^{\prime} \boxtimes \Pi_v^{\prime \prime})^{\infty}
=\Pi_v \boxtimes \Pi_v^{\prime} \boxtimes \Pi_v^{\prime \prime}$ 
if $v$ is finite. 
Then we have $l_v(\eta_v \boxtimes \eta_v' \boxtimes \eta_v'')\ne 0$.
\par

Denote by $1$ (resp. $e$) the identity element of $G_{n+2a}$ (resp. $G_{m+2a}$).  Choose $\varphi' \in \mathcal{E}^1(\sigma,\pi_2)$ and $\ f \in \nu_{W_{m+2a}}$ such that 
\begin{enumerate}
\item 
$\delta_{P_{a}}^{-\frac{1}{2}}\cdot \varphi'_{P_{a}}=\boxtimes_v \varphi'_v, \ 
\Theta_{P_{a}}^{\psi}(f)=\boxtimes_v f_v$;
\item 
$\varphi'_v(e)=\eta_v$, $f_v((1,e))=\eta_v''$.
\end{enumerate}
(We can choose such $f$ by Remark \ref{eval} and Remark \ref{ind}.)  \par
\noindent Let $U_{n+2a,a}^{-}$ be the unipotent radical of the opposite parabolic subgroup $P^-$ of  $P$ and $\alpha_v$ be a smooth function on $U_{n+2a,a,v}^{-}$, which has a compact support. We can define a section $\phi_v$ by requiring 
\[
\phi_v(muu_{-})
=\delta_{P_{v}}(m)^{\frac{1}{2}}\alpha{_v} (u_{-})\cdot(\Pi_v)^{\infty}(m)\eta_v, 
\quad m\in M_{n+2a,a,v},  u \in U_{n+2a,a,v}, u_{-} \in U_{n+2a,a,v}^{-}.
\] 


Since $N_{n+2a,r}\simeq (\Hom(Y_a,X) \times \Hom(Y_a^*,X)) \rtimes N_{n,r}$, we regard elements of $\Hom(Y_a,X)$ and $\Hom(Y_a^*,X)$ as elements of  $N_{n+2a,r}$ via the natural inclusion maps. By Remark \ref{eval} and (\ref{w6}), for every $p_1 \in \Hom(Y_a,X)_v$, \[f_v((p_1,e))=
f_v((1,e))=\eta_v''.\] Using the Iwasawa decomposition of $G_{n+2a,v}$ with respect to $P_{v}$, it is easy to check that $\Hom(Y_a,X)_v$ is contained in $U_{n+2a,a,v}^{-}$. By taking the support of $\alpha_v$ sufficiently small near $e$, we have
\[
\int_{\Hom(Y_a,X)_v}l_v
\left(  \phi_v(p_1) \boxtimes  \eta'_v
  \boxtimes f_v\left((p_1,e)
\right) 
\right) \cdot e^{\frac{H_{P}(p_1)}{2}}  dp_1= \int_{\Hom(Y_a,X)_v}l_v
\left(  \phi_v(p_1) \boxtimes  \eta'_v   \boxtimes \eta_v''
\right) \cdot e^{\frac{H_{P}(p_1)}{2}}   dp_1\ne 0.
\]
On the other hand, there is a small neighborhood $N_v$ of $0$ in $\Hom(Y_a^*,X)_v$ such that 
\[\int_{\Hom(Y_a^*,X)_v}\int_{\Hom(Y_a,X)_v}l_v\left( \phi_v(p_1p_2) \boxtimes  \eta'_v
  \boxtimes f_v\left((p_1p_2,e)
\right)   \right) \cdot e^{\frac{H_{P}(p_1p_2)}{2}}   \cdot \chi_{N_v}(p_2)dp_1 dp_2 \ne 0.
\]
(Here, $\chi_{N_v}$ is the characteristic function on $N_v$.)
\par

By taking the support of the Schwartz function $f$ sufficiently small neighborhood of $0 \in Y_a^*$, 
we may assume that the support of $f_{v}$ is contained in $N_v$. 

Then 
\begin{align}
\label{ii} 
&\int_{\Hom(Y_a^*,X)_v}\int_{\Hom(Y_a,X)_v}l_v\left(\phi_v(p_1p_2) \boxtimes  \eta'_v
  \boxtimes f_v(\left(p_1p_2,e)\right)   \right) \cdot e^{\frac{H_{P}(p_1p_2)}{2}}  dp_1 dp_2
\\\notag
&=\int_{\Hom(Y_a^*,X)_v}\int_{\Hom(Y_a,X)_v}l_v\left( \phi_v(p_1p_2) \boxtimes  \eta'_v
  \boxtimes f_v\left((p_1p_2,e)\right)  \right)  \cdot e^{\frac{H_{P}(p_1p_2)}{2}}   \cdot \chi_{N_v}(p_2)dp_1 dp_2,
\end{align}
and so (\ref{ii}) is nonzero.
 \par
 
Put 
\[
I(\phi_v)=\int_{K_{a,v}} \int_{\Hom(Y_a^*,X)_v}\int_{\Hom(Y_a,X)_v}l_v
\left( \phi_v(p_1p_2k) \boxtimes  \varphi_v(k)
  \boxtimes f_v\left((p_1p_2,k)
\right) \right) \cdot e^{\frac{H_{P}(p_1p_2k)}{2}}   
dp_1 dp_2dk.
\] 
Since $P_{a,v}\cdot U_{a,v}^{-}$ is an open dense subset of $G_{m+2a,v}$, 
we can rewrite the local integral as 
\[
I(\phi_v)
=\int_{U_{a,v}^{-}} \int_{\Hom(Y_a^*,X)_v}\int_{\Hom(Y_a,X)_v}
\alpha_v(u_{-}) l_v
\left( \phi_v(p_1p_2u_{-}) \boxtimes  \eta_v'
  \boxtimes f_v\left((p_1p_2,u_{-})
\right)\right) \cdot e^{\frac{H_{P}(p_1p_2u_{-})}{2}}   
dp_1 dp_2 du_{-}.
\]
We can choose $\alpha_v$ to be supported in a small neighborhood of $e$ so that $I(\phi_v) \ne 0$. 
Since $I$ is continuous and $K_{n+2a,v}$-finite vectors are dense in the induced representation, 
we can choose $K_{n+2a,v}$-finite function $\phi_v$ such that $I(\phi_v)\ne 0$. This completes the proof of the (i) $\to$ (ii) direction. 

The proof of (ii) $\to$ (i) direction is almost immediate. From what we have seen in the above, we see that the Fourier--Jacobi period integral $\mathcal{FJ}_{\psi,\mu}(\pi_1, \pi_2,\nu_{W_m})$ is just the inner period integral in (\ref{rec}). Therefore, if $\mathcal{FJ}_{\psi,\mu}(\pi_1,\pi_2,\nu_{W_m})=0$, the integral (\ref{rec}) is always zero.
\end{proof}

By combining Lemma \ref{l2} with Lemma \ref{l3}, we get the following reciprocal non-vanishing theorem.

\begin{thm}\label{rthm} With the same notation as in Lemma \ref{l1}, assume $\sigma \simeq \sigma_i$ for some $1 \le i \le t$. 
Then $\mathcal{FJ}_{\psi,\mu}(\pi_1,\pi_2,\nu_{W_m}) \ne 0$ is equivalent to $\mathcal{FJ}_{\psi,\mu}(\mathcal{E}^0(\mu \cdot \sigma^{\vee},\pi_1), \mathcal{E}^1(\sigma,\pi_2),\nu_{W_{m+2a}}) \ne 0$.

\end{thm}
\bigbreak


\section{\textbf{Proof of the Main Theorem}}

\begin{proof}
We know that $BC(\pi_2)$ is an isobaric sum of the form $\sigma_1 \boxplus \cdots \boxplus \sigma_t$, 
where $\sigma_1,\dots,\sigma_t$ are distinct irreducible unitary cuspidal automorphic representations of general linear groups 
such that the (twisted) Asai $L$-function $L(s,\sigma_i,As^{(-1)^{m-1}})$ has a pole at $s=1$. 
Then for each $1 \le i \le t$,  $L(s,\mu^{-1}\cdot  \sigma_i,As^{(-1)^{m}})$ has a pole at $s=1$. 
On the other hand, $\mathcal{E}^0(\mu\cdot \sigma_i^{\vee},\pi_1)$ is nonzero by Theorem \ref{rthm}. 
Thus by Proposition \ref{p1}, we have $L(\frac{1}{2},BC(\pi_1^{\vee})\times \mu \cdot \sigma_i^{\vee})\ne 0$ 
and so $L(\frac{1}{2},BC(\pi_1)\times \mu^{-1}\sigma_i)\ne 0$ by the functional equation. 
Thus 
\[
L\left(\frac{1}{2},BC(\pi_1) \times BC(\pi_2)\otimes \mu^{-1}\right) 
= \prod_{i=1}^tL\left(\frac{1}{2},BC(\pi_1 )\times \mu^{-1}\sigma_i\right)\ne 0 
\] 
and so our claim is proved.
\end{proof}


\begin{thebibliography}{alpha}

\bibitem{A1} 
J. Arthur, 
\emph{A trace formula for reductive groups, \textbf{I}: Terms associated to classes in $G(\mathbb{Q})$}, 
Duke Math. J., \textbf{45} No. 1. (1978), 911--952.

\bibitem{A2} 
J. Arthur, 
\emph{A trace formula for reductive groups. \textbf{II}: Applications of a truncation operator},  
Compositio Math. \textbf{40} (1980), 87--121.

\bibitem{Af}
J. Arthur, 
\emph{On the inner product of truncated Eisenstein series}, Duke Math. J. \textbf{49} (1982), 35-70.

\bibitem{Ae}
J. Arthur, 
\emph{The endoscopic classification of representations: orthogonal and symplectic groups},  
Colloquium Publications \textbf{61}, American Mathemaical Society, (2013)

\bibitem{A3}  
J. Arthur, 
\emph{The trace formula in invariant form}, 
Ann. of. Math. (2) \textbf{114} (1981), 1--74.

\bibitem{BZ0} 
J. Bernstein and A. Zelevinsky, 
\emph{Representations of the group $GL(n,F)$, where $F$ is a non-archimedean local field},  
Russian Math. Surveys, \textbf{31} (1976), 1-68.

\bibitem{BPCZ}
R. Beauzart-Plessis, Pierre-Henri Chaudouard, M. Zydor,
\emph{The global Gan-Gross-Prasad conjecture for unitary groups: the endoscopic case},
Publ.math.IHES, \textbf{135}, (2022), 183–336. 

\bibitem{BZ} 
J. Bernstein and A. Zelevinsky, 
\emph{Induced representations of reductive p-adic groups I},  
Ann. Sci. Ec. Norm. Super. (4) \textbf{10} (1977), 441--472.

\bibitem{BLZZ} 
R. Beauzart-Plessis, Yifeng Liu, Wei Zhang, Xinwen Zhu, \emph{Isolation of cuspidal spectrum, with application to the Gan-Gross-Prasad conjecture}, Ann. of Math. \textbf{194} (2021), 519--584.

\bibitem{Ca} 
W. Casselman, 
\emph{Canonical extensions of Harish-Chandra modules to representations of $G$}, 
Canad. J. Math. \textbf{41} (1989), 385--438.


\bibitem{Cps} 
J. Cogdell, I. Piatetski-Shapiro and F. Shahidi, 
\emph{Functoriality for the classical groups}, 
Publ. Math. Inst. Hautes Etudes Sci. \textbf{99} (2004), 163--233.

\bibitem{MF1}
M. Furusawa, K. Morimoto, 
\emph{On special Bessel periods and the Gross-Prasad conjecture for SO(2n+1) x SO(2) }, Math. Ann., \textbf{368(1-2)}, (2017), 561-586.

\bibitem{MF2}
M. Furusawa, K. Morimoto, 
\emph{Refined global Gross-Prasad conjecture on special Bessel periods and Boecherer's conjecture}, J. Eur. Math. Soc., \textbf{23(4)}, (2021), 1295 - 1331.


\bibitem{GJBR} D. Ginzburg, D. Jiang, B. Liu, S.Rallis, \emph{Erratum to ``On the Non-vanishing of the Central Value of the Rankin-Selberg L-functions''}, preprint 2019, \href{https://arxiv.org/pdf/1905.02644.pdf}{https://arxiv.org/pdf/1905.02644.pdf}

\bibitem{GJR}
D. Ginzburg, D. Jiang, S.Rallis, \emph{On the nonvanishing of the central critical value of the Rankin-Selberg $L$-functions}, J. Amer. Math. Soc., \textbf{17}, (2004), no. 3, 679-722.

\bibitem{GP}
B. Gross and D. Prasad, \emph{On the decomposition of a representation of $SO(n)$ when restricted to $SO(n-1)$}, Cand. J. Math. \textbf{44}, (1992), 974-1002.

\bibitem{Gan1}
W. T. Gan, B. Gross and D. Prasad, \emph{Branching laws for classical groups: The non-tempered case}, 
Compositio Mathematica , \textbf{156} , Issue 11 , (2020), 2298-2367.

\bibitem{Gan2} 
W. T. Gan, B. Gross and D. Prasad, 
\emph{Symplectic local root numbers, central critical $L$-values, and restriction problems in the representation theory of classical groups}, 
Ast{\'e}risque \textbf{346} (2012), 1--110.

\bibitem{Gan3} 
W. T. Gan, B. Gross and D. Prasad, 
\emph{Restrictions of representations of classical groups: examples}, 
Ast{\'e}risque \textbf{346} (2012), 111--170.


\bibitem{gi}W. T. Gan and A. Ichino , \emph{The Gross-Prasad conjecture and local theta correspondence}, Invent. Math. \textbf{206} (2016), 705--799. 


\bibitem{gi1}W. T. Gan and A. Ichino , \emph{The Shimura--Waldspurger correspondence for $Mp_{2n}$}, Annals of Math. \textbf{188} (2018), 965--1016. 

\bibitem{GR}
S.S. Gelbart, J.D. Rogawski, \emph{$L$-functions and Fourier-Jacobi coefficients for the unitary group $U(3)$}, Invent. Math. \textbf{105}(3), (1991), 445-472.

\bibitem{gjr} 
D. Ginzburg, S. Rallis and D. Soudry, 
\emph{The descent map from automorphic representations of $GL(n)$ to classical groups}, 
World Scientific, Hackensack (2011).

\bibitem{GPSR}
D. Ginzburg, Piatetski-Shapiro, S. Rallis, 
\emph{$L$ functions for the orthogonal group}, Mem. Amer. Math. Soc.  128  (1997),  no. 611.

\bibitem{H1}
J. Haan,
\emph{The Bessel periods : the orthogonal and unitary groups cases}, in preparation.

\bibitem{H2}
J. Haan,
\emph{The Bessel periods of Eisenstein series}, in preparation.

\bibitem{H3}
J. Haan,
\emph{The Fourier--Jacobi periods : the case of $Mp(2n+2r) \times Sp(2n)$}, \href{https://arxiv.org/abs/2201.03270}{https://arxiv.org/abs/2201.03270}.

\bibitem{H4}
J. Haan,
\emph{The regularized trilinear periods for general linear groups}, in preparation.



\bibitem{Ich} 
A. Ichino and T. Ikeda, 
\emph{On the periods of automorphic forms on special orthogonal groups and the Gross-Prasad conjecture}, 
Geometric Functional Analysis. 19 (\textbf{5}) (2010), 1378--1425.

\bibitem{IY1} 
A. Ichino and S. Yamana,
\emph{Period of automorphic forms: The case of $(GL_{n+1}\times GL_n,GL_n)$}, 
Compositio. Math. (\textbf{151}) (2015), 665--712.

\bibitem{IY} 
A. Ichino and S. Yamana, 
\emph{Period of automorphic forms: The case of $(U_{n+1}\times U_n,U_n)$}, 
J. Reine Angew. Math. 19 (\textbf{5}) (2019), 1--38.

 \bibitem{JLR} 
H. Jacquet, E. Lapid, J. Rogawski, 
\emph{Periods of automorphic forms}, 
J. Amer. Math. Soc., \textbf{12} (1999), 173--240.

\bibitem{JR} H. Jacquet and S. Rallis, \emph{On the Gross-Prasad conjecture for unitary groups. in On certain $L$-functions}, Clay Math. Proc. \textbf{13},
 American Mathematical Society, Providence (2011), 205-264.
 
 
 \bibitem{JSZ} D. Jiang, D. Soudry, L. Zhang, \emph{The unramified computation of Rankin-Selberg integrals for quasi-split classical groups: Bessel model case}, in preparation.

\bibitem{JZ} D. Jiang, L. Zhang, \emph{Arthur parameters and cuspidal automorphic modules of classical groups}, Ann. of Math. (2) \textbf{191} (2020), 739-827.


\bibitem{JZ1} D. Jiang, L. Zhang, \emph{Automorphic Integral Transforms for Classical group II: Twisted Descents, in \emph{Representation Theory, Number Theory, and Invariant Theory}}, J. Cogdell et al. (eds.), Progress in Mathematics \textbf{323} (2017), 303-335.

 \bibitem{JZ2} D. Jiang, L. Zhang, \emph{Arthur parameters and cuspidal automorphic modules of classical groups II: Fourier-Jacobi case}, (2015), in preparation

\bibitem{JZ3} D. Jiang, L. Zhang, 
\emph{\it A product of tensor product L-functions for classical groups of hermitian type},
Geom. Funct. Anal. 24 (2014), no. 2, 552-609.

\bibitem{KMSW}
T. Kaletha, A. Minguez, S. W. Shin and P.-J. White, \emph{Endoscopic classification of representations: Inner forms of unitary groups}, preprint 2014, \href{https://arxiv.org/pdf/1409.3731.pdf}{https://arxiv.org/pdf/1409.3731.pdf}.

\bibitem{KM1}
H. Kim, M. Krishnamurthy, \emph{Base change lift for odd unitary groups}, in: Functional analysis VIII, Various Publ. Ser. (Aarhus) \textbf{47}, Aarhus University, Aarhus (2004), 116--125.

\bibitem{KM2}
H. Kim, M. Krishnamurthy, \emph{Stable base change lift from unitary groups to $GL_N$}, IMRN \textbf{1} (2005), 1--52.

\bibitem{Ku}
S. S. Kudla, \emph{Splitting metaplectic covers of dual reductive pairs}, Israel J. Math. \textbf{87} (1994), 361--401.

\bibitem{Mo} 
C. M{\oe}glin and J. L. Waldspurger, 
\emph{Spectral decomposition and Eisenstein series}, 
Cambridge Tracts in Math.  \textbf{113}, Cambridge Universitym Press, Cambridge (1995).

\bibitem{Mok}
C. P. Mok, \emph{Endoscopic classification of represenations of quasi-split unitary groups}, Mem. Amer. Math. Soc. \textbf{235} (2015), no. 1108.

\bibitem{Mu}
G. Muic, \emph{A proof of Casselman--Shahidi's conjecture for quasi-split classical groups}, Canad. Math. Bull. Vol \textbf{44} (3) (2001), 298--312.

\bibitem{LR} 
E. Lapid and J. Rogawski, 
\emph{Periods of Eisenstein series$\co$ the Galois side}, 
Duke Math. J. \textbf{120}(1) (2003), 153--226.

\bibitem{L}
Y. Liu, \emph{Refined global Gan-Gross-Prasad conjecture for Bessel periods}, 
J. Reine Angew. Math. (\textbf{717}) (2016), 133-194.

\bibitem{LS} 
Y. Liu and B. Sun, 
\emph{Uniqueness of Fourier-Jacobi models: the Archimedean case}, 
J. Funct. Anal. \textbf{265} (2013), 3325--3340.




\bibitem{S-ICM}
D. Soudry, \emph{Rankin-Selberg integrals, the descent method, and Langlands functoriality},
International Congress of Mathematicians. Vol. II, Eur. Math. Soc., Z\"urich, (2006), 1311-1325.

\bibitem{S-I}
D. Soudry, \emph{The unramified computation of Rankin-Selberg integrals expressed in terms of Bessel models for split orthogonal groups: Part I},
Israel J. Math., \textbf{222}, (2017), no. 2, 711--786.

\bibitem{S-II}
D. Soudry, 
\emph{The unramified computation of Rankin-Selberg integrals expressed in terms of Bessel models for split orthogonal groups: Part II}, J.  Number Theory., \textbf{186}, (2018), 62-102

\bibitem{Sha0}
F. Shahidi, 
\emph{Eisenstein series and Automorphic $L$-functions}, 
Colloquium Publications, Vol 58, AMS, (2010)

\bibitem{Sha} 
F. Shahidi, 
\emph{On certain $L$-functions}, 
Amer. J. Math., \textbf{103} (2), (1981), 297--355.

\bibitem{Sch}
R. Berndt, Ralf Schmidt, \emph{Elements of the Representation Theory of the Jacobi group},
Progress in Mathematics, Vol. 163, Birkhauser Verlag, Basel, (1998)

\bibitem{BS} 
B. Sun, 
\emph{Multiplicity one theorems for Fourier-Jacobi models}, 
Amer. J. Math., \textbf{134} (6) (2012), 1655--1678.

\bibitem{Wa} 
N. Wallach, 
\emph{Real Reductive Groups vol. II}, 
Pure and Applied Mathematics, Volume \textbf{132} (Academic Press, Inc., Boston, MA, (1992). xiv+454 pp.

\bibitem{We} A. Weil, \emph{Sur certain groupes d' opeateurs unitaries}, Acta Math, \textbf{111} (1964), 143--211.

\bibitem{Xue2} H. Xue, \emph{Fourier-Jacobi periods and the central critical value of Rankin-Selberg $L$-functions}, Israel J. Math., \textbf{212}, (2016), no. 2, 547-633.


\bibitem{Xue} H. Xue, \emph{Refined global Gan-Gross-Prasad conjecture for Fourier-Jacobi periods on symplectic groups}, Compositio. Math. (\textbf{153}) (2017), no. 1, 68--131.

\bibitem{Xue1} H. Xue, \emph{The Gan-Gross-Prasad conjecture for $U(n) \times U(n)$}, Adv. Math., \textbf{262} (2014), 1130-1191.



\bibitem{Y} 
S.Yamana, 
\emph{Period of automorphic forms: The trilinear case}, 
J. Inst. Math. Jussieu., \textbf{17}(1), (2018), 59--74.

\bibitem{Y2} 
S.Yamana, 
\emph{Periods of residual automorphic forms}, 
J. Funct. Anal. \textbf{268} (2015), 1078--1104.

\bibitem{Z} 
A. Zelevinsky, 
\emph{Induced representations of reductive $p$-adic groups. II: On irreducible representations of $GL(n)$}, 
Ann. Sci. Ec. Norm. Super. (4) \textbf{13} (1980), 165-210. 

\bibitem{W1} W. Zhang, \emph{Fourier transform and the global Gan-Gross-Prasad conjecture for unitary groups}, Ann. of Math., (2) \textbf{180} (2014), 971-1049.

\bibitem{W2} W. Zhang, \emph{Automorphic period and the central value of Rankin-Selberg $L$-function}, J. Amer. Math. Soc., \textbf{27} (2014), 541-612.

\end{thebibliography}
\end{document}